\documentclass[a4paper,11pt]{article}
\usepackage[nointlimits]{amsmath}
\usepackage{amsfonts}
\usepackage{amssymb}
\usepackage{amsfonts}
\usepackage{amssymb}
\usepackage{amsthm}
\usepackage{amsmath}
\usepackage{latexsym}
\usepackage{graphicx}
\usepackage{layout}
\usepackage[english]{babel}
\usepackage{color}
\usepackage{subfigure}

\numberwithin{equation}{section}

\usepackage{graphics}
\usepackage{graphicx}
\usepackage{epsfig}
\usepackage{psfrag}
 
\newtheorem{Definition}{Definition}[section]
\newtheorem{Theorem}{Theorem}[section]
\newtheorem{Lemma}{Lemma}[section]
\newtheorem{Proposition}{Proposition}[section]
\newtheorem{Example}{Example}[section]
\newtheorem{Remark}{Remark}[section]

\newtheorem{Corollary}{Corollary}[section]

\newcommand{\nada}[1]   {}

\newcommand{\A}         {\mathcal A}
\newcommand{\ab}         {[a,b]}
\newcommand{\area}         {\areafun}
\newcommand{\areafun}         {\mathcal A}

\newcommand{\BV}        {\mathrm{BV}}
\newcommand{\BVo}       {\BV(\Om;\R^2)}
\newcommand{\C}         {\mathcal C}
\newcommand{\Ce}         {\dommap_\eps}

\newcommand{\Cuno}	{{\mathcal C}^1(\Om;\R^2)}

\newcommand{\dommap}        {D}

\newcommand{\dKe}{\partial \Ke}
\newcommand{\disk}      {B}
\newcommand{\disco} {\disk}
\newcommand{\discoest}  {\disk^{{\rm ext}}}

\newcommand{\DM} {{\rm D}(\Om;\R^2)}
\newcommand{\DMjump} {{\rm D}(\Om \setminus \osaltou ;\R^2)}
\newcommand{\dx}      {\partial_x}
\newcommand{\dy}      {\partial_y}

\newcommand{\e}       {\mu}
\newcommand{\extdisk}       {\discoest}
\newcommand{\extimmmapmin}       {\Sigma^{{\rm ext}}}
\newcommand{\extmapmin}  {Y^{{\rm ext}}}
\newcommand{\extdommap}  {\dommap^{{\rm ext}}}
\newcommand{\eps}       {\varepsilon}

\newcommand{\gammam}         {\gamma^-}
\newcommand{\gammamu}         {\gamma^-_1}
\newcommand{\gammamd}         {\gamma^-_2}
\newcommand{\gammap}         {\gamma^+}
\newcommand{\gammapu}         {\gamma^+_1}
\newcommand{\gammapd}         {\gamma^+_2}

\newcommand{\gammapm}         {\gamma^\pm}

\newcommand{\Gammau}         {\Gamma}
\newcommand{\gauss}         {\mathcal{N}}
\newcommand{\gm}         {\gammam}
\newcommand{\gp}         {\gammap}

\newcommand{\grad}      {\nabla}
\renewcommand{\H}       {\mathcal{H}}
\newcommand{\height}  {h}

\newcommand{\hp}        {$\rm \tilde{u}$}
\newcommand{\immmapmin}	{\Sigma_{\min}}

\newcommand{\K}{K}
\newcommand{\Ke}{K_\e}

\newcommand{\Luno}	{L^1(\Om;\R^2)}

\newcommand{\M}         {\mathcal M}
\newcommand{\mapmin}    {\minpar}

\newcommand{\map}      {\mappa}
\newcommand{\mappa}      {\mathbf{u}}
\newcommand{\mappav}      {\mathbf{v}}
\newcommand{\mappavuno}      {v_1}
\newcommand{\mappavdue}      {v_2}
\newcommand{\mappai}      {{\rm u}_i}
\newcommand{\mappadue}      {{\rm u}_2}
\newcommand{\mappauno}      {{\rm u}_1}
\newcommand{\minpar}    {X}

\newcommand{\Om}        {\Omega}

\newcommand{\openMorse}      {{\rm B}}
\newcommand{\osalto}        {\osaltou}
\newcommand{\osaltou}        {\overline J_\mappa}

\newcommand{
\parahilduno
}        {u}
\newcommand{\parabordo}   {{\mathit g}}
\newcommand{\parabordopalla}   {\mathbf{b}}
\newcommand{\parahilddue}        {v}
\newcommand{\paramap}        {Y}
\newcommand{\phd} {\parahilddue}

\newcommand{\phu} {\parahilduno}
\newcommand{\pie} {\pi_\e}
\newcommand{\ppa} {\primoparametroastratto}
\newcommand{\primacoordsource}        {x}
\newcommand{\primacoordtarget}        {\xi}
\newcommand{\primoparametroastratto}        {t}
\newcommand{\ps} {S}
\newcommand{\pn} {N}
\newcommand{\R}         {\ensuremath{\mathbb R}}

\newcommand{\Reps}      {R_\eps}
\newcommand{\rel}      {\overline \A}
\newcommand{\rettangolo} {R}

\def\rest{\hskip 1pt{\hbox to 10.8pt{\hfill
\vrule height 7pt width 0.4pt depth 0pt\hbox{\vrule height 0.4pt
width 7.6pt depth 0pt}\hfill}}}

\newcommand{\saltou}        {J_\mappa}

\newcommand{\secondacoordsource}        {y}
\newcommand{\secondacoordtarget}        {\eta}
\newcommand{
\secondoparametroastratto
}        {s}
\newcommand{\Sigmamu}{\widehat{\Sigma}_\mu}

\newcommand{\source}  {\R^2_{(\primacoordsource,\secondacoordsource)}}
\newcommand{\spa} {\secondoparametroastratto}
\newcommand{\target}  {\R^2_{(\primacoordtarget,\secondacoordtarget)}}
\newcommand{\tandisco}  {\tau_{\partial  B}}
\newcommand{\Te}  {T_\eps}
\newcommand{\Teu}  {T_{\eps 1}}
\newcommand{\Ted}  {T_{\eps 2}}
\newcommand{\triple}  {\mappa_{{\rm tr}}}
\newcommand{\ue}  {\mappa_\eps}
\newcommand{\ued}  {{\rm u}_{\eps 2}}
\newcommand{\uei}  {{\rm u}_{\eps i}}
\newcommand{\ueu}  {{\rm u}_{\eps 1}}

\newcommand{\Wuu}  {W^{1,1}(\Om; \R^2)}

\newcommand{\Xed}  {X_{\e 2}}
\newcommand{\Xet}  {X_{\e 3}}

\makeatletter
\def\@makefnmark{\hbox{\@textsuperscript{\normalfont(\@thefnmark)}}}
\makeatother

\begin{document}

\title{On the area of the graph of a piecewise smooth map from the 
plane to the plane with a curve discontinuity
}

\author{
Giovanni Bellettini\footnote{
Dipartimento di Matematica,
Universit\`a di Roma Tor Vergata,
via della Ricerca Scientifica 1, 00133 Roma, Italy,
and
INFN Laboratori Nazionali di Frascati, Frascati, Italy.
E-mail: belletti@mat.uniroma2.it
                      }
\and
Maurizio Paolini\footnote{
Dipartimento di Matematica,
              Universit\`a Cattolica ``Sacro Cuore'',
via Trieste 17, 25121 Brescia, Italy
E-mail: paolini@dmf.unicatt.it
                         }
\and
Lucia Tealdi\footnote{International School for Advanced Studies, S.I.S.S.A., via Bonomea 265, 34136 Trieste, Italy. E-mail: ltealdi@sissa.it}
}

\date{}

\maketitle
\thanks{}

\begin{abstract}
In this paper we provide an estimate from above for the value of
the relaxed area functional $\rel(\mappa,\Omega)$ for an $\R^2$-valued map $\mappa$ defined 
on a bounded domain $\Om$ of the plane and 
discontinuous on a $\C^2$ simple curve $\osaltou \subset \Om$, with two endpoints.
We show that, under certain assumptions on $\mappa$,  $\rel(\mappa,\Omega)$ does not exceed the area
of the regular part of $\mappa$, with the addition of a singular
term measuring the area of a \textit{disk-type} solution $\immmapmin$ of the
Plateau's problem spanning the two traces of $\mappa$ on $\osaltou$. The result is valid
also when $\immmapmin$ has self-intersections.
A key element in our argument is to show the existence of
  what we call a
\textit{semicartesian parametrization} of $\immmapmin$, namely 
a conformal parametrization of $\immmapmin$ defined on a suitable parameter space,
which is the identity in the first component. 
To prove our result, various tools of  parametric minimal
surface theory are used, as well as some results from Morse theory.
\end{abstract}
%\keywords{}
%\subjclass{Primary ; Secondary}
%
\section{Introduction}\label{sec:intro}
Given a bounded open 
set
$\Omega \subset \R^2=\R^2_{(x,y)}$  and
a map $\mappav = (v_1,v_2): \Omega \to \R^2 = \R^2_{(\xi,\eta)}$
of class  $\C^1$, 
the area $\area(\mappav,\Om)$ of the 
graph of $\mappav$ in $\Omega$ is given by 
$$
\A(\mappav, \Om)=
\displaystyle \int_{\Om}|
\M(\grad \mappav)|
\,dx\,dy,
$$
where $|\cdot|$ denotes the euclidean norm, $\grad \mappav $ is the Jacobian
matrix of $\mappav$ and $\M(\grad \mappav)$ is the vector whose 
components are the determinants of all minors\footnote{
Including 
the determinant of
order zero, which by definition is equal to one.}  of $\grad \mappav$, hence
\begin{equation*}\label{eq:M}
\vert \M(\grad \mappav)\vert
=\sqrt{1 + |\grad{\mappavuno}|^2 + |\grad{\mappavdue}|^2+ 
\left(\partial_x \mappavuno \partial_y \mappavdue - \partial_y \mappavuno \partial_x \mappavdue \right)^2}.
\end{equation*}
The polyconvex \cite{DC:89}
 functional 
 $\area(\mappav,\Omega)$ 
has linear growth and measures the 
area of the graph of $\mappav$, a 
smooth two-codimensional surface in $\R^4 = \R^2_{(x,y)}\times
\R^2_{(\xi,\eta)}$.
When considering
the perspective of the direct method of the Calculus of Variations, 
it is important
to assign a reasonable 
notion of area also to the graph
of a {\it nonsmooth} map, namely to extend 
the functional $\area(\cdot,\Omega)$ out of $\Cuno$ in a natural way.
We agree in defining this extended area as the 
$L^1(\Omega; \R^2)$-lower semicontinuous envelope
$\overline \area(\cdot, \Omega)$ (or relaxed functional for short)
  of $\A(\cdot, \Om)$, i.e.,
\begin{equation}\label{eq:arearilassata}
\rel(\mappav, \Om):= 
\inf \left\{  \liminf_{\eps \to 0^+} \A(\mappav_\eps,\Om) \right\}
\end{equation}
where the infimum is taken over all sequences\footnote{In this paper we consider families of 
functions (or functionals, or points) indicized by a continuous
parameter;  
with a small abuse of language, these families are still called
sequences.}
 $(\mappav_\eps)
\subset \Cuno$
 converging to $\mappav$ in
$\Luno$.
The interest of definition \eqref{eq:arearilassata}
is clearly seen in the scalar case\footnote{Namely,
for functions $v : \Om \to \R$.},
 where this notion of 
extended  area 
is useful for solving non-parametric minimal 
surface problems, under various type of boundary conditions
(see for instance \cite{Gi:84}, \cite{MaMi:84}, 
\cite{GiMoSo:98}).
We recall that in the scalar case $\A(\cdot,\Omega)$ happens to be  convex, and 
$\overline \area(\cdot,\Omega)$ is completely characterized:
its domain is the space ${\rm BV}(\Omega)$ of  functions with bounded
variation in $\Omega$, and its expression is suitably
given in integral form. 

The analysis of the properties of $\overline \area(\mappav,\Omega)$ 
for maps $\mappav$ from a subset of the plane to the plane is much
more difficult \cite{GiMoSo:98}; geometrically, 
the problem is to understand which could be 
the most ``economic'' way, in terms
of two-dimensional area in $\R^4$, of approximating 
a {\it nonsmooth two-codimensional graph} of a map 
$\mappav$ of bounded variation, with graphs
of smooth maps, where
the approximation takes place 
in  $L^1(\Om;\R^2)$.
It is the aim of the present paper to address this problem 
 for discontinuous maps $\bf v$ of class $\BVo$,
having a $\mathcal C^2$-curve of discontinuity and satisfying suitable properties.

In \cite{AcDa:94} Acerbi and Dal Maso studied
the relaxation of polyconvex functionals with linear growth
in arbitrary dimension and codimension. 
In particular, they proved 
that
 $\rel(\cdot,\Omega)=\A(\cdot,\Omega)$ on $\Cuno$, and that 
 for $p\in [2,+\infty]$, 
\begin{equation*}\label{eq:estensione} 
\rel
(\mappav, \Omega) = \displaystyle \int_\Om |\M (\grad \mappav) |\,dx\,dy,
\qquad \mappav \in W^{1,p}(\Om; \R^2),
\end{equation*} 
and the exponent $p$ is 
optimal.
Concerning the representation
of $\overline \area(\cdot,\Omega)$ in $\BVo$, they proved 
\cite[Theorem 2.7]{AcDa:94}
that
the domain of $\overline \area(\cdot,\Omega)$
is contained in $\BVo$, and
\begin{equation}\label{eq:altra_diseq}
\rel(\mappav,\Om)\geq \int_{\Om}|\M(\grad \mappav)|\,dx\,dy + |D^s \mappav |(\Om),
\qquad \mappav \in \BVo,
\end{equation}
where $\grad \mappav$ and $D^s \mappav$ denote
the absolutely continuous
and the singular part of the distributional gradient $D \mappav$
of $\mappav$, respectively. 
In addition, if $\mappav \in {\rm BV}(\Omega; \{\alpha_1,\dots,\alpha_m\})$ 
where $\alpha_1,..,\alpha_m$ 
are vectors 
of $\R^2$,
and denoting
by
$\mathcal{L}^2$ and $\H^1$ the  
Lebesgue measure and the one-dimensional Hausdorff measure 
in $\R^2$ respectively, 
\begin{equation}\label{eq:duevalori}
\rel(\mappav,\Om)= \mathcal{L}^2(\Om) + 
\sum_{ \tiny{\begin{matrix}k,l \in  \{1,..,m\} \\ k<l \end{matrix}}}
|\alpha_k - \alpha_l|\,\H^1(J_{kl}),
\end{equation}
provided
 $\partial \Om$ and the jump curves  $J_{kl}$ forming the jump
set $J_{\bf v}$ of ${\bf v}$ are smooth enough 
and that $\mappav$ takes locally
only two vectors around $J_{kl}$, see 
\cite[Theorem 2.14]{AcDa:94} for the details. Finally, 
and maybe more interestingly, 
it is proven in \cite[Section 3]{AcDa:94}
that 
the relaxed area
is not subadditive with respect to $\Om$, 
thus in particular it does not admit an integral representation,
hence it is \textit{non-local}.
The non-subadditivity of $\rel(\mappav,\cdot)$, 
conjectured by De Giorgi in \cite{DG:92}, 
concerns 
the triple junction map $\triple$, which is a map
defined on the unit disk of the source 
plane, and 
 assumes as values
three non-collinear vectors  on three circular congruent sectors.  
The proof given in \cite{AcDa:94} does not 
supply the precise value of $\rel(\triple, \Omega)$, however it
provides
a nontrivial lower bound and an upper bound. 
The upper bound was refined in 
\cite{BelPao:10}, where the authors exhibited an approximating
sequence (conjectured to be optimal\footnote{In the 
sense that equality should hold in \eqref{eq:arearilassata} along the 
above mentioned sequence.}, at least
under symmetry assumptions)
 constructed by solving 
three (similar) Plateau-type problems coupled at the 
triple point\footnote{The construction
of \cite{BelPao:10} is intrinsically four-dimensional and 
cannot be reduced to a three-dimensional
construction.}.  The singular contribution concentrated
 over to the 
triple point
arising in 
this construction,
consists of a term penalizing the length of the Steiner-graph
connecting the three values in the target space $\R^2$,
with weight two. If the construction of \cite{BelPao:10} were 
optimal, it would shed some light on the nonlocality
phenomenon addressed in \cite{DG:92} and \cite{AcDa:94}. 

The question arises as to whether the nonlocality is due to the
special form of the triple junction map $\triple$, 
or whether it can be obtained
for other qualitatively different maps  $\mappav$.
We are not still able to answer this question,
which nevertheless can be considered as the 
main motivation of the present paper. In this direction, 
our idea is to study the properties
of $\rel(\cdot,\Omega)$, for maps  
generalizing those in \eqref{eq:duevalori}, with
no triple or multiple  junctions. 
Namely, we are interested in $\rel(\mappa,\Omega)$,
where $\mappa$ is 
regular enough in $\Om \setminus \osaltou$, and the jump 
set $\saltou$ is a $\C^2$ simple curve compactly contained\footnote{
As one can deduce from our proofs,
the case when $\osaltou\cap \partial \Omega \neq \emptyset$
 requires a separate study, leading to 
a Plateau-type problem with partial free boundary, and will be
investigated elsewhere. Also, the case 
when $\osaltou \subset \Omega$ is a closed simple curve 
 is out of the scope of the present paper,
since it leads to the study of minimal immersions of 
an annulus in $\R^3$.}
 in $\Om$. 
It is worth anticipating that we are concerned here
only with an estimate from above of the value of the relaxed area,
and we shall not face the problem of the estimate from below. 
Nevertheless, we believe our construction of the recovery sequence
to be optimal, at least for a reasonably large class of maps.

Referring to the next sections for the details,
we now briefly sketch the main
results and the ideas of the 
present paper.
Suppose that 
 $\mappa \in \BVo$ is a vector valued map regular enough in $\Omega \setminus \osaltou$, 
and let us parametrize 
$\osaltou$ with a map $\alpha : t\in \ab \to \alpha(t) \in \osaltou$. Denote
by  
$\mappa^\pm$ the two traces of 
$\mappa$ on $\osaltou$, and
let $\gamma^\pm$, defined in $\ab$, be the composition of $\mappa^\pm$
with the parametrization $\alpha$.
Let us define $\Gamma$ as the union of the graphs of 
$\gamma^+$ and $\gamma^-$. Our regularity 
assumptions ensure that 
$\Gamma$ is a rectifiable, simple and closed space curve, with a special 
structure, due to the fact that it is union of graphs of 
two vector maps defined in the same interval $\ab$
(Definition \ref{def:semicart_curve}). Finally, 
let us denote by $\Sigma_{\min}$ an area minimizer 
solution of 
the
 Plateau's problem for $\Gamma$, in the class of surfaces spanning
$\Gamma$ and having 
the {\it topology of the disk} \cite{DiHiSa:10}. 
Suppose that $\Sigma_{\min}$ admits what we call a {\it semicartesian
parametrization} (Definition \ref{def:semicart_par}), namely 
 a global parametrization whose first component coincides with 
the parameter $t\in \ab$.
Our first result reads as follows.
\begin{Theorem}\label{teo:primo}
Under the above
assumptions,
there exists a sequence $(\ue)$ of sufficiently regular\footnote{
$(\ue) \subset {\rm Lip}(\Om;\R^2)$ in Theorem \ref{teo:graph_main}, 
and $(\ue) \subset W^{1,2}(\Om;\R^2)$ in Theorem \ref{teo:general_main}.
}
maps
converging to $\mappa$ in $\Luno$  such that 
\begin{equation}\label{eq:rilaintro}
\lim_{\eps \to 0^+} \rel(\ue,\Omega) 
%= 
%\rel(\mappa, \Om\setminus\osaltou)+\H^2(\immmapmin)
= \int_{\Om\setminus \osaltou} |\M(\grad \mappa)|\,dx\,dy + \H^2(\immmapmin).
\end{equation}
In particular
\begin{displaymath}
\rel(\mappa, \Om)\leq 
 \int_{\Om\setminus \osaltou} |\M(\grad \mappa)|\,dx\,dy + \H^2(\immmapmin).
%\rel(\mappa, \Om\setminus\osaltou)+\H^2(\immmapmin).
\end{displaymath}
\end{Theorem}

Under the hypothesis that there exists a semicartesian parametrization 
$$
X(t,s) = 
(t, X_2(t,s), X_3(t,s))
$$
of $\Sigma_{\min}$ defined on a plane 
domain $\dommap \subset \R^2_{(t,s)}$, the key point of the construction
stands in the definition of $\ue$ in a suitable neighborhood of the jump
$\saltou$. For $(x,y)$ in this neighbourhood 
 we define
the pair of functions $(t(x,y), s(x,y)) \in D$ corresponding
to the parametrization of the nearest point on $\osaltou$ to $(x,y)$, and to the 
 signed distance from $\saltou$, respectively.
Next, we define
\begin{equation}\label{eq:ueintro}
\ue(x,y):=\left( X_2\left(t(x,y), \frac{s(x,y)}{\eps}\right), X_3\left(t(x,y), \frac{s(x,y)}{\eps} \right)\right)
\end{equation}
for $(x,y)$ such that $\left(t(x,y), \frac{s(x,y)}{\eps}\right)
\in \dommap$.
Note carefully that, in this way, the definition of $\ue$ 
cannot be reduced to a one-dimensional profile, being 
intrinsically two-dimensional.
The explicit computation ({\tt step 9} of the proof of 
Theorem \ref{teo:graph_main})
 of the area of the graph of $\ue$ localized in this region
is the source of the term
$$
\H^2(\immmapmin)
$$
appearing in \eqref{eq:rilaintro}. 

It is interesting to comment on the role of the term
\begin{equation}\label{eq:deter}
\left(\partial_x {\ue}_1 \partial_y {\ue}_2 
- \partial_y {\ue}_1\partial_x {\ue}_2 \right)^2
\end{equation}
in the details of the computation. 
 If $X$ is semicartesian, the area
of $\immmapmin$ 
is given by
\begin{displaymath}
\int_{\dommap}\sqrt{|\partial_s X_2|^2 + |\partial_s X_3|^2 +
(\partial_t X_2 \partial_s X_3 - \partial_s X_2 \partial_t X_3)^2}\,dt\,ds.
\end{displaymath}
The first two addenda under the square root 
are obtained, in the limit, from 
$
\vert \grad {\ue}_1\vert^2 + 
\vert \grad {\ue}_2\vert^2$, while the 
last addendum is originated in the limit exactly by 
\eqref{eq:deter}. 

Various technical difficulties are present in the estimate of $\area(\ue,\cdot)$
outside of the above mentioned neighbourhood of $\saltou$.
Far from $\saltou$ we set $\ue:=\mappa$, while 
in a (small) intermediate 
neighbourhood  the map $\ue$ is suitably defined 
in such a way that
the corresponding contribution of the area is negligible.
The technical  point behind this construction 
is to guarantee that $\ue$ is sufficiently smooth. 
In Theorem \ref{teo:graph_main} we study the case in which 
$\immmapmin$ is the graph of a map defined on a two-dimensional 
convex domain, the so-called non-parametric case;
here an approximating argument leads to 
the Lipschitz regularity of $\ue$ in $\Omega$.
In Theorem \ref{teo:general_main}, instead, 
we study a more general situation, 
managing in building a sequence $(\ue)$ in
$W^{1,2}(\Om;\R^2)$. In this case we need to
modify the domain of the semicartesian parametrization,
in order to gain the $L^1$ integrability
of the gradients of $\ue$
and to make a further regularization near the
\textit{crack tips}, that is the end points of $\saltou$,
(see {\tt steps 1} and {\tt 2}
of  Theorem \ref{teo:general_main}). 

\smallskip
Several other comments are in order concerning Theorem \ref{teo:primo}.
First of all, and as already mentioned, 
our result provides only  an estimate from above of
the value of $\rel(\mappa, \Om)$. Only if
$\Gamma$ is contained in a plane, we are able to prove
that inequality  \eqref{eq:rilaintro}
is actually  an equality\footnote{
We believe the sequence $(\ue)$ to be a recovery sequence much 
more generally, at 
least when $\Sigma_{\min}$ 
can be identified with the support of the ``vertical component'' of a cartesian
current \cite{GiMoSo:98}
obtained by minimizing the mass among all cartesian currents
coinciding with the graph of $\mappa$ out of the jump.
In this respect,  we observe 
that the precise knowledge
of several qualitative properties of $\Sigma_{\min}$ is required
in order to prove  Theorems \ref{teo:primo} and
 \ref{teo:secondo}. 
For this reason generalizing the proof using 
an area-mininizing cartesian current seems not to be easy.}, so that $(\ue)$ becomes a recovery sequence.
This case is a slight generalization of the 
piecewise constant case \eqref{eq:duevalori} considered in \cite{AcDa:94},
and seems not enough for answering the nonlocality question
on $\rel$. 

After this remark, we come back to the important issue of
the {\it semicartesian parametrization}. First of all, a
semicartesian parametrization
 represents an intermediate situation between the non-parametric case, 
and the general case in which 
$\Sigma_{\min}$ is just an area-minimizing surface spanning $\Gamma$ 
and having the topology of the disk. We stress that
the assumptions on $\Gamma$ that ensure the existence of a
semicartesian parametrization of $\Sigma_{\rm min}$ are not 
so restrictive\footnote{
Roughly speaking, we can say (as we shall prove) that the special
structure of $\Gamma$ as union of two graphs, ``propagates''
into $\Sigma_{\min}$, ensuring the existence of a semicartesian
parametrization.}; for example the analytic
curves displayed in Figures \ref{fig:intro}(a) and (b)
satisfy the hypotheses of Theorem \ref{teo:secondo} below, and thus 
the corresponding $\immmapmin$ admit a semicartesian parametrization 
and Theorem \ref{teo:general_main} applies.
Observe that the surface $\immmapmin$ in
Figure \ref{fig:manuintro} (area-minimizing and with the topology
of the disk)
 has self-intersections\footnote{It is possible to find embedded surfaces spanning the 
same boundary with non zero-genus and lower area, 
see for example \cite[Figure 8.1.1 and Figure 8.1.2]{Mor:88}. 
Nevertheless our argument seems to be hardly generalizable 
to surfaces not of disk-type.}. 
In this case 
the map $\ue$ defined in \eqref{eq:ueintro} 
is not injective; 
of course, the source of this phenomenon is due to the higher 
codimension of ${\rm graph}(\mappa)$, and  it 
does not arise in the scalar
case.

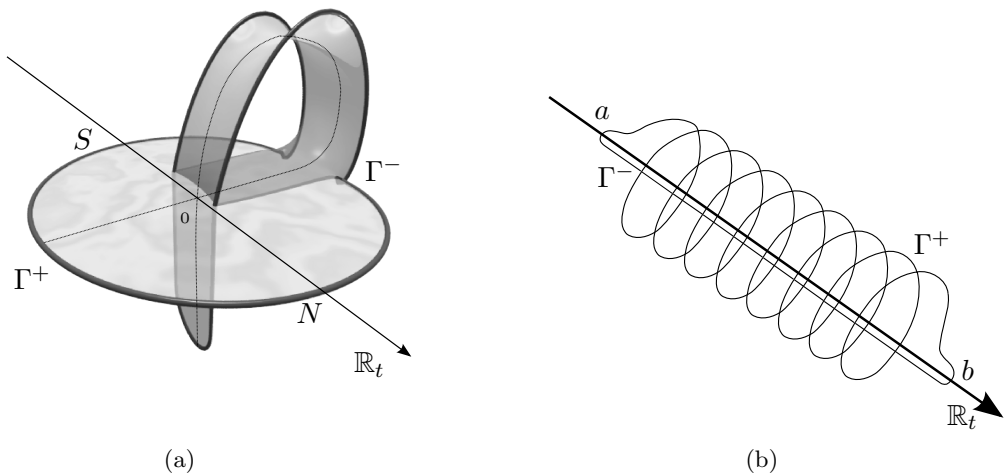
\begin{figure}[htbp] 
\centering%
\subfigure [\label{fig:manuintro}]%
{
\def\svgwidth{6.5cm}
%% Creator: Inkscape inkscape 0.48.4, www.inkscape.org
%% PDF/EPS/PS + LaTeX output extension by Johan Engelen, 2010
%% Accompanies image file 'manu_intro.pdf' (pdf, eps, ps)
%%
%% To include the image in your LaTeX document, write
%%   \input{<filename>.pdf_tex}
%%  instead of
%%   \includegraphics{<filename>.pdf}
%% To scale the image, write
%%   \def\svgwidth{<desired width>}
%%   \input{<filename>.pdf_tex}
%%  instead of
%%   \includegraphics[width=<desired width>]{<filename>.pdf}
%%
%% Images with a different path to the parent latex file can
%% be accessed with the `import' package (which may need to be
%% installed) using
%%   \usepackage{import}
%% in the preamble, and then including the image with
%%   \import{<path to file>}{<filename>.pdf_tex}
%% Alternatively, one can specify
%%   \graphicspath{{<path to file>/}}
%% 
%% For more information, please see info/svg-inkscape on CTAN:
%%   http://tug.ctan.org/tex-archive/info/svg-inkscape
%%
\begingroup%
  \makeatletter%
  \providecommand\color[2][]{%
    \errmessage{(Inkscape) Color is used for the text in Inkscape, but the package 'color.sty' is not loaded}%
    \renewcommand\color[2][]{}%
  }%
  \providecommand\transparent[1]{%
    \errmessage{(Inkscape) Transparency is used (non-zero) for the text in Inkscape, but the package 'transparent.sty' is not loaded}%
    \renewcommand\transparent[1]{}%
  }%
  \providecommand\rotatebox[2]{#2}%
  \ifx\svgwidth\undefined%
    \setlength{\unitlength}{565.98884815bp}%
    \ifx\svgscale\undefined%
      \relax%
    \else%
      \setlength{\unitlength}{\unitlength * \real{\svgscale}}%
    \fi%
  \else%
    \setlength{\unitlength}{\svgwidth}%
  \fi%
  \global\let\svgwidth\undefined%
  \global\let\svgscale\undefined%
  \makeatother%
  \begin{picture}(1,0.87709877)%
    \put(0,0){\includegraphics[width=\unitlength]{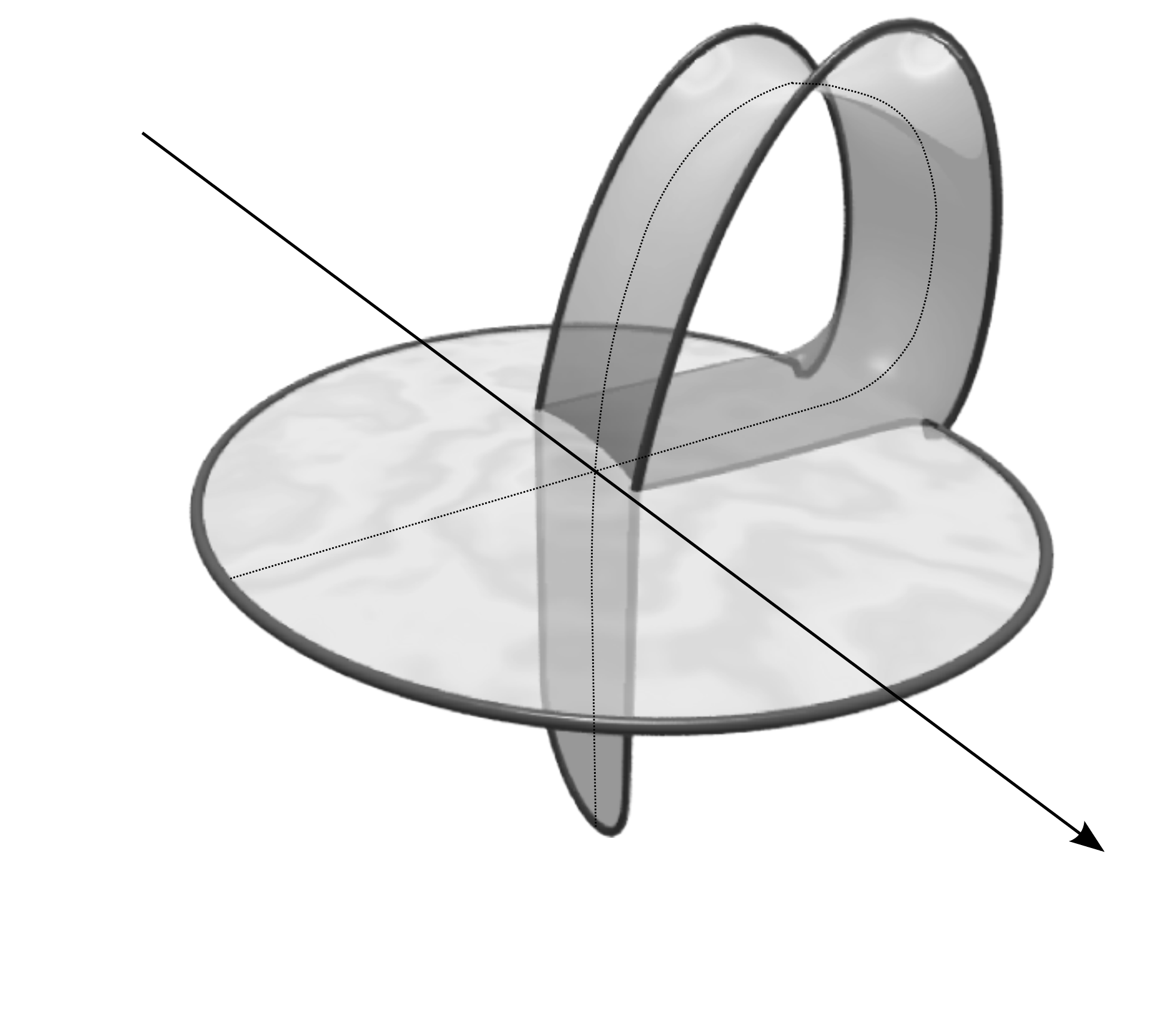}}%
    \put(0.8238426,0.12029434){\color[rgb]{0,0,0}\makebox(0,0)[lb]{\smash{$\R_t$}}}%
    \put(0.70714496,0.22035963){\color[rgb]{0,0,0}\makebox(0,0)[lb]{\smash{$N$}}}%
    \put(0.25533313,0.5756301){\color[rgb]{0,0,0}\makebox(0,0)[lb]{\smash{$S$}}}%
    \put(0.13470371,0.29023207){\color[rgb]{0,0,0}\makebox(0,0)[lb]{\smash{$\Gamma^+$}}}%
    \put(0.84621022,0.51081262){\color[rgb]{0,0,0}\makebox(0,0)[lb]{\smash{$\Gamma^-$}}}%
    \put(0.47166586,0.42588649){\color[rgb]{0,0,0}\makebox(0,0)[lb]{\smash{\tiny{$0$}}}}%
  \end{picture}%
\endgroup%

}\qquad\quad
\subfigure[ \label{fig:elica}]%
{
\def\svgwidth{6cm}
%% Creator: Inkscape inkscape 0.48.4, www.inkscape.org
%% PDF/EPS/PS + LaTeX output extension by Johan Engelen, 2010
%% Accompanies image file 'elica.pdf' (pdf, eps, ps)
%%
%% To include the image in your LaTeX document, write
%%   \input{<filename>.pdf_tex}
%%  instead of
%%   \includegraphics{<filename>.pdf}
%% To scale the image, write
%%   \def\svgwidth{<desired width>}
%%   \input{<filename>.pdf_tex}
%%  instead of
%%   \includegraphics[width=<desired width>]{<filename>.pdf}
%%
%% Images with a different path to the parent latex file can
%% be accessed with the `import' package (which may need to be
%% installed) using
%%   \usepackage{import}
%% in the preamble, and then including the image with
%%   \import{<path to file>}{<filename>.pdf_tex}
%% Alternatively, one can specify
%%   \graphicspath{{<path to file>/}}
%% 
%% For more information, please see info/svg-inkscape on CTAN:
%%   http://tug.ctan.org/tex-archive/info/svg-inkscape
%%
\begingroup%
  \makeatletter%
  \providecommand\color[2][]{%
    \errmessage{(Inkscape) Color is used for the text in Inkscape, but the package 'color.sty' is not loaded}%
    \renewcommand\color[2][]{}%
  }%
  \providecommand\transparent[1]{%
    \errmessage{(Inkscape) Transparency is used (non-zero) for the text in Inkscape, but the package 'transparent.sty' is not loaded}%
    \renewcommand\transparent[1]{}%
  }%
  \providecommand\rotatebox[2]{#2}%
  \ifx\svgwidth\undefined%
    \setlength{\unitlength}{332.08882821bp}%
    \ifx\svgscale\undefined%
      \relax%
    \else%
      \setlength{\unitlength}{\unitlength * \real{\svgscale}}%
    \fi%
  \else%
    \setlength{\unitlength}{\svgwidth}%
  \fi%
  \global\let\svgwidth\undefined%
  \global\let\svgscale\undefined%
  \makeatother%
  \begin{picture}(1,0.76449259)%
    \put(0,0){\includegraphics[width=\unitlength]{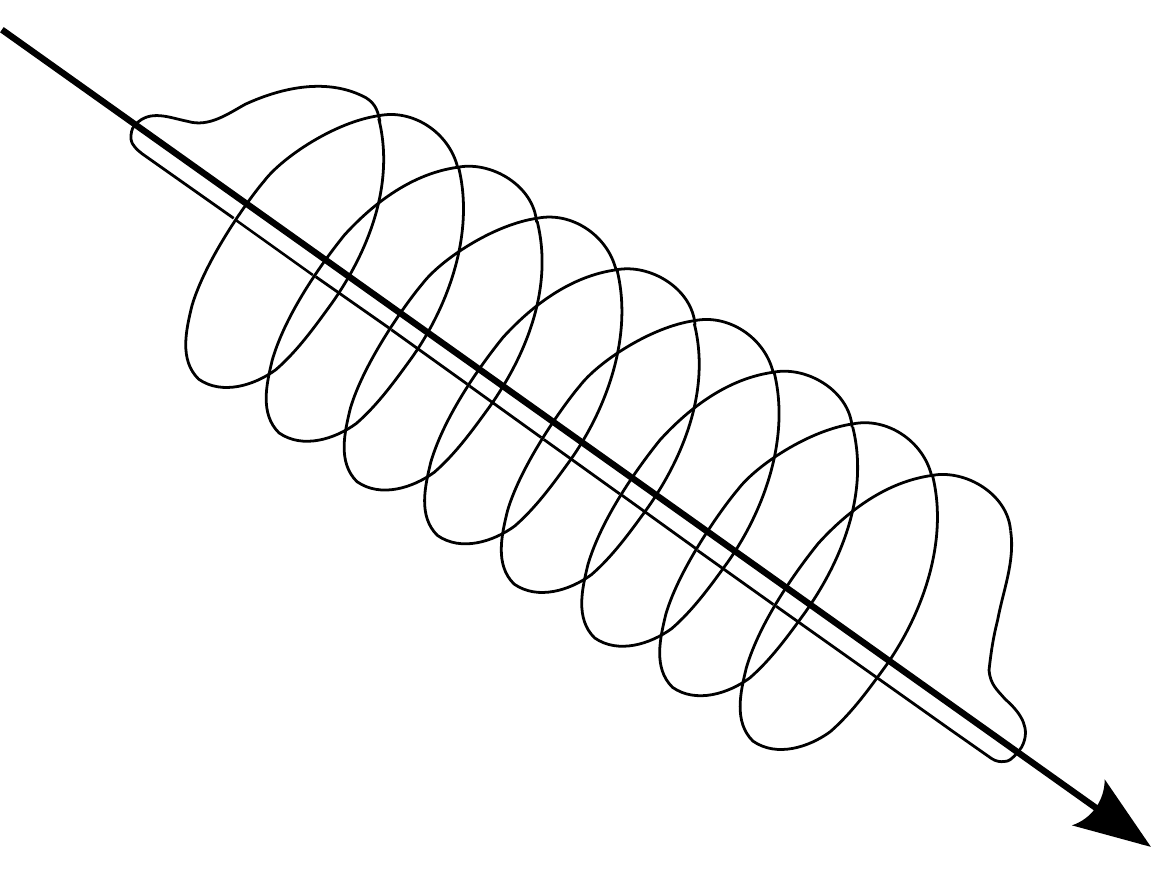}}%
    \put(0.87420255,0.0141658){\color[rgb]{0,0,0}\makebox(0,0)[lb]{\smash{$\R_t$}}}%
    \put(0.10140407,0.69128176){\color[rgb]{0,0,0}\makebox(0,0)[lb]{\smash{$a$}}}%
    \put(0.90600728,0.11435637){\color[rgb]{0,0,0}\makebox(0,0)[lb]{\smash{$b$}}}%
    \put(0.80107775,0.39173709){\color[rgb]{0,0,0}\makebox(0,0)[lb]{\smash{$\Gamma^+$}}}%
    \put(0.10757076,0.53886797){\color[rgb]{0,0,0}\makebox(0,0)[lb]{\smash{$\Gamma^-$}}}%
  \end{picture}%
\endgroup%

}
\caption{\small{(a): an example of $\immmapmin$ with self-intersections 
admitting a semicartesian parametrization. We also plot
 the
intersection of $\immmapmin$
 with the plane $\{\overline{t}=0\}$: this is a non-simple 
curve connecting $(\overline t, 
\gamma^-(\overline{t}))$ and $(\overline{t}, \gamma^+(\overline{t}))$. 
(b): an other analytic curve
$\Gamma$ leading to a $\immmapmin$ admitting a semicartesian parametrization. 
In this case $\gamma^-$ is approximatively constant
in $[a+\delta, b-\delta]$ for some small $\delta>0$, 
so that its graph $\Gamma ^-$ is almost a segment (we cannot
require constancy due to analiticity). 
The graph $\Gamma^+$ of $\gamma^+$ is, instead,  an helix
around $\Gamma^-$. It is clear that 
this situation is very far from the non-parametric
case. The qualitative properties of $\Gamma$ 
in correspondence to the points $a$ and $b$ are not arbitrary,
and will be discussed in detail in the next sections (see also
the assumptions in Theorem \ref{teo:secondo}).}}
\label{fig:intro}   
\end{figure}

Let us now inspect the delicate problem of the  existence of 
a domain $D \subset \R^2_{(t,s)}$ and 
a semicartesian
parametrization $X: D \to \R^3$. 
Besides the non-parametric case, 
in this paper we exhibit other sufficient conditions 
for the existence of a semicartesian parametrization,
and 
we refer to Theorem \ref{prop:analytic} for all details.

\begin{Theorem}\label{teo:secondo}
Suppose that $\Gamma$ admits a parametrization which is analytic, and nondegenerate 
in the sense of \eqref{eq:nondege} at the junctions
between $\gamma^-$ and $\gamma^+$. 
Then $\Sigma_{\min}$ admits a semicartesian parametrization. 
\end{Theorem}

Before commenting on
 the proof,
which represent maybe the most technical part of the present paper,
 we want to briefly discuss Figure 
\ref{fig:manuintro}, since  it is a sort of prototypical example
in our work.
The boundary of the represented surface satisfies all hypotheses of
Theorem \ref{teo:secondo}. It is built as the union of two graphs of two analytic maps
$\gamma^\pm: \ab\to \R^2_{(\xi,\eta)}$.
We take the graph of $\gamma^-$ as the (planar) 
half-circle starting from the south pole $S$
and ending at the north pole $N$.
The graph of $\gamma^+$
is the remaining part of the boundary. 
Clearly
$\gamma^-$ and $\gamma^+$ {\it join 
in an analytic way}. 
We stress that  for 
$\overline{t}\in (a,b)$ the intersection of the plane $\{t=\overline{t}\}$ with $\Gamma$
is just the set of points $\{ (\overline{t},\gamma^-(\overline{t})),(\overline{t},\gamma^+(\overline{t}))\}$,
while the intersection with the surface $\immmapmin$ is a connected, possibly non-simple, 
curve\footnote{The surface in \cite[Figure 8.1.2]{Mor:88} 
mentioned in footnote (10) does not satisfy this property.}.
Moreover, near the two poles, $\Gamma$ is essentially a circumference, and
this implies, as we shall see later 
({\tt step 4} in the proof of Theorem \ref{lem:piani})
that the nondegeneracy  assumption
mentioned in the statement of Theorem \ref{teo:secondo}
is satisfied.

The analiticity of $\Gamma$ in Theorem \ref{teo:secondo}
is a strong assumption: indeed it forces
$\mappa$ to have a rather rigid structure, 
in particular near the crack tips,
and it also implies that 
the traces $\mappa^- $ and $\mappa^+$ cannot be independent. 
As we shall clarify below, 
the reason for which we require  analyticity is that we need
to exclude of branch points and 
boundary branch points on $\immmapmin$.
Finding sufficient conditions
on $\Gamma$ ensuring the existence of a semicartesion
parametrization of $\Sigma_{\min}$, without assuming
analyticity, requires further investigation. 

Roughly speaking, the proof of Theorem \ref{teo:secondo}
runs as follows. 
First we need  to guarantee that
no  plane orthogonal to the $t-$axis 
is tangent to  $\immmapmin$ since,
under this transversality condition,
a classical result provides
a {\it local} semicartesian parametrization (Theorems \ref{lem:piani}
and \ref{theo:trasv}).
Let us
 consider a conformal parametrization $Y$ of $\immmapmin$ defined 
on the unit disk $B$; thanks to the analyticity of $\Gamma$, it is possible 
to extend $\immmapmin$ to a minimal surface $\extimmmapmin$, 
parametrized on $\discoest$, an open set containing $B$,
 by
an analytic map $\extmapmin = (\extmapmin_1,\extmapmin_2,\extmapmin_3)$ coinciding with $Y$ on $B$.
Now we define a {\it height} function $h$, defined on $\discoest$
and returning for each point $(u,v)$ the $t-$coordinate of its image
through $\extmapmin$, that is
\begin{displaymath}
\begin{split}
h: \discoest & \to  \R_t ,\\
h(u,v) & := \extmapmin_1(u,v).
\end{split}
\end{displaymath}  
We now observe that the tangent plane to $\extimmmapmin$ 
at $\extmapmin (u,v)$ is orthogonal to the $t-$axis if and only if
$(u,v)$ is a critical point for $h$.
Thus in order to get the desired transversality property,
we need to exclude the presence of critical points of $h$ on $\overline{B}$,
except for a minimum and a maximum on $\partial B$, which exist since
$h$ is continuous. 
Internal maxima and minima are excluded by a geometric argument, and
saddle points are excluded by using a Morse relation 
for closed domains (see Appendix \ref{sec:appb}). In this step,
proven in Theorem \ref{lem:piani}, 
the analiticity of $\Gamma$ is once more crucial,
because it prevents $\immmapmin$ to have boundary or internal branch 
points; this regularity and the nondegeneracy hypotheses on
the parametrization of $\Gamma$
imply that $h$ is a Morse function satisfying the
requirements of Theorem \ref{teo:mors}.

In this way we have obtained 
the existence of a \textit{local} semicartesian parametrization.
Using 
 the simple connectedness of 
$\immmapmin$, it is finally possible to globalize the argument,
and 
provide a semicartesian parametrization (Section \ref{sec:global}).
We notice here that several properties of the (a priori
unknown) parameter domain $D$ can be proven, as shown in 
Section \ref{sec:domshape}: in particular, it turns out that 
$\partial D$ is union of the graphs of two functions $\sigma^\pm$, 
which are locally Lipschitz
(but not Lipschitz) with a local Lipschitz constant 
controlled by the Lipschitz constant of $\gamma^\pm$. 
We refer to Section \ref{sec:par} for the details of the proofs,
but it is clear  that the analyticity assumption
is fundamental in most of the arguments.

\bigskip 
The plan of the paper is the following.
In Section \ref{sec:notation} we fix some notation and we introduce
the space $\DM$ (some properties of which are given in Section \ref{sec:appdomain}). 
We also give the definition of semicartesian parametrization.
In Section \ref{sec:graph} we prove Theorem \ref{teo:primo} for maps whose associated
Plateau's problem admits a non parametric solution.
In Section \ref{sec:general} we provide a generalization of this result
for possibly self-intersecting area-minimizing  surfaces,
underlying that what is really important is that the solution of the Plateau's problem
admits a semicartesian parametrization.
In Section \ref{sec:semicart} we give
some sufficient conditions on $\mappa$ for the existence of a semicartesian parametrization
of $\immmapmin$, see Theorem \ref{prop:analytic},
the proof of which is given in Section \ref{sec:par} and is 
the most technical part of the paper.
In Sections \ref{sec:app} and \ref{sec:appb} we collect some classical 
results of minimal surfaces and Morse Theory needed in our proofs.

\section{Notation}\label{sec:notation}
If $n 
\geq 2$, we denote by $\cdot, \vert \cdot \vert$ the euclidean scalar 
product and norm in $\R^n$, respectively, and by $\overline E$ and ${\rm 
int}(E)$ the closure and the interior part of a set $E\subseteq \R^n$. 
$\mathcal H^2$ is the Hausdorff measure in $\R^n$ and $\mathcal{L}^2$ is the Lebesgue measure in $\R^2$.
$B\subset \R^2=\R^2_{(\parahilduno,\parahilddue)}$ is the open unit disk and  
$\partial \disk$ is its boundary. 
We choose an arc-length parametrization 
\begin{equation}\label{eq:bordo_palla}
\parabordopalla: 
\theta \in [0,2\pi)\to \parabordopalla(\theta) \in \partial \disk,
\end{equation} 
and take $\theta_{\rm s}, \theta_{\rm n}
 \in [0,2\pi)$, with $\theta_{\rm s} < \theta_{\rm n}$, so that
$$
\parabordopalla(\theta_{\rm s}) = (0,-1), \qquad 
\parabordopalla(\theta_{\rm n}) = (0,1).
$$
For a differentiable map $Y : B \to \R^3$,
the components are denoted by $Y = (Y_1, Y_2, Y_3)$, and the 
partial derivatives by $Y_\phu = \partial_\phu Y=
(\partial_\phu Y_{1}, \partial_\phu Y_{2}, \partial_\phu Y_{3})
$ 
and $Y_\phd = \partial_\phd Y 
=
(\partial_\phd Y_{1}, \partial_\phd Y_{2}, \partial_\phd Y_{3})$.

\smallskip

$\Omega$ is a bounded open subset of the source space
$\source$, while 
the target space is denoted by $\target$. 
When no confusion is possible, we often write $\R^2$ in place of
the source or of the target space.

As in the introduction, if $\mappav \in \BV (\Om;\R^2)$ we denote
by $\grad \mappav$ and $D^s \mappav$ the absolutely continuous and 
the singular part of the distributional 
gradient of $\mappav$, respectively.  

With $\DM$ we denote the subset of $\BV(\Om;\R^2)$ on which 
the relaxed area functional admits the following integral representation:
\begin{equation}\label{eq:space_D}
\rel(\mappav, \Om)= \int_\Om |\M\left(\grad \mappav \right)|\,dx\,dy<+\infty.
\end{equation}
As we have already noticed in the introduction, $W^{1,p}(\Om;\R^2)$ is contained in $\DM$ 
for every $p\in [2,+\infty]$. 
In Appendix \ref{sec:appdomain} we report the characterization of $\DM$  
given in \cite{AcDa:94} and we prove that the functional 
$\rel$ can be obtained also by relaxing from $\DM$. 

\smallskip

We now give the useful definition
of \textit{semicartesian parametrization}.

\begin{Definition}[\textbf {Union of two graphs}]\label{def:semicart_curve}\textup{
A closed simple rectifiable curve 
$\Gamma\subset\R^3=\R_t \times \R^2_{(\xi,\eta)}$, is said to be 
 \textit{union of two graphs} 
if there exists an interval $\ab\subset \R_t$ such that
$\Gamma$ is the union of the graphs of two continuous maps 
$\gamma^\pm\in\C(\ab;\R^2)\cap{\rm Lip}_{\rm loc}((a,b);\R^2)$.
That is $\Gamma=\Gamma^+\cup\Gamma^-$ where
\begin{displaymath}
\Gamma^\pm=\{(t,\xi,\eta): t\in \ab, (\xi, \eta)= \gamma^\pm(t)\}.
\end{displaymath}
When necessary, we shall say that $\Gamma$ is union of the graphs of $\gamma^\pm$.
}\end{Definition}

\begin{Definition}[\textbf{Semicartesian parametrization}]\label{def:semicart_par}
\textup{
A disk-type surface $\Sigma$ in $\R^3$ (possibly with self intersections) is said to admit
a \textit{semicartesian parametrization} if $\Sigma=X(\dommap)$, where
\begin{itemize}
\item[-] 
$\dommap \subset \R^2_{(t,s)}$ is given by
\begin{equation}\label{eq:formadom}
\dommap=\{(\ppa,\spa):  t\in\ab,\, \sigma^-(t)\leq s \leq\sigma^+(t)\},
\end{equation}
with $\sigma^\pm\in {\rm Lip}_{\rm loc}((a,b))$ satisfying
\begin{equation}\label{eq:sigmapm}
\begin{split}
\sigma^-(a)&=0=\sigma^+(a),
\\
\sigma^-(b)&=\sigma^+(b),
\\
\sigma^-&<\sigma^+ \mbox{ in }(a,b);
\end{split}
\end{equation}
\item[-] $X\in W^{1,2}(\dommap;\R^3)$ has the following form: 
\begin{equation}\label{eq:good_par}
X(t,s)=(t,X_2(\ppa,\spa),X_3(\ppa,\spa)) \quad {\rm a.e.}~ (t,s)\in \dommap.
\end{equation}
\end{itemize}
}
\end{Definition}
Sometimes we refer to a semicartesian parametrization
as to a global semicartesian parametrization; on the other
hand, a local semicartesian parametrization is a $W^{1,2}$ map of 
the form \eqref{eq:good_par}, defined in a neighourhood of
a point. 

\section{Non-parametric case: graph over a convex domain}\label{sec:graph}
As explained in the introduction,
our aim is
 to estimate from above the area of the graph of a di\-scon\-ti\-nuous map 
with a curve discontinuity compactly contained in $\Om$. 
In this section we study a case which leads to consider a non-parametric
Plateau's problem over a convex domain.

%%%%%%%%%%%%%%%%%%%%%%%%%%%%%%%%%%%%%%%%%%%%%%%%%%%%%%%%%%%%
\subsection{Hypotheses on $\mappa$  and statement for the non-parametric case}\label{hypstate}
%%%%%%%%%%%%%%%%%%%%%%%%%%%%%%%%%%%%%%%%%%%%%%%%%%%%%%%%%%%
Let $\Om \subset \R^2=\source$ be a bounded open set  and assume that 
$$
\mappa = (\mappauno, \mappadue) : \Om \to \R^2 = \target
$$
satisfies the following properties $({\rm u}1)-({\rm u}4)$:
\begin{itemize}
\item[({\rm u}1)] $\mappa \in \BVo\cap L^\infty(\Omega; \R^2)$ 
and $\osaltou$ is a
non-empty \textit{simple} curve of class $\mathcal C^2$ (not reduced to a point)
contained in $\Omega$. We shall write
$$
\osaltou = \alpha(\ab),
$$
where
$a$ and $b$ are two real numbers
with $a<b$, and 
$$ 
\alpha : \ppa
\in \ab \subset \R = \R_\ppa \to \alpha(\ppa) \in  
\osaltou
$$ 
is an {\it arc-length} 
parametrization of $\osaltou$ of class $\mathcal C^2$. Note that we are assuming that 
if $t_1, t_2 \in \ab$, $t_1 \neq t_2$ then $\alpha(t_1) \neq 
\alpha(t_2)$, and moreover
\begin{equation*}\label{eq:nonintersbound}
\osaltou \cap \partial \Om = \emptyset.
\end{equation*}
In particular, the two distinct 
crack tips are $\osalto \setminus \saltou = \{\alpha(a), \alpha(b)\} 
\subset \Omega$ (see Figure \ref{fig:Om}).
\item[({\rm u}2)] $\mappa \in W^{1,\infty}\left(\Om\setminus\osaltou;\R^2\right)$;  
by the Sobolev embeddings (see for example \cite[Theorem 4.12]{Ad:40})
we have  $\mappa \in \C\left(\Om\setminus\osaltou;\R^2\right)$.
\end{itemize}
As a consequence of ({\rm u}1) and ({\rm u}2), 
there exists the trace of $\mappa$ on $\osaltou$ on each side of the jump:
\begin{equation*}
\begin{aligned}
\gammam(t) & = \gammam[\mappa](t) = (\gammamu(t), \gammamd(t)) := \mappa^-(\alpha(\ppa)) \in \R^2, 
\\
\gammap(t)  &= \gammap[\mappa](t) = (\gammapu(t), \gammapd(t)) := \mappa^+(\alpha(\ppa)) \in \R^2, 
\end{aligned}
\qquad \ppa \in [a,b],
\end{equation*}
and the functions
$$
t\in\ab \longrightarrow\gammapm(t)
$$
are H\"older continuous\footnote{
Indeed $W^{1,\infty} (U) 
\subset \C^{0, \lambda}(\overline U)$ for every $\lambda \in(0,1)$ if $U
\subset \R^2$ is  
a smooth  enough open set (see again \cite[Theorem 4.12]{Ad:40}), and we can consider,
 for each trace, a sufficiently smooth open set $U \subset \Omega$ 
such that $\osaltou \subset \partial U$.}.
 
Notice that 
\begin{equation}\label{eq:estremi}
\gammam(a)=\gammap(a), \qquad \gammam(b)=\gammap(b).
\end{equation}
\begin{itemize}
\item[({\rm u}3)] 
$\gammapm \in {\rm Lip}(\ab;\R^2)$ and 
there exists  a finite set of 
points $t_0:=a < t_1<\dots<t_m<t_{m+1}=b$ of $\ab$
such that $\gammapm \in \mathcal 
C^2((t_i, t_{i+1}))\cap
\C^1([t_i, t_{i+1}])$ 
for any 
$i=0,\dots,m$. Moreover we require
\begin{equation}\label{eq:interni}
\gamma^-(t) \neq \gamma^+(t), \qquad
t \in (a,b).
\end{equation}
\end{itemize}
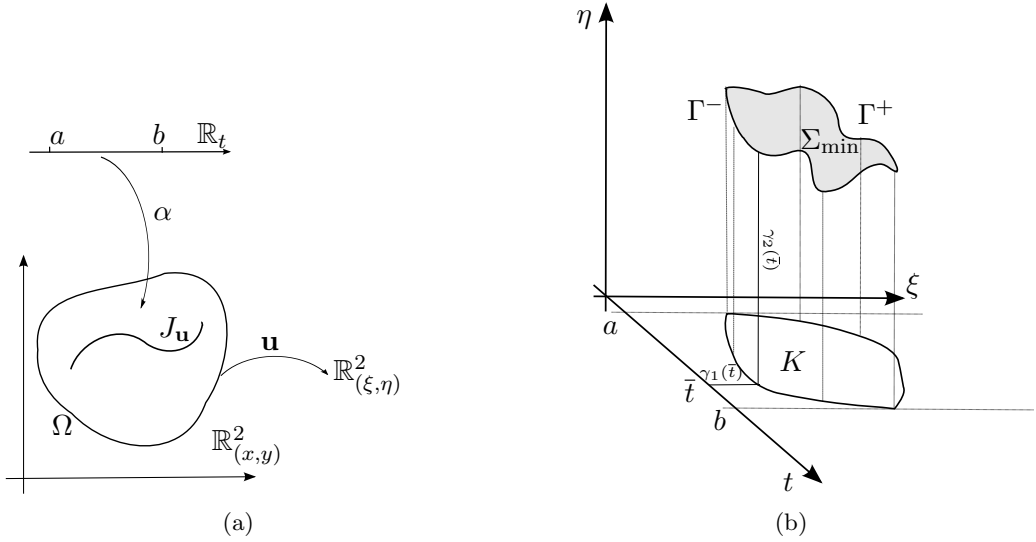
\begin{figure}[htbp]
\centering%
\subfigure [\label{fig:Om}]%
{
\def\svgwidth{6.5cm}
%% Creator: Inkscape inkscape 0.48.4, www.inkscape.org
%% PDF/EPS/PS + LaTeX output extension by Johan Engelen, 2010
%% Accompanies image file 'omega.pdf' (pdf, eps, ps)
%%
%% To include the image in your LaTeX document, write
%%   \input{<filename>.pdf_tex}
%%  instead of
%%   \includegraphics{<filename>.pdf}
%% To scale the image, write
%%   \def\svgwidth{<desired width>}
%%   \input{<filename>.pdf_tex}
%%  instead of
%%   \includegraphics[width=<desired width>]{<filename>.pdf}
%%
%% Images with a different path to the parent latex file can
%% be accessed with the `import' package (which may need to be
%% installed) using
%%   \usepackage{import}
%% in the preamble, and then including the image with
%%   \import{<path to file>}{<filename>.pdf_tex}
%% Alternatively, one can specify
%%   \graphicspath{{<path to file>/}}
%% 
%% For more information, please see info/svg-inkscape on CTAN:
%%   http://tug.ctan.org/tex-archive/info/svg-inkscape
%%
\begingroup%
  \makeatletter%
  \providecommand\color[2][]{%
    \errmessage{(Inkscape) Color is used for the text in Inkscape, but the package 'color.sty' is not loaded}%
    \renewcommand\color[2][]{}%
  }%
  \providecommand\transparent[1]{%
    \errmessage{(Inkscape) Transparency is used (non-zero) for the text in Inkscape, but the package 'transparent.sty' is not loaded}%
    \renewcommand\transparent[1]{}%
  }%
  \providecommand\rotatebox[2]{#2}%
  \ifx\svgwidth\undefined%
    \setlength{\unitlength}{442.98661923bp}%
    \ifx\svgscale\undefined%
      \relax%
    \else%
      \setlength{\unitlength}{\unitlength * \real{\svgscale}}%
    \fi%
  \else%
    \setlength{\unitlength}{\svgwidth}%
  \fi%
  \global\let\svgwidth\undefined%
  \global\let\svgscale\undefined%
  \makeatother%
  \begin{picture}(1,0.7514599)%
    \put(0,0){\includegraphics[width=\unitlength]{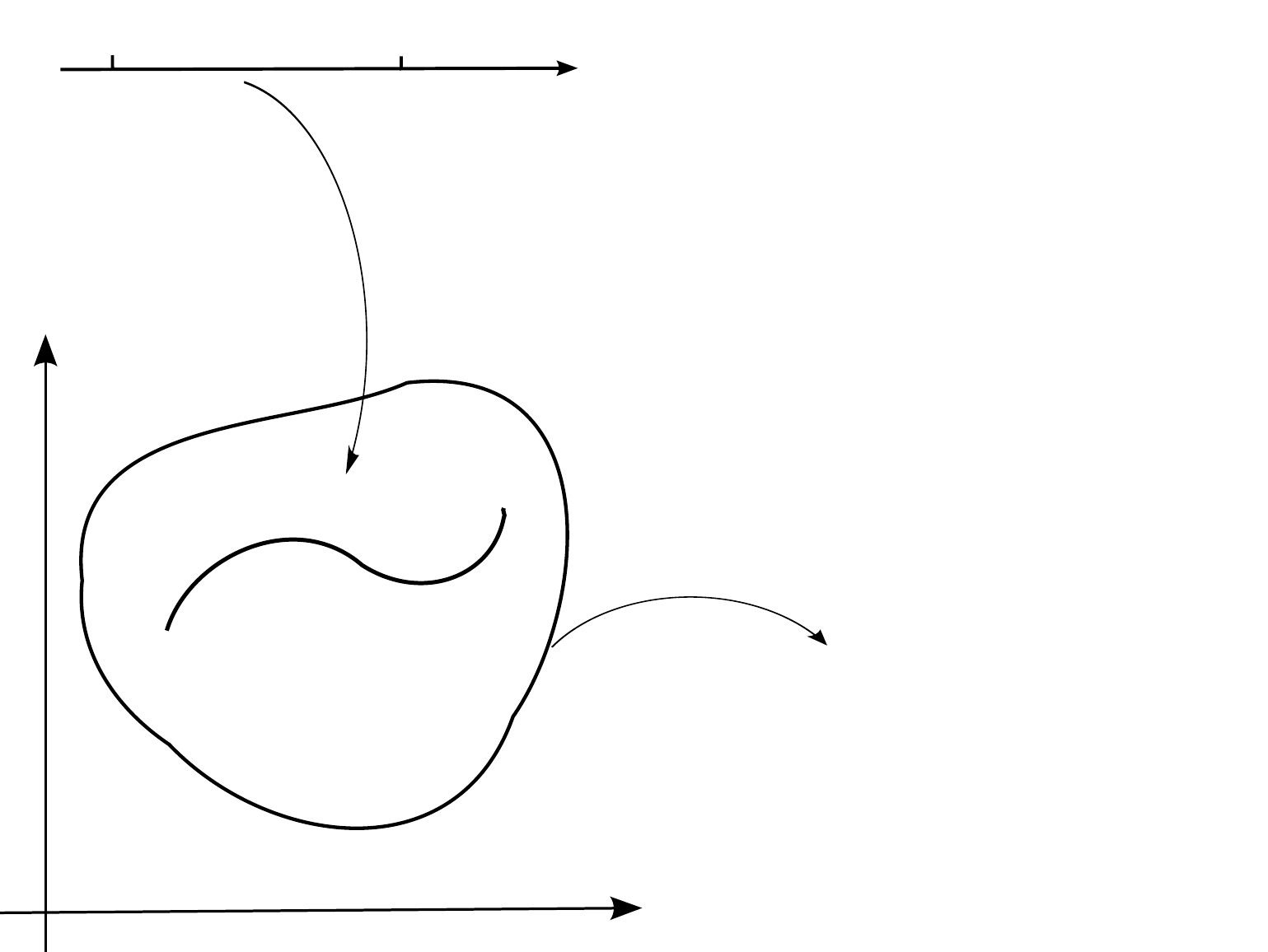}}%
    \put(0.29898726,0.55928292){\color[rgb]{0,0,0}\makebox(0,0)[lb]{\smash{$\alpha$}}}%
    \put(0.30940215,0.31993892){\color[rgb]{0,0,0}\makebox(0,0)[lb]{\smash{$J_\mappa$}}}%
    \put(0.09334493,0.12028837){\color[rgb]{0,0,0}\makebox(0,0)[lb]{\smash{$\Om$}}}%
    \put(0.41717187,0.08884202){\color[rgb]{0,0,0}\makebox(0,0)[lb]{\smash{$\R^2_{(x,y)}$}}}%
    \put(0.66307451,0.23226659){\color[rgb]{0,0,0}\makebox(0,0)[lb]{\smash{$\R^2_{(\xi,\eta)}$}}}%
    \put(0.51531358,0.29275849){\color[rgb]{0,0,0}\makebox(0,0)[lb]{\smash{$\mappa$}}}%
    \put(0.08882608,0.71332076){\color[rgb]{0,0,0}\makebox(0,0)[lb]{\smash{$a$}}}%
    \put(0.29684908,0.71348486){\color[rgb]{0,0,0}\makebox(0,0)[lb]{\smash{$b$}}}%
    \put(0.38746213,0.71451583){\color[rgb]{0,0,0}\makebox(0,0)[lb]{\smash{$\R_{\ppa}$}}}%
  \end{picture}%
\endgroup%

}\qquad
\subfigure[ \label{fig:Sigma}]%
{
\def\svgwidth{6cm}
%% Creator: Inkscape inkscape 0.48.4, www.inkscape.org
%% PDF/EPS/PS + LaTeX output extension by Johan Engelen, 2010
%% Accompanies image file 'sigma.pdf' (pdf, eps, ps)
%%
%% To include the image in your LaTeX document, write
%%   \input{<filename>.pdf_tex}
%%  instead of
%%   \includegraphics{<filename>.pdf}
%% To scale the image, write
%%   \def\svgwidth{<desired width>}
%%   \input{<filename>.pdf_tex}
%%  instead of
%%   \includegraphics[width=<desired width>]{<filename>.pdf}
%%
%% Images with a different path to the parent latex file can
%% be accessed with the `import' package (which may need to be
%% installed) using
%%   \usepackage{import}
%% in the preamble, and then including the image with
%%   \import{<path to file>}{<filename>.pdf_tex}
%% Alternatively, one can specify
%%   \graphicspath{{<path to file>/}}
%% 
%% For more information, please see info/svg-inkscape on CTAN:
%%   http://tug.ctan.org/tex-archive/info/svg-inkscape
%%
\begingroup%
  \makeatletter%
  \providecommand\color[2][]{%
    \errmessage{(Inkscape) Color is used for the text in Inkscape, but the package 'color.sty' is not loaded}%
    \renewcommand\color[2][]{}%
  }%
  \providecommand\transparent[1]{%
    \errmessage{(Inkscape) Transparency is used (non-zero) for the text in Inkscape, but the package 'transparent.sty' is not loaded}%
    \renewcommand\transparent[1]{}%
  }%
  \providecommand\rotatebox[2]{#2}%
  \ifx\svgwidth\undefined%
    \setlength{\unitlength}{299.35140943bp}%
    \ifx\svgscale\undefined%
      \relax%
    \else%
      \setlength{\unitlength}{\unitlength * \real{\svgscale}}%
    \fi%
  \else%
    \setlength{\unitlength}{\svgwidth}%
  \fi%
  \global\let\svgwidth\undefined%
  \global\let\svgscale\undefined%
  \makeatother%
  \begin{picture}(1,1.08739944)%
    \put(0,0){\includegraphics[width=\unitlength]{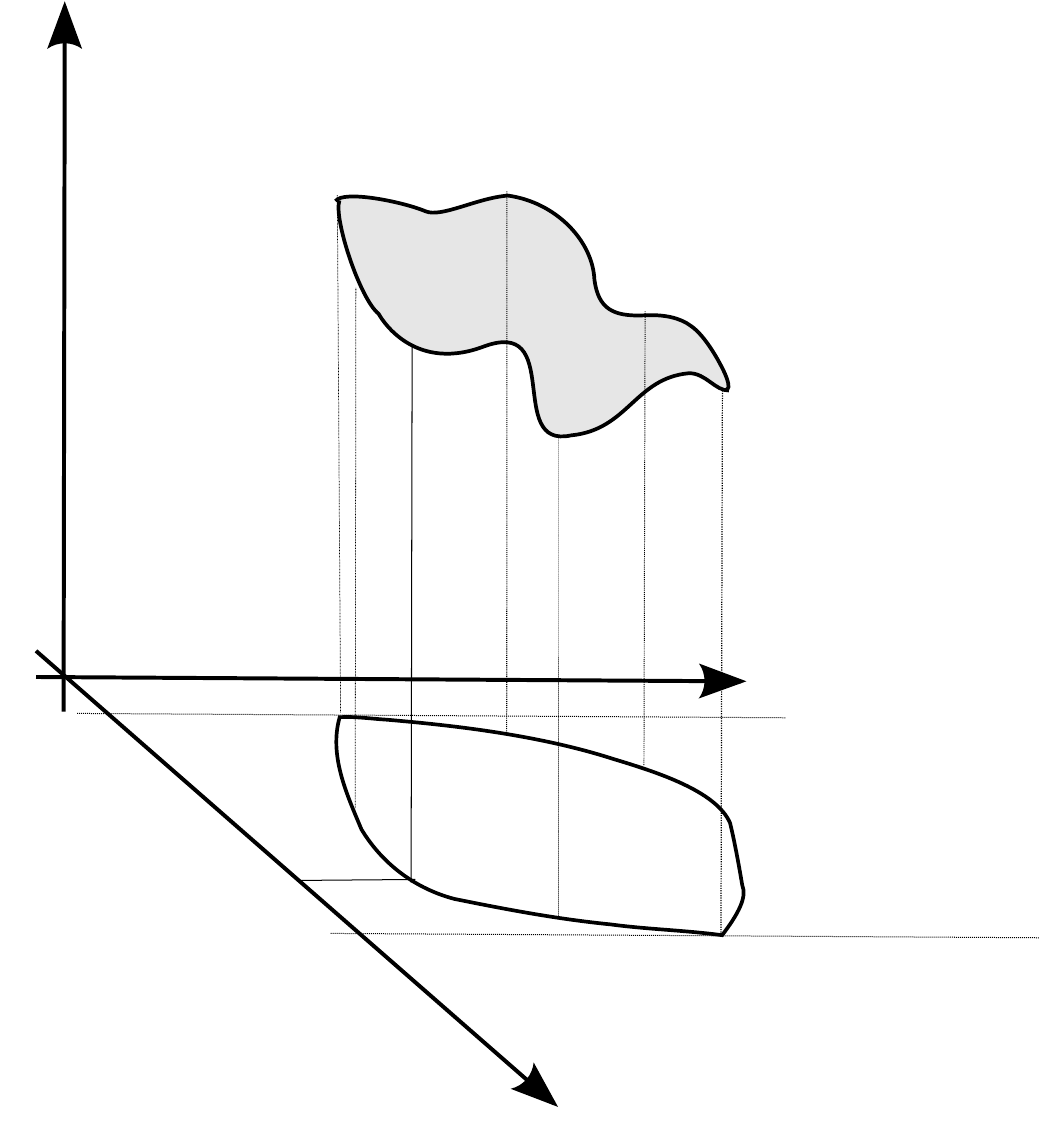}}%
    \put(0.44260097,0.27173934){\color[rgb]{0,0,0}\makebox(0,0)[lb]{\smash{$K$}}}%
    \put(0.48774735,0.76272637){\color[rgb]{0,0,0}\makebox(0,0)[lb]{\smash{$\immmapmin$}}}%
    \put(0.2396112,0.82554525){\color[rgb]{0,0,0}\makebox(0,0)[lb]{\smash{$\Gamma^-$}}}%
    \put(0.61384334,0.81785465){\color[rgb]{0,0,0}\makebox(0,0)[lb]{\smash{$\Gamma^+$}}}%
    \put(-0.00270633,1.04114598){\color[rgb]{0,0,0}\makebox(0,0)[lb]{\smash{$\eta$}}}%
    \put(0.71991761,0.44290142){\color[rgb]{0,0,0}\makebox(0,0)[lb]{\smash{$\xi$}}}%
    \put(0.05494419,0.35273123){\color[rgb]{0,0,0}\makebox(0,0)[lb]{\smash{$a$}}}%
    \put(0.297565,0.14119281){\color[rgb]{0,0,0}\makebox(0,0)[lb]{\smash{$b$}}}%
    \put(0.45026326,0.0067753){\color[rgb]{0,0,0}\makebox(0,0)[lb]{\smash{$t$}}}%
    \put(0.23435364,0.20243725){\color[rgb]{0,0,0}\makebox(0,0)[lb]{\smash{$\overline{t}$}}}%
    \put(0.26716226,0.25745068){\color[rgb]{0,0,0}\makebox(0,0)[lb]{\smash{\tiny{$\gamma_1(\overline{t})$}}}}%
    \put(0.41386677,0.58596935){\color[rgb]{0,0,0}\rotatebox{-89.33803781}{\makebox(0,0)[lb]{\smash{\tiny{$\gamma_2(\overline{t})$}}}}}%
  \end{picture}%
\endgroup%

}
\caption{\small{(a): the domain $\Om$, 
the arc-length parametrization of the jump of the map $\mappa$. Notice 
that the closure of the jump is contained in $\Om$.
(b): the Lipschitz curve $\Gamma$, union of the graphs on $\ab$ of the vector valued functions 
$\gamma^-$ and $\gamma^+$. $K$ is a closed convex set in $\R^2_{(t,\xi)}$, having non empty interior, 
and $\immmapmin$ is the area-minimizing surface spanning $\Gamma$. We observe that 
$\partial K$ is not differentiable at $(a,\gamma_1^+(a))$ and $(b,\gamma_1^+(b))$, and $\Gamma$ is not differentiable 
at $(a, \gamma^+(a))$, $(b,\gamma^+(b))$.
}}    
\end{figure}
In order to state our last assumption $({\rm u}4)$, 
we denote by $\Gamma^\pm = \Gamma^\pm[\mappa]$ 
the graphs of the maps $\gammapm$, 
$$
\begin{aligned}
\Gamma^-  = \Gamma^-[\mappa]& := \{(\ppa, 
\primacoordtarget, \secondacoordtarget) \in \ab \times 
\R^2 : 
(\primacoordtarget, \secondacoordtarget) = 
\gamma^- (\ppa)\},
\\
\Gamma^+ = \Gamma^+[\mappa]& := 
\{(\ppa, \primacoordtarget, \secondacoordtarget) \in \ab \times \R^2 : 
(\primacoordtarget, \secondacoordtarget) = 
\gamma^+(\ppa)\},
\end{aligned}
$$
and we set
\begin{equation}\label{eq:Gammau}
\Gamma = \Gamma[\mappa]:=\Gamma^- \cup\Gamma^+.
\end{equation}
In view of assumptions ({\rm u}2) and ({\rm u}3), $\Gamma\subset \R^3$
is a closed, simple, Lipschitz and piecewise $\C^2$ curve obtained as union
of two curves; 
moreover $(a,\gamma^+(a))$ and $(b,\gamma^+(b))$ (coinciding with  
$(a,\gamma^-(a))$ and $(b,\gamma^-(b))$ respectively) are {\it nondifferentiability 
points} of $\Gamma$.
The next assumption requires introducing the projection on a plane
spanned by $t$ and one of the two coordinates, say $\xi$,
 in the target space
$\target$.
We suppose that:
\begin{itemize}
 \item[({\rm u}4)] the orthogonal projection of $\Gamma$ on the plane 
 $\R^2_{(t,\primacoordtarget)}$ is the boundary of a closed convex set 
 $\K$ with non-empty interior. 
In particular, without loss of generality, 
$$
\gammamu(t)<\gammapu(t)\, , \quad t\in(a,b),
$$
and we assume that $\gammamu$ is convex and $\gammapu$ is concave. 
Moreover thanks to hypothesis ({\rm u}3), 
$$
\gammapm_1 \in {\rm Lip}(\ab)
$$
and therefore $(a,\gammamu(a))$ and $(b, \gammamu(b))$ 
are nondifferentiability points of $\partial \K$. 
\end{itemize} 

Summarizing, 
$\partial K = {\rm graph}(\gamma_1^-) \cup\,
{\rm graph}(\gamma_1^+)$ is of class $\mathcal C^1$ up 
to a finite set of points containing
 $(a,\gammamu(a))$ and  $(b, \gammamu(b))$. In particular,
$\partial K$ is {\it not} of class $\mathcal C^2$.

\medskip
\begin{Remark}\label{rem:angoli}\rm
The hypothesis that $\Gamma$ has corners
in 
$(a,\gamma^-(a))$ and $(b, \gamma^-(b))$
is related to the regularity assumptions made on
$\mappa$ in $({\rm u}2)$: requiring 
that $\Gamma$ is differentiable at 
$\left(a,\gamma^-(a)\right)$ and $\left(b, \gamma^-(b)\right)$ would {\it prevent} 
$\mappa$
to belong to $W^{1,\infty}\left(\Omega \setminus \osalto; \R^2\right)$. 
On the other hand, it is useful to require  
$\mappa \in W^{1,\infty}\left(\Omega \setminus \osalto; \R^2\right)$: 
indeed, in this case, we can infer (see the proof of 
Theorem \ref{teo:graph_main}, for example  {\tt step 8}) that the approximating maps
$\ue$ are Lipschitz and thus in particular that they can be used to estimate $\rel(\mappa,\Om)$.
In Section \ref{sec:general} we manage in weakening this requirement
(compare condition (\hp 2)).
\end{Remark}
%
%\begin{Remark}\label{rem:angoli}\rm
%It is essential in our argument that the jump is compactly contained in the domain 
%$\Omega$. The cases in which one end point of the jump or both lie on the $\partial \Om$
%could be treated by considering different kinds of Plateau Problem. 
%\end{Remark}
\smallskip

Before stating our first result, we need the following definition 
(for further details, see Section \ref{sec:app}).

\begin{Definition}\label{defSigmamin}
\textup{We denote by $\immmapmin\subset \R^3=\R_\ppa\times\R^2_{(\xi,\eta)}$ 
an area-minimizing surface of disk-type spanning $\Gamma$, that is the image of the unit disk 
through a solution of the Plateau's problem \eqref{eq:plateau} for $\Gamma$.}
\end{Definition} 
Now we are in a position to state our first theorem.
\begin{Theorem}\label{teo:graph_main}
Suppose that $\mappa$ satisfies assumptions $({\rm u}1)$-$({\rm u}4)$.
Then there exists a sequence 
\begin{equation}\label{uelip}
(\ue)_\eps \subset \rm{Lip}(\Om;\R^2)
\end{equation}
converging to $\mappa$ in $L^1(\Omega; \R^2)$ as $\eps \to 0^+$ 
such that 
\begin{equation}\label{eq:lim_ue}
		\lim_{\eps \to 0^+} \rel(\ue,\Omega)=
\rel(\mappa,\Omega\setminus \osaltou)+
\H^2(\immmapmin)
= 
\int_{\Omega} \vert \mathcal M(\grad \mappa)\vert~dx\,dy + 
\H^2(\immmapmin).
\end{equation}  
In particular 
	\begin{equation}\label{limsup}
		\rel(\mappa,\Omega)
\leq 
\int_{\Omega} \vert \mathcal M(\grad \mappa)\vert~dx\,dy 
+\H^2(\immmapmin).
	\end{equation}  
\end{Theorem}

%%%%%%%%%%%%%%%%%%%%%%%%%%%%%%%%%%%%%%%%%%%%%%%%%%%%%%%%
\subsection{Proof of Theorem \ref{teo:graph_main}}\label{sec:graph_proof}
%%%%%%%%%%%%%%%%%%%%%%%%%%%%%%%%%%%%%%%%%%%%%%%%%%%%%%%
The proof of Theorem \ref{teo:graph_main} is rather long, and we split it
into several steps.

\smallskip
{\tt Step 1}. Definition of the function $z$ and representation of the surface $\immmapmin$.

Since $\Gamma$ in \eqref{eq:Gammau} admits a convex 
one-to-one
parallel projection, we can apply Theorem \ref{teo:convex_proj}.
In particular, there exists a scalar function $z\in \C(\K)\cap \mathcal C^\omega({\rm int}(\K))$ 
such that 
\begin{displaymath}
\immmapmin=
\big\{(t, \primacoordtarget , \secondacoordtarget) \in \R_t \times \target:
~  (t, \primacoordtarget) \in \K,\, \secondacoordtarget =z(t, \primacoordtarget) \big\}
= {\rm graph}(z),
\end{displaymath}
where $z$ solves 
\begin{equation}\label{system}
\begin{cases}
{\rm div}\left(
\displaystyle \frac{\grad z}{\sqrt{1+|\grad z|^2}}\right)=0 & {\rm in}~ {\rm int}(K),
\\
z=\phi & {\rm on}~ \partial K,
%z=\gamma_2^- & {\rm on}~ {\rm graph}(\gamma_1^-),
%\\
%z=\gamma_2^+ & {\rm on}~ {\rm graph}(\gamma_1^+).
\end{cases}
\end{equation}
and
$$\phi=\gamma_2^\pm \quad \mbox{ on } {\rm graph}(\gamma_1^\pm). $$ 
\begin{Remark}\label{rem:chiarisce}\rm
It is worthwhile to stress  the different role played in \eqref{system} by the 
two components of the traces $\gammapm$
: the {\it first components} $\gamma_1^\pm$ determine the boundary of 
the domain $K$ where solving the non-parametric Plateau's problem, 
the Dirichlet condition of which is given by the {\it second components}
$\gamma_2^\pm$ (see Figure \ref{fig:Sigma}).
\end{Remark}

\begin{Remark}\label{rem:compatraidueapprocci}\rm
$\immmapmin$ is the unique area-minimizing surface
among all graph-like surfaces on ${\rm int}(K)$ 
satisfying  the Dirichlet condition in \eqref{system}
\end{Remark}

Due to the presence of corners in $\partial K$, we cannot
directly infer from Theorem \ref{teo:giusticonvesso} that
 $z \in {\rm Lip}(\overline K)$. Since 
the  Lipschitz regularity of $z$ is strictly related
to the Lipschitz regularity of $\ue$, in order to 
ensure inclusion \eqref{uelip}  a smoothing
argument is required (see Figure \ref{fig:graph_approximation}).

\smallskip

{\tt Step 2}. Smoothing of $\partial K$ and $\gamma_2^\pm$: 
definition of the function $z_\mu$ and of the surface $\immmapmin^\mu$.

Since $\partial K$ has only a finite number 
of nondifferentiability points, we  smoothen the corners of $K$
obtaining, for a suitable $\overline \mu>0$ small enough,
 a  sequence $(\Ke)_{\mu \in (0,\overline \mu)}$ 
of sets with the following properties:
\begin{itemize}
\item[-] each $\Ke$ is convex, closed,  with non-empty interior
and is  contained in $K$. Moreover, $\Ke$ coincides
with $K$ out of the disks  of radius $\mu$
 centered at the nondifferentiability
points of $\partial K$; 
\item [-] $\partial \Ke\in \mathcal C^2$;
% and $\partial K_\mu \subset \{(t,\xi) \in K : {\rm dist}((t,\xi), \partial K)<\mu\}$;
\item [-] $\mu_1<\mu_2$ implies $K_{\mu_1} \supset K_{\mu_2}$,
\end{itemize}
see Figure \ref{fig:D_mu}.

In order to apply Theorem \ref{teo:giusticonvesso}, we need not only to smoothen the 
set $K$, but also the Dirichlet condition $\gamma_2^\pm$ at the same time.
Firstly we observe that since both $K$ and $K_\mu$ are convex sets and 
$K_\mu\subset K$, there exist a point $O\in K_\mu$ and a projection 
$\pi_\mu$ acting as follows:
\begin{displaymath}
\begin{split}
\pi_\mu: \partial K_\mu & \to \partial K \\
p & \to \pi_\mu(p),
\end{split}
\end{displaymath}  
where $\pi_\mu(p)$ is the unique point of $\partial K$ lying on the 
half-line rising from $O$ 
and passing through $p$.

Now using this projection and again the fact that $\gamma_2^\pm$ 
are Lipschitz and piecewise $\mathcal C^2$, for every $\mu\in(0,\overline{\mu})$ we can 
define a function $\phi_\mu$ with the following properties:
\begin{itemize}
\item[-]
$\phi_\mu : \dKe \to \R$ is of class $\mathcal C^{2}$; 
\item[-]
$\phi_\mu$ coincides 
with $\gamma^\pm_2$ on $\partial K_\mu \cap \partial K$ 
out of the balls of radius $\mu$ 
centered at the nondifferentiability points 
of $\gamma_2^\pm$;
\item[-] the Hausdorff distance between the graph of $\phi_\mu$ and 
$\Gamma$ is less than $\mu$;
\item[-] there holds
\begin{equation}\label{eq:phimulip}
\frac{ \left| \phi_\mu(p)-\gamma^\pm_2 (\pi_\mu(p)) \right|}{\left|p-\pi_\mu(p)\right|}
\leq C,\qquad p\in \partial K_\mu,
\end{equation}
where $C$ is a positive constant independent of $\mu$.
\end{itemize}

For any $\mu \in (0,\overline \mu)$ let us denote by $z_{\mu}$ the solution to 
\begin{equation*}\label{systemmu}
\begin{cases}
{\rm div}\left(
\displaystyle \frac{\grad z_{\mu}}{\sqrt{1+|\grad z_{\mu}|^2}}\right)=0 & {\rm in}~ {\rm int}(K_\mu),
\\
z_{\mu}=\phi_\mu & {\rm on}~ \partial K_\mu.
\end{cases}
\end{equation*}

Theorem \ref{teo:giusticonvesso} yields
$$
z_{\mu}\in {\rm Lip}(\Ke)\cap \C^\omega({\rm int}(\Ke)).
$$
We denote by
$\Sigma_{\rm min}^{\mu}$ the graph of $z_{\mu}$.
Applying \cite[\textsection 305]{Nit:89}
it follows\footnote{An argument
leading to an equality of the type \eqref{conve} in a nonsmooth situation 
was proved in \cite{BelPao:10}.}
\begin{equation}\label{conve}
\lim_{\mu \to 0^+}
\H^2(\Sigma_{\rm min}^{\mu})= 
\H^2(\immmapmin).
\end{equation}
%
%\begin{Remark}\label{rem:llip}\rm
%By construction, we have that  $\immmapmin^{\mu}$ is Lipschitz.
%\end{Remark}

In order to assert that the maps $\mappa_\eps$ in 
{\tt step 6} are Lipschitz
continuous, in particular close to the crack tips of $J_\mappa$,
we need to extend $z_\mu$ to $K$.

\begin{figure}[htbp]
\centering%
\subfigure [\label{fig:D_mu}]%
{
\def\svgwidth{6cm}
%% Creator: Inkscape inkscape 0.48.4, www.inkscape.org
%% PDF/EPS/PS + LaTeX output extension by Johan Engelen, 2010
%% Accompanies image file 'graph_approx.pdf' (pdf, eps, ps)
%%
%% To include the image in your LaTeX document, write
%%   \input{<filename>.pdf_tex}
%%  instead of
%%   \includegraphics{<filename>.pdf}
%% To scale the image, write
%%   \def\svgwidth{<desired width>}
%%   \input{<filename>.pdf_tex}
%%  instead of
%%   \includegraphics[width=<desired width>]{<filename>.pdf}
%%
%% Images with a different path to the parent latex file can
%% be accessed with the `import' package (which may need to be
%% installed) using
%%   \usepackage{import}
%% in the preamble, and then including the image with
%%   \import{<path to file>}{<filename>.pdf_tex}
%% Alternatively, one can specify
%%   \graphicspath{{<path to file>/}}
%% 
%% For more information, please see info/svg-inkscape on CTAN:
%%   http://tug.ctan.org/tex-archive/info/svg-inkscape
%%
\begingroup%
  \makeatletter%
  \providecommand\color[2][]{%
    \errmessage{(Inkscape) Color is used for the text in Inkscape, but the package 'color.sty' is not loaded}%
    \renewcommand\color[2][]{}%
  }%
  \providecommand\transparent[1]{%
    \errmessage{(Inkscape) Transparency is used (non-zero) for the text in Inkscape, but the package 'transparent.sty' is not loaded}%
    \renewcommand\transparent[1]{}%
  }%
  \providecommand\rotatebox[2]{#2}%
  \ifx\svgwidth\undefined%
    \setlength{\unitlength}{366.33907375bp}%
    \ifx\svgscale\undefined%
      \relax%
    \else%
      \setlength{\unitlength}{\unitlength * \real{\svgscale}}%
    \fi%
  \else%
    \setlength{\unitlength}{\svgwidth}%
  \fi%
  \global\let\svgwidth\undefined%
  \global\let\svgscale\undefined%
  \makeatother%
  \begin{picture}(1,0.74268627)%
    \put(0,0){\includegraphics[width=\unitlength]{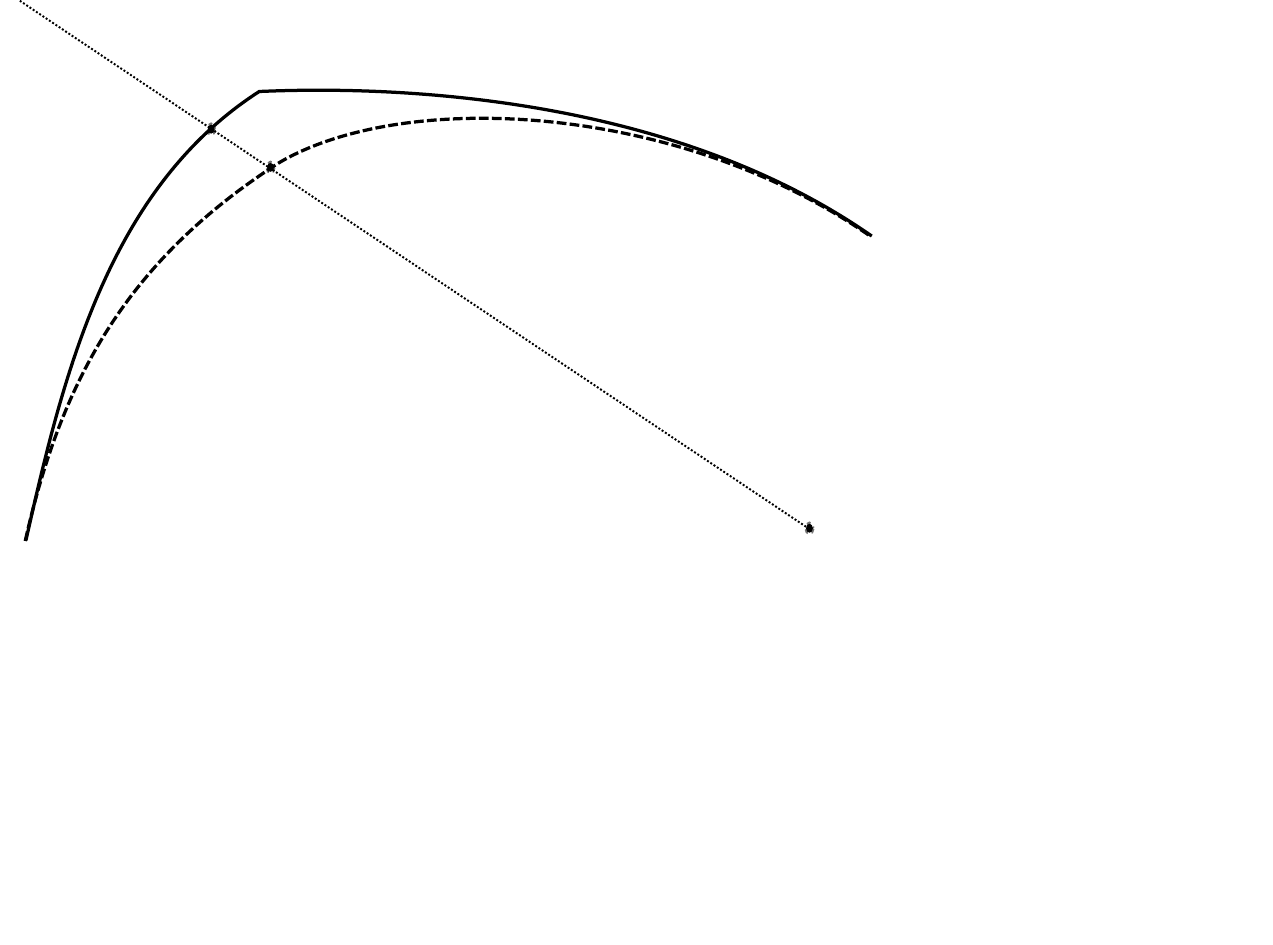}}%
    \put(0.65258451,0.3298646){\color[rgb]{0,0,0}\makebox(0,0)[lb]{\smash{$O$}}}%
    \put(0.19071726,0.55283647){\color[rgb]{0,0,0}\makebox(0,0)[lb]{\smash{$p$}}}%
    \put(-0.00162638,0.61973193){\color[rgb]{0,0,0}\makebox(0,0)[lb]{\smash{$\pi_\mu(p)$}}}%
    \put(0.24203588,0.33523657){\color[rgb]{0,0,0}\makebox(0,0)[lb]{\smash{\Large{$K_\mu$}}}}%
  \end{picture}%
\endgroup%

}\qquad\qquad
\subfigure[ \label{fig:Gamma_mu}]%
{
\def\svgwidth{7cm}
%% Creator: Inkscape inkscape 0.48.4, www.inkscape.org
%% PDF/EPS/PS + LaTeX output extension by Johan Engelen, 2010
%% Accompanies image file 'graph_approx_b.pdf' (pdf, eps, ps)
%%
%% To include the image in your LaTeX document, write
%%   \input{<filename>.pdf_tex}
%%  instead of
%%   \includegraphics{<filename>.pdf}
%% To scale the image, write
%%   \def\svgwidth{<desired width>}
%%   \input{<filename>.pdf_tex}
%%  instead of
%%   \includegraphics[width=<desired width>]{<filename>.pdf}
%%
%% Images with a different path to the parent latex file can
%% be accessed with the `import' package (which may need to be
%% installed) using
%%   \usepackage{import}
%% in the preamble, and then including the image with
%%   \import{<path to file>}{<filename>.pdf_tex}
%% Alternatively, one can specify
%%   \graphicspath{{<path to file>/}}
%% 
%% For more information, please see info/svg-inkscape on CTAN:
%%   http://tug.ctan.org/tex-archive/info/svg-inkscape
%%
\begingroup%
  \makeatletter%
  \providecommand\color[2][]{%
    \errmessage{(Inkscape) Color is used for the text in Inkscape, but the package 'color.sty' is not loaded}%
    \renewcommand\color[2][]{}%
  }%
  \providecommand\transparent[1]{%
    \errmessage{(Inkscape) Transparency is used (non-zero) for the text in Inkscape, but the package 'transparent.sty' is not loaded}%
    \renewcommand\transparent[1]{}%
  }%
  \providecommand\rotatebox[2]{#2}%
  \ifx\svgwidth\undefined%
    \setlength{\unitlength}{279.92791891bp}%
    \ifx\svgscale\undefined%
      \relax%
    \else%
      \setlength{\unitlength}{\unitlength * \real{\svgscale}}%
    \fi%
  \else%
    \setlength{\unitlength}{\svgwidth}%
  \fi%
  \global\let\svgwidth\undefined%
  \global\let\svgscale\undefined%
  \makeatother%
  \begin{picture}(1,1.02714506)%
    \put(0,0){\includegraphics[width=\unitlength]{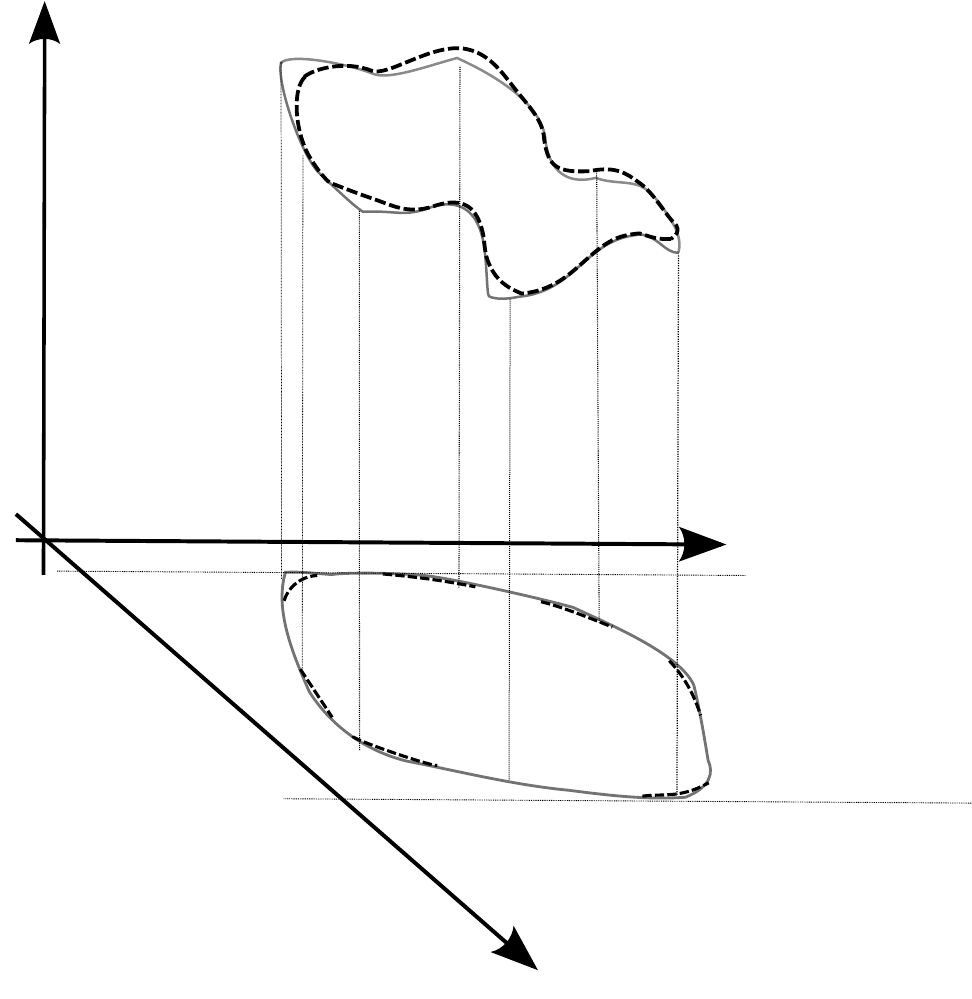}}%
    \put(-0.00289412,0.97587742){\color[rgb]{0,0,0}\makebox(0,0)[lb]{\smash{$\eta$}}}%
    \put(0.75491462,0.49006987){\color[rgb]{0,0,0}\makebox(0,0)[lb]{\smash{$\xi$}}}%
    \put(0.42148043,0.29109083){\color[rgb]{0,0,0}\makebox(0,0)[lb]{\smash{$K_\mu$}}}%
    \put(0.05523174,0.3936432){\color[rgb]{0,0,0}\makebox(0,0)[lb]{\smash{$a$}}}%
    \put(0.30887293,0.16128967){\color[rgb]{0,0,0}\makebox(0,0)[lb]{\smash{$b$}}}%
    \put(0.47035998,0.00724542){\color[rgb]{0,0,0}\makebox(0,0)[lb]{\smash{$t$}}}%
  \end{picture}%
\endgroup%

}
\caption{\small{(a): the approximation process near a nondifferentiability point
of $\partial K$ and the action of the projection map $\pi_\mu$. (b): 
the graph of the boundary value $\phi_\mu$, approximating the space curve $\Gamma$.}}  
\label{fig:graph_approximation}
\end{figure}
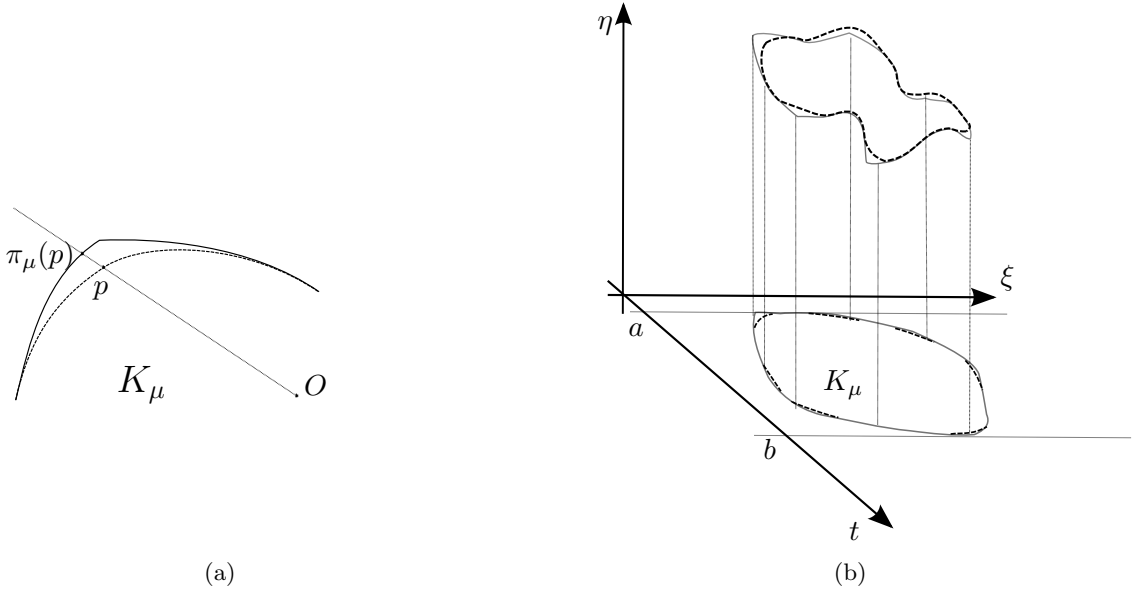

\smallskip

{\tt Step 3}. Extension  of
$z_{\mu}$ on $K$: definition of the extended surface $\Sigmamu$.

We consider again the projection $\pi_\mu$ defined in the previous step and
we observe that for every point $(\ppa,\xi) \in \K\setminus \Ke$ there exist a unique 
$p\in \dKe$ and $\rho\in (0,1] $ such that
$$
(\ppa, \xi) = \rho p + (1- \rho)\pie(p).
$$
Thus we extend $z_\mu$ to $K$ defining
\begin{displaymath}
\widehat z_{\e}(t,\xi):=
\begin{cases}
\rho \phi_\mu(p) + (1- \rho)\phi(\pie(p)), & (t,\xi) \in K\setminus K_\mu, \\
z_\mu(t,\xi), & (t,\xi) \in K_\mu.
\end{cases}
\end{displaymath}
Notice that 
\begin{equation}\label{eq:tracciazhatz}
\widehat z_{\e}=
z \qquad {\rm on}~ \partial K.
\end{equation}
We denote by $\Sigmamu$
the graph of $\widehat z_{\e}$ on $K$.
Property \eqref{eq:phimulip} gives a uniform control of the gradient 
of $\widehat{z}_\mu$ on $\K\setminus\K_\mu$, which implies that
\begin{equation*}\label{eq:nu}
\lim_{\mu\to 0^+}\H^2(\widehat z_{\mu} (K\setminus K_\e))=0.
\end{equation*}
Thus from \eqref{conve}
\begin{equation}\label{convemu}
\lim_{\mu \to 0^+}
\mathcal H^2(\Sigmamu)
= \mathcal H^2(\immmapmin).
\end{equation}
\begin{Remark}\label{rem:llipbis}\rm
By construction, we have that $\widehat z_\mu$ is Lipschitz continuous.
\end{Remark}
\smallskip

{\tt Step 4}. Definition of the  parameter space $D$.

For our goals, 
it is convenient to choose a parameter space $D$ different from $\K$, for 
parametrizing
$\immmapmin$ and 
$\Sigmamu$. 
Set
$$
\sigma(t) := \frac{\gammapu(t) - \gammamu(t)}{2}, \qquad t \in [a,b].
$$
Let $\dommap\subset\R^2_{(\ppa,\secondoparametroastratto)}$ be defined as follows:
\begin{equation*}\label{eq:dommap}
\dommap :=
\left\{(\ppa, \secondoparametroastratto) : ~ \ppa \in[a,b], |\secondoparametroastratto| 
\leq 
 \sigma(t) \right\},
\end{equation*}
which has the same qualitative properties of $K$, in particular
$\partial \dommap= {\rm graph}(\sigma) \cup {\rm graph}(-\sigma)$,
and 
$\dommap$ has two angles 
in correspondence of $t=a$ and $t=b$ (same angles as the 
corresponding ones
of $K$). We notice that the segment $(a,b)\times \{0\}$ 
is contained in ${\rm int}(\dommap)$, 
 see Figure \ref{fig:Lambda}.

\smallskip

{\tt Step 5}. Definition of the maps $X$ and $X_\mu$.

The construction of the function $\mappa_\eps$ in the statement of the
theorem is mainly based on the maps
$$
X: \dommap \to \R^3, \qquad
X_\mu: \dommap \to \R^3,
$$
defined as follows: for any $(\ppa,\spa)\in \dommap$
\begin{equation}\label{eq:par_graph}
\begin{aligned}
X(\ppa, \secondoparametroastratto) 
:= &  \left(\ppa, \secondoparametroastratto + 
\frac{\gamma^+_1(\ppa)+
\gamma^-_1(\ppa)}{2}, z\left(\ppa, \secondoparametroastratto + 
\frac{\gamma^+_1(\ppa)+
\gamma^-_1(\ppa)}{2}\right)\right)
\\ = &(t,X_2(t,s), X_3(t,s)),
\\
\\
X_\mu(\ppa, \secondoparametroastratto) 
:= & \left(\ppa, \secondoparametroastratto + 
\frac{\gamma^+_1(\ppa)+
\gamma^-_1(\ppa)}{2}, \widehat z_\e \left(\ppa, \secondoparametroastratto + 
\frac{\gamma^+_1(\ppa)+
\gamma^-_1(\ppa)}{2}\right)\right)
\\
= & (t,X_{\e 2}(t,s), X_{\e 3}(t,s)).
\end{aligned}
\end{equation}

\begin{Remark}\label{rem:semi}\rm
We stress that  the maps $X$ and $X_\mu$
are semicartesian.
In particular, where they are differentiable,
their gradient
never vanishes on $D$. Observe also that, from Remark \ref{rem:llipbis}, 
it follows
\begin{equation}\label{eq:Xmulip}
X_\mu \in {\rm Lip}(D; \R^3).
\end{equation}
\end{Remark}

\smallskip
{\tt Step 6}. Definition of the map $\mappa_\eps$.

For the definition of $\mappa_\eps$ we need some preparation.
Denote by 
$^{\perp}$ the counterclockwise rotation of $\pi/2$ in $\source$.
Hypothesis $({\rm u}1)$ implies that there exists
$\delta>0$ and a closed 
set contained in $\Omega$ and containing $\osaltou$ of the form 
$\Lambda(\rettangolo)$, where
$\rettangolo:= \ab \times [-\delta,\delta]$ and 
$$
\Lambda(\primoparametroastratto,\secondoparametroastratto) 
:= \alpha(\primoparametroastratto) + 
\secondoparametroastratto \dot \alpha(\primoparametroastratto)^\perp,
\qquad (\primoparametroastratto,\secondoparametroastratto) \in \rettangolo,
$$
is a diffeomorphism of class $\mathcal C^1(\rettangolo;\Lambda(R))$, see Figure \ref{fig:Lambda}.
If  $\Lambda^{-1}: \Lambda(R) \to \rettangolo$ 
is the inverse
of $\Lambda$, we have   
$$
\Lambda^{-1}(x,y) = (\primoparametroastratto(x,y), \secondoparametroastratto(x,y)),
$$
where
\begin{itemize}
\item[-]
$\secondoparametroastratto
(x,y) = d(x,y)$ is the 
distance 
of $(x,y)$ from $\saltou$ on the side of $\saltou$ corresponding to the trace $\mappa^+$,
and minus the distance of $(x,y)$ from $\saltou$ on the other side, 
\item[-]
$\primoparametroastratto(x,y)$ is so that 
$\alpha(\primoparametroastratto(x,y)) = (x,y) - d(x,y) \grad d(x,y)$ 
is the unique point on $\osaltou$ nearest to 
$(x,y)$.
\end{itemize}

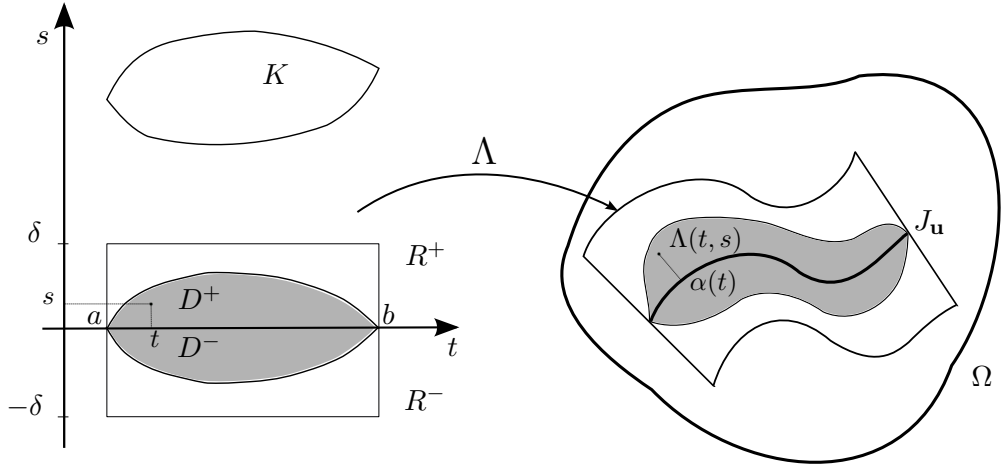
\begin{figure}
\centering
\def\svgwidth{14cm}
%% Creator: Inkscape inkscape 0.48.4, www.inkscape.org
%% PDF/EPS/PS + LaTeX output extension by Johan Engelen, 2010
%% Accompanies image file 'domain.pdf' (pdf, eps, ps)
%%
%% To include the image in your LaTeX document, write
%%   \input{<filename>.pdf_tex}
%%  instead of
%%   \includegraphics{<filename>.pdf}
%% To scale the image, write
%%   \def\svgwidth{<desired width>}
%%   \input{<filename>.pdf_tex}
%%  instead of
%%   \includegraphics[width=<desired width>]{<filename>.pdf}
%%
%% Images with a different path to the parent latex file can
%% be accessed with the `import' package (which may need to be
%% installed) using
%%   \usepackage{import}
%% in the preamble, and then including the image with
%%   \import{<path to file>}{<filename>.pdf_tex}
%% Alternatively, one can specify
%%   \graphicspath{{<path to file>/}}
%% 
%% For more information, please see info/svg-inkscape on CTAN:
%%   http://tug.ctan.org/tex-archive/info/svg-inkscape
%%
\begingroup%
  \makeatletter%
  \providecommand\color[2][]{%
    \errmessage{(Inkscape) Color is used for the text in Inkscape, but the package 'color.sty' is not loaded}%
    \renewcommand\color[2][]{}%
  }%
  \providecommand\transparent[1]{%
    \errmessage{(Inkscape) Transparency is used (non-zero) for the text in Inkscape, but the package 'transparent.sty' is not loaded}%
    \renewcommand\transparent[1]{}%
  }%
  \providecommand\rotatebox[2]{#2}%
  \ifx\svgwidth\undefined%
    \setlength{\unitlength}{514.08510927bp}%
    \ifx\svgscale\undefined%
      \relax%
    \else%
      \setlength{\unitlength}{\unitlength * \real{\svgscale}}%
    \fi%
  \else%
    \setlength{\unitlength}{\svgwidth}%
  \fi%
  \global\let\svgwidth\undefined%
  \global\let\svgscale\undefined%
  \makeatother%
  \begin{picture}(1,0.4379813)%
    \put(0,0){\includegraphics[width=\unitlength]{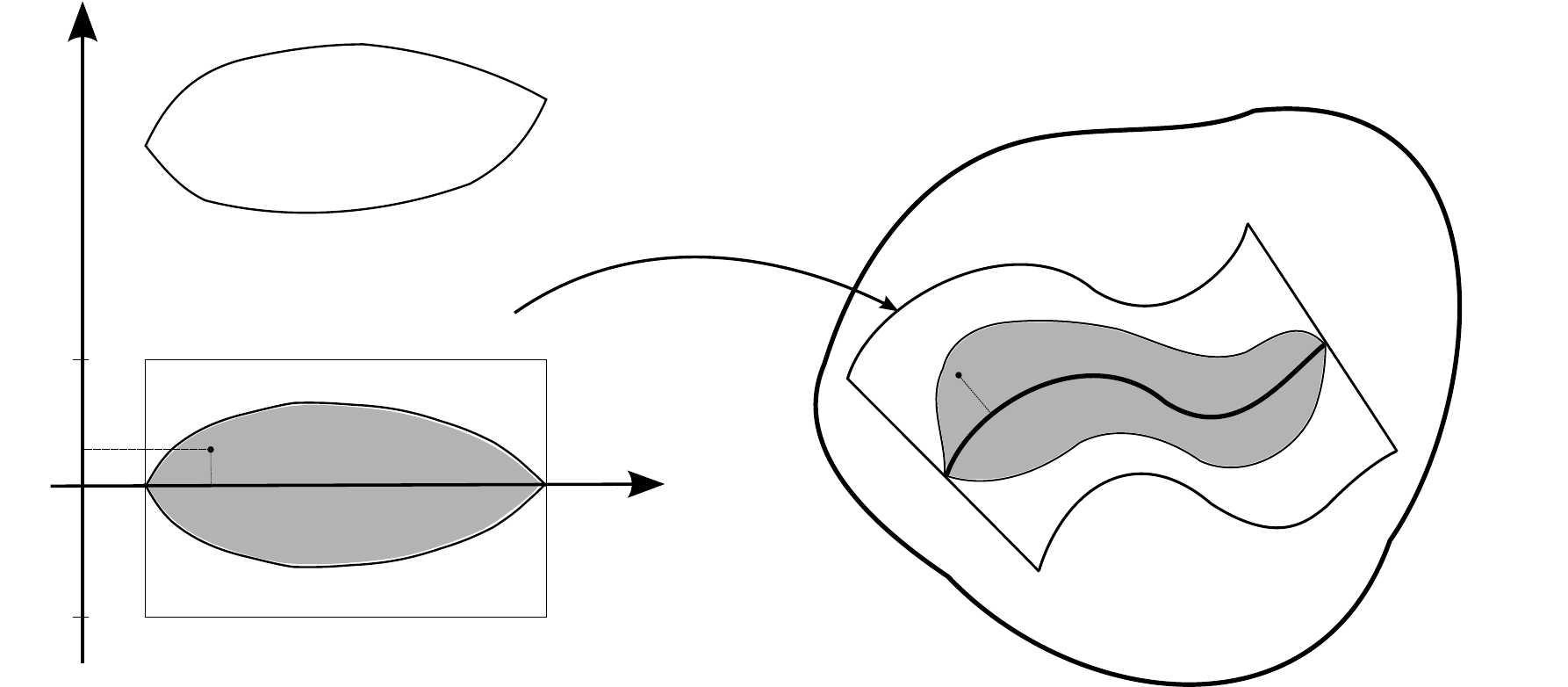}}%
    \put(0.23887047,0.35924657){\color[rgb]{0,0,0}\makebox(0,0)[lb]{\smash{$K$}}}%
    \put(0.15958227,0.14648093){\color[rgb]{0,0,0}\makebox(0,0)[lb]{\smash{$\dommap^+$}}}%
    \put(0.15903822,0.10090865){\color[rgb]{0,0,0}\makebox(0,0)[lb]{\smash{$\dommap^-$}}}%
    \put(0.37318419,0.18794456){\color[rgb]{0,0,0}\makebox(0,0)[lb]{\smash{$\rettangolo^+$}}}%
    \put(0.3723709,0.05133656){\color[rgb]{0,0,0}\makebox(0,0)[lb]{\smash{$\rettangolo^-$}}}%
    \put(0.07400233,0.13282477){\color[rgb]{0,0,0}\makebox(0,0)[lb]{\smash{$a$}}}%
    \put(0.35163641,0.13225329){\color[rgb]{0,0,0}\makebox(0,0)[lb]{\smash{$b$}}}%
    \put(0.01825545,0.21050695){\color[rgb]{0,0,0}\makebox(0,0)[lb]{\smash{$\delta$}}}%
    \put(-0.00140022,0.0467114){\color[rgb]{0,0,0}\makebox(0,0)[lb]{\smash{$-\delta$}}}%
    \put(0.43543072,0.28372786){\color[rgb]{0,0,0}\makebox(0,0)[lb]{\smash{\Large{$\Lambda$}}}}%
    \put(0.90539181,0.0736459){\color[rgb]{0,0,0}\makebox(0,0)[lb]{\smash{$\Om$}}}%
    \put(0.85092613,0.22148136){\color[rgb]{0,0,0}\makebox(0,0)[lb]{\smash{$J_\mappa$}}}%
    \put(0.62136117,0.20574659){\color[rgb]{0,0,0}\makebox(0,0)[lb]{\smash{\begin{small}$\Lambda(\ppa,\spa)$\end{small}}}}%
    \put(0.03207069,0.151806){\color[rgb]{0,0,0}\makebox(0,0)[lb]{\smash{\begin{small}$s$\end{small}}}}%
    \put(0.02603702,0.39791384){\color[rgb]{0,0,0}\makebox(0,0)[lb]{\smash{$\spa$}}}%
    \put(0.64087361,0.16489621){\color[rgb]{0,0,0}\makebox(0,0)[lb]{\smash{\begin{small}$\alpha(\ppa)$\end{small}}}}%
    \put(0.13364753,0.1104412){\color[rgb]{0,0,0}\makebox(0,0)[lb]{\smash{\begin{small}$t$\end{small}}}}%
    \put(0.41308405,0.102002){\color[rgb]{0,0,0}\makebox(0,0)[lb]{\smash{$t$}}}%
  \end{picture}%
\endgroup%

\caption{\small{ 
We display the domain $\dommap=\overline{\dommap^+\cup\dommap^-}$ obtained by 
symmetrizing $K$. It is contained in the rectangle
$R=\ab\times[-\delta,\delta]$ on which it is defined the diffeomorphism $\Lambda$; $\Lambda(\ab\times\{0\})$ is exactly
the closure $\osalto$ of the discontinuity curve.}}\label{fig:Lambda}
\end{figure}

Since $\osaltou$ is of class $\C^2$, we have that 
$d$ is of class $\C^2$ on $\Lambda(R)$\footnote{It is sufficient to slightly extend
$\osaltou$ and consider $d$ on a small enough tubolar neighborhood of the extension.}
 and 
$\primoparametroastratto$ is of class $\C^1$ on $\Lambda(R)$. 

We can always suppose 
\begin{equation}\label{eq:incD}
\dommap \setminus \left((a,0)\cup(b,0)\right) \subset {\rm int}(\rettangolo),
\end{equation}
since,  if not, we choose   $c \in(0,1)$
so that
$\dommap_c:=\{(\primoparametroastratto,
\secondoparametroastratto)\in\R^2:\,\,(\primoparametroastratto,
\secondoparametroastratto/ c )\in \dommap\} \subset \rettangolo$,
and we prove the result with $\dommap _c$ in place of $\dommap$
and $X_c(t,s):=X(t,s/ c)$ in place of $X(t,s)$.

Set  
$R^+ := \ab\times (0,\delta]$, $R^- := \ab \times [-\delta,0)$, and
$$
\dommap^+ := \dommap \cap R^+, \qquad 
\dommap^- := \dommap \cap R^-.
$$
For any $\eps \in (0,1)$
let 
\begin{displaymath}
	\Ce 
:=\{(\primoparametroastratto,
\secondoparametroastratto)\in\R^2:\,\,(\primoparametroastratto,
\secondoparametroastratto/\eps)\in \dommap\},
\end{displaymath}
and 
$$
\Ce^\pm 
:=\{(\primoparametroastratto,
\secondoparametroastratto)\in\R^2:\,\,(\primoparametroastratto,
\secondoparametroastratto/\eps)\in \dommap^\pm\},
$$
so that 
$$
{\rm int}(\Ce) \supset (a,b) \times \{0\}.
$$
We set  $R_\eps := \ab \times (-\eps \delta, \eps \delta)$ and
$R_\eps^+ := \ab \times (0,\eps \delta]$, 
$R_\eps^- := \ab \times [-\eps\delta, 0)$. 
{}From \eqref{eq:incD}, we have
$\Ce \subset \rettangolo_\eps$.

\smallskip
We are now in a position to 
define the sequence $(\ue)\subset {\rm Lip}(\Omega; \R^2)$. We do this  
in three steps as follows:
\begin{itemize}
\item[-] \textbf{outer region.} If $(x,y)\in \Om \setminus \Lambda(\rettangolo_\eps)$
\begin{equation}
\label{eq:ue1}
\ue(x,y) := \mappa(x,y);
\end{equation}
\item[-] \textbf{opening the fracture: intermediate region.}
If $(x,y)\in \Lambda(\rettangolo_\eps^\pm \setminus \Ce^\pm)$
\begin{equation}
\label{eq:ue2}
\ue(x,y):=\mappa(T_\eps^\pm(x,y)),
\end{equation}
where $\Te^\pm:= \Lambda \circ\Phi_\eps^\pm \circ \Lambda^{-1}$ with
\begin{equation*}
\begin{aligned}
\Phi_\eps^+
: \rettangolo_\eps^+ \setminus \Ce^+ &\to \rettangolo_\eps^+, \qquad  
\Phi_\eps^+
(t,s):= \left(t, \frac{s- \eps \sigma(t)}{\delta - \sigma(t)}\delta\right),
\\
\Phi_\eps^-
: \rettangolo_\eps^- \setminus \Ce^- &\to\rettangolo_\eps^-,
\qquad
\Phi_\eps^-
(t,s):= \left(t, \frac{s + \eps \sigma(t)}{\delta - \sigma(t)}\delta\right).
\end{aligned}
\end{equation*}
Notice that $T^\pm_\eps$ is the 
identity on $\partial \rettangolo^\pm_\eps 
\setminus (\ab \times \{0\})$, see Figure \ref{fig:mappaT}.
\item[-] \textbf{opening the fracture: inner region.} If $(x,y)\in \Lambda(\Ce)$
\begin{equation}
\label{eq:ue3}
\ue(x,y):=\left(X_{\e_\eps 2}\left(t(x,y), \frac{d(x,y)}{\eps}\right), X_{\e_\eps 3}\left(t(x,y), \frac{d(x,y)}{\eps}\right)\right),
\end{equation}
for a suitable choice of  the sequence $(\e_\eps)_\eps$ 
converging to $0$ as $\eps\to 0^+$, that 
will be selected later\footnote{See the conclusion 
of {\tt step 9}.}.
\end{itemize}

\begin{figure}[h!]
\centering
\def\svgwidth{9cm}
%% Creator: Inkscape inkscape 0.48.4, www.inkscape.org
%% PDF/EPS/PS + LaTeX output extension by Johan Engelen, 2010
%% Accompanies image file 'mappat.pdf' (pdf, eps, ps)
%%
%% To include the image in your LaTeX document, write
%%   \input{<filename>.pdf_tex}
%%  instead of
%%   \includegraphics{<filename>.pdf}
%% To scale the image, write
%%   \def\svgwidth{<desired width>}
%%   \input{<filename>.pdf_tex}
%%  instead of
%%   \includegraphics[width=<desired width>]{<filename>.pdf}
%%
%% Images with a different path to the parent latex file can
%% be accessed with the `import' package (which may need to be
%% installed) using
%%   \usepackage{import}
%% in the preamble, and then including the image with
%%   \import{<path to file>}{<filename>.pdf_tex}
%% Alternatively, one can specify
%%   \graphicspath{{<path to file>/}}
%% 
%% For more information, please see info/svg-inkscape on CTAN:
%%   http://tug.ctan.org/tex-archive/info/svg-inkscape
%%
\begingroup%
  \makeatletter%
  \providecommand\color[2][]{%
    \errmessage{(Inkscape) Color is used for the text in Inkscape, but the package 'color.sty' is not loaded}%
    \renewcommand\color[2][]{}%
  }%
  \providecommand\transparent[1]{%
    \errmessage{(Inkscape) Transparency is used (non-zero) for the text in Inkscape, but the package 'transparent.sty' is not loaded}%
    \renewcommand\transparent[1]{}%
  }%
  \providecommand\rotatebox[2]{#2}%
  \ifx\svgwidth\undefined%
    \setlength{\unitlength}{538.75276133bp}%
    \ifx\svgscale\undefined%
      \relax%
    \else%
      \setlength{\unitlength}{\unitlength * \real{\svgscale}}%
    \fi%
  \else%
    \setlength{\unitlength}{\svgwidth}%
  \fi%
  \global\let\svgwidth\undefined%
  \global\let\svgscale\undefined%
  \makeatother%
  \begin{picture}(1,0.62760699)%
    \put(0,0){\includegraphics[width=\unitlength]{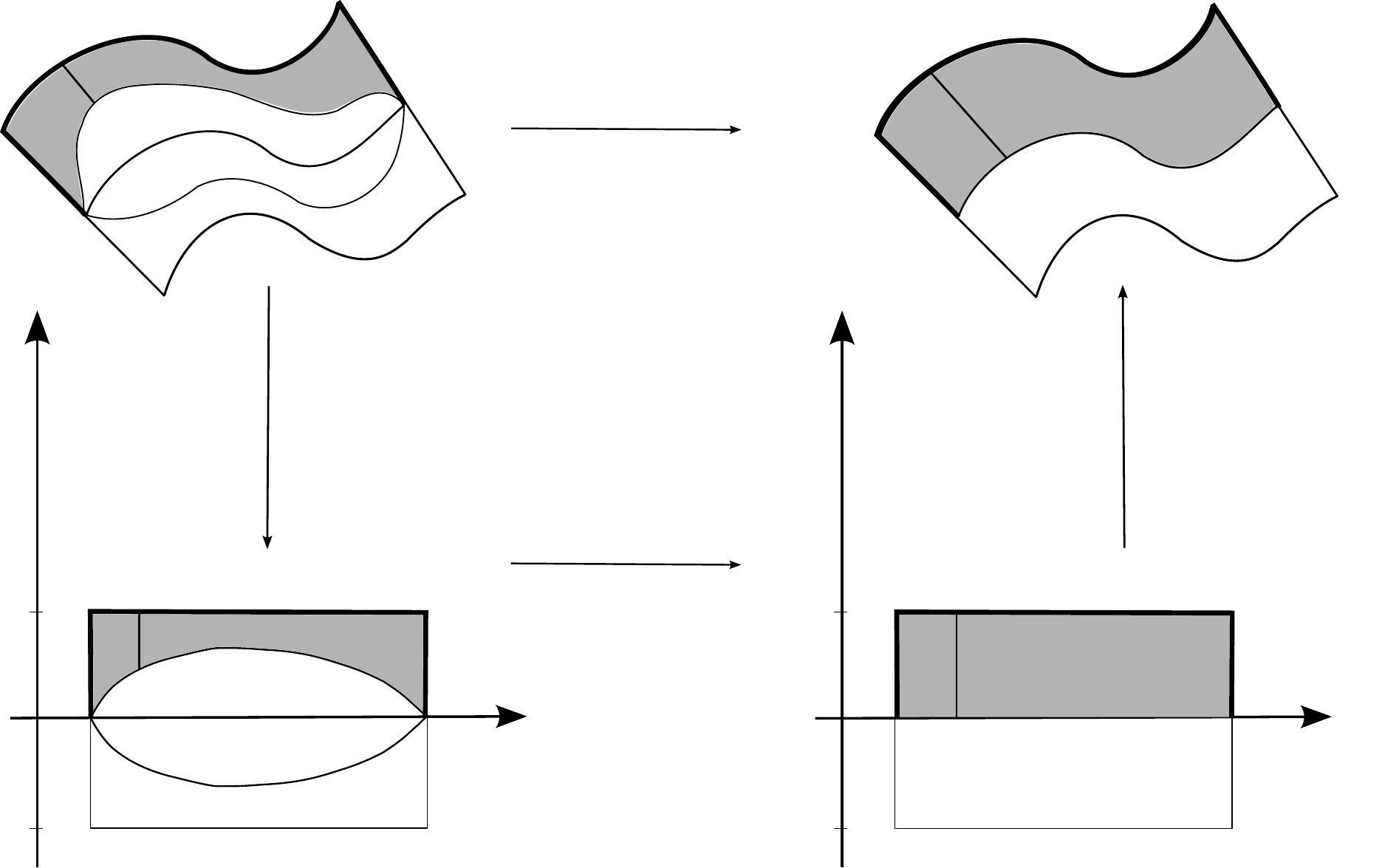}}%
    \put(0.20202561,0.34912234){\color[rgb]{0,0,0}\makebox(0,0)[lb]{\smash{$\Lambda^{-1}$}}}%
    \put(0.83119354,0.35039297){\color[rgb]{0,0,0}\makebox(0,0)[lb]{\smash{$\Lambda$}}}%
    \put(0.43805702,0.54792592){\color[rgb]{0,0,0}\makebox(0,0)[lb]{\smash{$T^+_\eps$}}}%
    \put(0.31810726,0.15091947){\color[rgb]{0,0,0}\makebox(0,0)[lb]{\smash{$\rettangolo^+_\eps$}}}%
    \put(0.31834701,0.03770614){\color[rgb]{0,0,0}\makebox(0,0)[lb]{\smash{$\rettangolo^-\eps$}}}%
    \put(0.12149015,0.0699504){\color[rgb]{0,0,0}\makebox(0,0)[lb]{\smash{$\dommap^-_\eps$}}}%
    \put(0.12076737,0.12287708){\color[rgb]{0,0,0}\makebox(0,0)[lb]{\smash{$\dommap^+_\eps$}}}%
    \put(0.9005046,0.15122218){\color[rgb]{0,0,0}\makebox(0,0)[lb]{\smash{$\rettangolo^+_\eps$}}}%
    \put(0.90074445,0.03800886){\color[rgb]{0,0,0}\makebox(0,0)[lb]{\smash{$\rettangolo^-_\eps$}}}%
    \put(0.43781497,0.23309896){\color[rgb]{0,0,0}\makebox(0,0)[lb]{\smash{$\Phi^+_\eps$}}}%
  \end{picture}%
\endgroup%

\caption{\small{The action of the map $T^+_\eps$. Any oblique small segment on the top left
is mapped in the parallel longer segment reaching the fracture, on the top right.
}}
\label{fig:mappaT}
\end{figure}
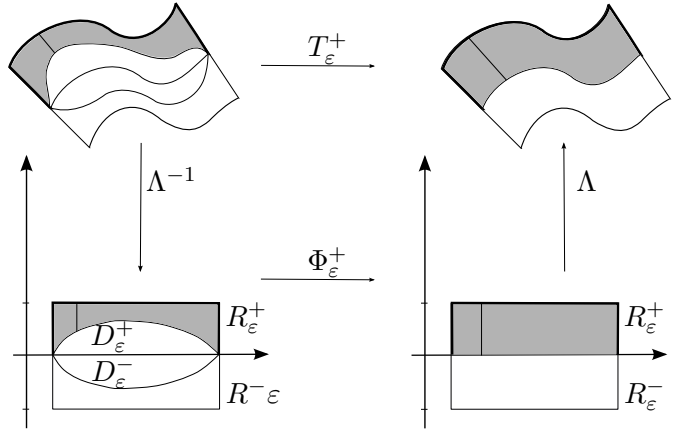

\begin{Remark}\label{rem:lipue}\rm
We have
\begin{equation}\label{ueuelip}
\ue \in \rm{Lip} 
(\Omega;\R^2). 
\end{equation}
Indeed 
\begin{itemize}
\item[-] by  assumption $({\rm u}2)$ it follows  $\mappa \in 
W^{1,\infty}(\Omega\setminus \Lambda(\Reps);\R^2)$, hence
$\mappa_\eps 
\in W^{1,\infty}(\Omega\setminus \Lambda(\Reps);\R^2)$;
\item[-] in $\Lambda(\Ce)$
 the regularity of $\ue$ 
is the same as the Lipschitz regularity of $X_{\e_\eps}$, see \eqref{eq:Xmulip};
\item[-]
in $\Lambda(\rettangolo_\eps\setminus \Ce)$, $\ue$ 
is defined as the composition of $\mappa \in W^{1,\infty}\left(\Om\setminus
\osaltou; \R^2\right)$
and a Lipschitz  deformation.
\end{itemize}
Since by construction $\mappa_\eps$ is continuous (remember \eqref{eq:tracciazhatz}), inclusion 
\eqref{ueuelip} follows.
\end{Remark}

\begin{Remark}\label{rem:luc}\rm
We have $$\sup_{\eps \in (0,1]} \Vert \ue\Vert_{L^\infty(\Om; \R^2)} < +\infty,$$
since $\mappa \in L^\infty(\Omega; \R^2)$ by assumption $(\rm u 1)$ and, for some $\overline \mu>0$,
$\sup_{\mu \in (0,\overline \mu)} 
\Vert X_\mu\Vert_{L^\infty(D)} < +\infty$.
Therefore
$\ue \to \mappa$ in $L^1(\Om ; \R^2)$. Indeed
\begin{displaymath}
\int_\Om |\ue - \mappa | \,dx\,dy = \int_{\Lambda(\rettangolo_\eps)} |\ue -\mappa | \,dx\,dy 
\to 0
\end{displaymath} 
as $\eps \to 0^+$, because the Lebesgue measure of $ \Lambda(R_\eps)$ tends to $0$.
\end{Remark}
 
\smallskip

{\tt Step 7}. We have
\begin{equation*}\label{eq:out_limit}
\lim_{\eps \to 0^+}
\rel(\ue, \Omega \setminus \Lambda(\rettangolo_\eps))
= \rel(\mappa, \Omega \setminus \saltou).
\end{equation*}
Indeed by \eqref{eq:ue1},
$$
\rel(\ue, \Om\setminus\Lambda(\rettangolo_\eps))=\rel(\mappa, \Om\setminus\Lambda(\rettangolo_\eps)).
$$ 
\medskip

Let us show that the 
contribution to the area in  the intermediate region
$\Lambda(\rettangolo_\eps \setminus \Ce)$ (definition \eqref{eq:ue2}) is 
negligible as $\eps \to 0^+$.

{\tt Step 8}. We have
$$
\lim_{\eps \to 0^+}
\rel(\ue, \Lambda(\rettangolo_\eps^\pm \setminus \Ce^\pm))
 = 0.
$$
We make the computation in 
$\Lambda(R_\eps^+ \setminus D_\eps^+)$,
the case in 
$\Lambda(R_\eps^- \setminus D_\eps^-)$ being similar.
To simplify notation,
 we write $\Te$ instead of $\Te^+$, and set $\Te = (\Teu,\Ted)$.

Take a constant $C>0$ so that
\begin{equation}\label{eq:estimate_area}
\begin{aligned}
& \A (\ue, \Lambda(\rettangolo_\eps^+\setminus \Ce^+))
= \int_{\Lambda(\rettangolo_\eps^+\setminus \Ce^+)}|\M(\grad \ue)|\,dx\,dy
\\ & \leq  C \int_{\Lambda(\rettangolo_\eps^+\setminus \Ce^+)}\big[ 1 + |\dx \ueu | +|\dx \ued | + |\dy \ueu | + | \dy \ued | + |\dx \ueu \dy \ued - \dy \ueu \dx \ued | \big]dx\,dy
\end{aligned}
\end{equation}
where for $i=1,2$ 
\begin{equation*}
\dx \uei = \dx \mappai \dx \Teu + \dy \mappai \dx \Ted, \qquad
\dy \uei = \dx \mappai \dy \Teu + \dy \mappai \dy \Ted.
\end{equation*}

From the definition of $\Te$ 
\begin{displaymath}
\grad\Te(x,y) = 
\grad \Lambda\left(\Phi_\eps(t(x,y),s(x,y))\right)^T  \grad \Phi_\eps^+ 
(\primoparametroastratto (x,y), \secondoparametroastratto (x,y)) \cdot \grad \Lambda^{-1}(x,y).
\end{displaymath}
$\Lambda$ is a $\C^1$ diffeomorphism, thus all 
components of its Jacobian are bounded; 
on the other hand the Jacobian of the transformation $\Phi_\eps^+$ is
\begin{displaymath}
\grad \Phi_\eps^+(\primoparametroastratto, \secondoparametroastratto) = 
\begin{bmatrix} \displaystyle{1} & \displaystyle{0} 
\\ \displaystyle{\frac{-\delta \dot \sigma(t)[\delta \eps - s]}{[\delta - \sigma(t)]^2}} & \displaystyle{\frac{\delta}{\delta- \sigma(t)}}
\end{bmatrix}.
\end{displaymath}
The denominator $\left(\delta - \sigma(t)\right)$ is strictly positive; 
moreover $\sigma \in {\rm Lip}(\ab)$ and thus all
terms of $\grad \Phi_\eps^+$ are 
uniformly bounded with respect to $\eps$. 

Then, since both $\grad \Te$ and $\grad \mappa$ are bounded, we obtain 
that also the integrand on 
the right hand side of  \eqref{eq:estimate_area} can be controlled by a constant independent of $\eps$ and
\begin{equation*}
\lim_{\eps \to 0^+} \A (\ue, \Lambda(\rettangolo_\eps^+\setminus \Ce^+))= 0 .
\end{equation*}

The main point is to show that the definition given in \eqref{eq:ue3} 
in the relevant region $\Lambda(\Ce)$ is such that the 
corresponding area gives origin to the term $\H^2(\immmapmin)$
in the limit $\eps \to 0^+$, and it is done in the next step.

%The most interesting computation comes in the next step. 
\smallskip

{\tt Step 9}. We have
\begin{equation}\label{eq:cruc}
\lim_{\eps \to 0^+}
\rel(\ue, \Lambda(\Ce)) = 
\H^2(\immmapmin).
\end{equation}
Let us fix $\e>0$; we denote by $\ue^\e$ the function defined on $\Lambda(\Ce)$ as
$$
\ue^\e(x,y):=(\Xed(\primoparametroastratto(x,y), d(x,y)/\eps), \Xet(\primoparametroastratto(x,y), d(x,y)/\eps))
= (\ueu^\e(x,y), \ued^\e(x,y)).
$$
In $\Lambda({\rm int}(\Ce))$ we have
$$
\grad \ueu^\e = 
\begin{pmatrix} 
\displaystyle
\partial_\primoparametroastratto \Xed ~ \partial_x \primoparametroastratto + \frac{1}{\eps} ~\partial_\secondoparametroastratto
\Xed~ \partial_x d
\\
\\ 
\displaystyle
\partial_\primoparametroastratto \Xed~ \partial_y \primoparametroastratto + \frac{1}{\eps} ~\partial_\secondoparametroastratto
\Xed~ \partial_y d
\end{pmatrix},
\qquad
\grad \ued^\e = 
\begin{pmatrix}  
\displaystyle
\partial_\primoparametroastratto \Xet ~\partial_x \primoparametroastratto + \frac{1}{\eps}~\partial_\secondoparametroastratto
\Xet ~\partial_x d
\\
\\
\displaystyle
\partial_\primoparametroastratto \Xet~ \partial_y \primoparametroastratto + \frac{1}{\eps}~ \partial_\secondoparametroastratto
\Xet ~\partial_y d
\end{pmatrix},
$$
where the left hand sides and $t$ and $d$ are evaluated at $(x,y)$, while
$\Xed$ and $\Xet$  are evaluated at $(\primoparametroastratto(x,y), d(x,y)/\eps)$.

Therefore 
\begin{equation}\label{eq:gradienti}
 \big\vert \grad \ueu^\e\big\vert^2 + 
\big\vert \grad \ued^\e\big\vert^2  
\\
= \frac{1}{\eps^2}
G_1  
 + \frac{2}{\eps} G_2 + G_3 \qquad {\rm in~} \Lambda({\rm int}(\Ce)), 
\end{equation}
with
$$
\left\{
\begin{aligned}
G_1   :=& 
\Big(
(\partial_\secondoparametroastratto \Xed)^2 
+(\partial_\secondoparametroastratto \Xet)^2 \Big) 
~\vert \grad d\vert^2 = (\partial_\secondoparametroastratto \Xed)^2 
+(\partial_\secondoparametroastratto \Xet)^2,
\\
 G_2 :=& 
\Big(
\partial_\primoparametroastratto \Xed 
\partial_\secondoparametroastratto
\Xed 
+\partial_\primoparametroastratto \Xet
\partial_\secondoparametroastratto \Xet \Big)
~\grad \primoparametroastratto \cdot \grad d,
\\
G_3 :=&  
\Big(
(\partial_\primoparametroastratto \Xed)^2
+(\partial_\primoparametroastratto \Xet)^2\Big)
~\vert \grad\primoparametroastratto\vert^2,
\end{aligned}
\right.
$$
where we have used the eikonal equation for the signed distance function
$$
\vert \grad d\vert^2=1 \qquad {\rm in}~ {\rm int}(\Lambda(R)).
$$
Notice that $\vert\grad \primoparametroastratto\vert^2$
is uniformly bounded 
with respect to $\eps$ on $\Ce$, by the 
assumption that $\osaltou$ is of class $\C^2$.

A direct computation shows that 
\begin{equation}\label{eq:direct}
\partial_x \ueu^\e \partial_y \ued^\e - \partial_x \ued^\e
\partial_y \ueu^\e = \frac{1}{\eps^2} E_1 + \frac{1}{\eps} \widetilde E_2
+E_3,
\end{equation}
with
\begin{equation}\label{eq:Di}
\left\{
\begin{aligned}
E_1 := &
\partial_\secondoparametroastratto \Xed \partial_x d 
\partial_\secondoparametroastratto \Xet \partial_y d - 
\partial_\secondoparametroastratto \Xed \partial_y d 
\partial_\secondoparametroastratto \Xet \partial_x d =0, 
\\
\widetilde E_2 := &
\partial_\primoparametroastratto \Xed 
\partial_\secondoparametroastratto \Xet 
\Big(
\partial_x \primoparametroastratto 
\partial_y d  - 
\partial_y \primoparametroastratto 
\partial_x d \Big)
+
\partial_\primoparametroastratto \Xet 
\partial_\secondoparametroastratto \Xed 
\Big(
\partial_x d 
\partial_y \primoparametroastratto 
-
\partial_y d 
\partial_x \primoparametroastratto \Big) 
\\
=&  
\Big(\partial_\primoparametroastratto \Xed 
\partial_\secondoparametroastratto \Xet 
-
\partial_\primoparametroastratto \Xet
\partial_\secondoparametroastratto \Xed\Big)
 \grad \primoparametroastratto \cdot \grad d^\perp,
\\
E_3:=& 
\partial_\primoparametroastratto \Xed 
\partial_x \primoparametroastratto 
\partial_\primoparametroastratto \Xet
\partial_y \primoparametroastratto 
-
\partial_\primoparametroastratto \Xed 
\partial_y \primoparametroastratto 
\partial_\primoparametroastratto \Xet
\partial_x \primoparametroastratto =0,
\end{aligned}
\right.
\end{equation}
Set
\begin{equation}\label{eq:pa}
E_2 := 
\partial_\primoparametroastratto \Xed 
\partial_\secondoparametroastratto \Xet 
-
\partial_\primoparametroastratto \Xet
\partial_\secondoparametroastratto \Xed.
\end{equation}
{}From \eqref{eq:direct}, \eqref{eq:Di} and \eqref{eq:pa}
we have
\begin{equation}\label{eq:determinante}
\big(\partial_x \ueu^\e \partial_y \ued^\e - \partial_x \ued^\e
\partial_y \ueu^\e\big)^2 = \frac{1}{\eps^2} (E_2)^2 ~\vert \grad \primoparametroastratto\cdot \grad
d^\perp\vert^2,
\end{equation}
where again 
$\Xed$ and $\Xet$  are evaluated at $(\primoparametroastratto(x,y), d(x,y)/\eps)$.

Notice that if $(x,y) \in \Lambda(\Ce)$ then the 
 vector $\grad d^\perp(x,y) = \grad d^\perp (\pi(x,y))$ is tangent
to $\saltou$ at $\pi(x,y)$, 
and has unit length. In addition, $\primoparametroastratto$ is constant
along the normal direction to $\saltou$, so that 
if $(x,y) \in \Lambda(\Ce)$ then 
$\grad \primoparametroastratto(x,y) = \grad \primoparametroastratto(\pi(x,y)) + \mathcal O(\eps)$, 
and $\grad \primoparametroastratto(\pi(x,y))$ is also tangent 
to $\saltou$, where
$$
\vert \mathcal O(\eps)\vert \leq 
c \Vert\kappa \Vert_{L^\infty(\saltou)} \max_{t\in \ab} (\sigma^+(t)-\sigma^-(t)),
$$
$\kappa$ being the curvature of $\saltou$,
for a positive constant $c$ independent of $\eps$. 

 Since $\alpha$ is an arc-length 
parametrization of $\saltou$, it follows that 
$\vert \grad \primoparametroastratto\vert=1$ on $\saltou$. 
Therefore
\begin{equation}\label{eq:lucia}
\vert \grad \primoparametroastratto \cdot \grad d^\perp\vert =1 + \mathcal
O(\eps)\qquad {\rm on}~ \Lambda(\Ce),
\end{equation}
and hence from \eqref{eq:determinante}
$$
\big(\partial_x \ueu^\e \partial_y \ued^\e - \partial_x \ued^\e
\partial_y \ueu^\e\big)^2 = \frac{1}{\eps^2} (E_2)^2 (1+\mathcal O(\eps)).
$$
Whence, from \eqref{eq:gradienti} and \eqref{eq:determinante},
$$
\begin{aligned}
& \rel(\ue^\e, 
\Lambda(\Ce))
\\
\\
=& \int_{\Lambda(\Ce)}
\sqrt{1 + \vert \grad \ueu^\e \vert^2 + \vert \grad \ued^\e
\vert^2 + 
\left(\partial_x \ueu^\e \partial_y \ued^\e - \partial_x \ued^\e
\partial_y \ueu^\e \right)^2
} ~dx dy
\\
\\
= &
\int_{\Lambda(\Ce)}
\sqrt{1 + G_3 + \frac{2}{\eps}~ G_2+ \frac{1}{\eps^2} \Big[
G_1+ (E_2)^2(1+\mathcal O(\eps))\Big]} ~dxdy.
\end{aligned}
$$
The area formula implies that 
$$
\rel(\ue^\e, \Lambda(\Ce)) = \int_{\Ce} \sqrt{1 + \widehat G_3  +
\frac{2}{\eps} \widehat G_2 + \frac{1}{\eps^2} \left[\widehat G_1 + 
(\widehat E_2)^2(1+\mathcal O(\eps))\right]
} ~ \vert{\rm det}(\Lambda)\vert ~d\primoparametroastratto d\secondoparametroastratto.
$$
Here, for $i=1,2,3$,  
$\widehat G_i$ (respectively $\widehat E_2$) equals
$G_i$ (respectively $E_2$) with 
$(x,y)$ replaced by $\Lambda^{-1}(x,y) = (\primoparametroastratto,
\secondoparametroastratto)$, 
where we have $\alpha(\primoparametroastratto) = \pi(x,y)$ and $\secondoparametroastratto = d(x,y)$;
in particular $\Xed$ and $\Xet$ are evaluated at $(\primoparametroastratto, \secondoparametroastratto/\eps)$.
Remember also that $\vert {\rm det}(\Lambda)\vert =
\vert 1 - \kappa \secondoparametroastratto\vert$, 
$\kappa$ being the curvature of $\saltou$ at
$\alpha(\primoparametroastratto)$.  
Making the change of variables $\secondoparametroastratto/\eps \to \secondoparametroastratto$ we get
$$
\rel(\ue^\e, \Lambda(\Ce)) = \int_{\dommap} \sqrt{\eps^2 + \eps^2 \widehat G_3  +
2\eps \widehat G_2 +  \left[\widehat G_1 + (\widehat E_2)^2(1+\mathcal 
O(\eps))\right]
} ~ \vert 1 - \eps \kappa \secondoparametroastratto\vert 
~d\primoparametroastratto d\secondoparametroastratto, 
$$
where now $\Xed$ and $\Xet$ are evaluated at $(\primoparametroastratto,\secondoparametroastratto)$,
and we notice that the term $\mathcal O(\eps)$ is unaffected
by the variable change.

Hence, by our regularity assumption on $\osaltou$ and \eqref{eq:lucia},
 we deduce
\begin{equation}\label{eq:fonda}
\lim_{\eps \to 0^+}
\rel(\ue^\e, \Lambda(\Ce)) = \int_{\dommap} \sqrt{
\widehat G_1 + (\widehat E_2)^2
} ~ 
~d\primoparametroastratto d\secondoparametroastratto.
\end{equation}
{}From \eqref{eq:par_graph} it follows
\begin{displaymath}
D\mapmin =
\begin{bmatrix} 
		1 					 &  0                 
\\ 
		\partial_\primoparametroastratto \Xed & \partial_\secondoparametroastratto \Xed 
\\ 
		\partial_\primoparametroastratto \Xet & \partial_\secondoparametroastratto \Xet
	\end{bmatrix},
\end{displaymath}
so that using the area formula 
\begin{displaymath}
\begin{aligned}
\H^2(\widehat \Sigma^{\e})
=&\int_D \sqrt{{\rm det}(DX_\e ^T  DX_\e)}~d\primoparametroastratto\,d\secondoparametroastratto
\\
=& \int_D \sqrt{(\partial_\secondoparametroastratto \Xed)^2 +(\partial_\secondoparametroastratto \Xet)^2
+(\partial_t \Xed\partial_\secondoparametroastratto \Xet - \partial_t \Xet \partial_\secondoparametroastratto \Xed)^2}
~dt\,d\secondoparametroastratto,
\end{aligned}
\end{displaymath}
which coincides with the right hand side of \eqref{eq:fonda}. 
This shows that for any $\e \in (0,\overline \mu)$
$$
\lim_{\eps \to 0^+}\rel(\ue^\e, \Lambda(\Ce))=\H^2(\widehat \Sigma^\e).
$$
Recalling \eqref{convemu}, by a diagonalization process we can choose $(\mu_\eps)_\eps$
such that, defining $$\ue(x,y):=\ue^{\mu_\eps}(x,y), \quad \quad  (x,y)\in  \Lambda(\Ce),$$
we get
$$
\lim_{\eps \to 0^+}\rel(\ue, \Lambda(\Ce))=\H^2(\immmapmin).
$$
This concludes the proof of \eqref{eq:cruc} and hence of \eqref{eq:lim_ue}.
Inequality \eqref{limsup} follows by observing that ${\rm Lip}(\Om;\R^2)\subset\DM$ 
(see \eqref{eq:space_D}) and applying 
Lemma \ref{lem:relaxation}.
\qed

\section{Parametric case}\label{sec:general}
In this section we relax the hypotheses of Theorem \ref{teo:graph_main}, 
in order to allow area-minimizing surfaces 
not of graph-type and possibly self-intersecting. 

\subsection{Hypotheses on $\mappa$ and statement for the parametric case}
We consider a map $\mappa=(\mappauno, \mappadue)$ satisfying condition $({\rm u}1)$ and:

\begin{itemize}
\item[(\hp 2)]  \begin{itemize}
  \item[-] $\mappa \in \C^1(\Omega \setminus \osaltou; \R^2)$;  
  \item[-] $\M(\grad \mappa) \in L^1(\Om\setminus \osaltou; \R^6 )$;
  \item[-] for $\zeta>0$ small enough, denoting by $B^a_\zeta$ (respectively $B^b_\zeta$) the open disk 
centered at $\alpha(a)$ (respectively $\alpha(b)$) with radius $\zeta$, we have
 $\mappauno \in W^{1,2}\left(\Om\setminus \overline{J_\mappa \cup B^a_\zeta \cup B^b_\zeta}\right)$;
  \item[-] $\mappadue \in W^{1,2}\left(\Om\setminus \osaltou\right)$.
  \end{itemize}
\item[(\hp 3)] The two traces 
$\gammapm$ belong to $\C\left(\ab;\R^2\right) \cap BV\left(\ab; \R^2\right)$ and \eqref{eq:estremi} and \eqref{eq:interni} hold.
\item[(\hp 4)] $\Gamma=\Gamma[\map]$, defined as in \eqref{eq:Gammau}, 
is such that the image $\immmapmin$ of an area-minimizing 
disk-type solution of the Plateau's problem admits a semicartesian parametrization
with domain $D$ (Definition \ref{def:semicart_par}) satisfying the two following conditions:
\begin{itemize}
\item[-]$(a,b)\times\{0\}\subset {\rm int}(\dommap)$,
\item[-]if $\sigma^+ \notin {\rm Lip}([a,b))$, near the point $(a,0)$
the graph of $\sigma^+$ is of the form $\{(\tau(\spa), \spa)\}$, for $\vert
\spa\vert$ 
small enough,
where
\begin{equation}\label{eq:taus}
\tau(\spa)= a+\alpha_2 \spa^2 + o(\spa^2)
\end{equation}
with $\alpha_2 > 0$. The analogue holds near the point $(b,0)$,
 and similar conditions also for $\sigma^-$.
\end{itemize}

\end{itemize}

Conditions on $\Gamma $ ensuring that $\immmapmin$ admits 
a semicartesian parametrization are given in Section \ref{sec:semicart}.

\begin{Remark}
\textup{Theorem \ref{teo:general_main} remains valid if we 
exchange the hypotheses on the two components of $\mappa$,
that is if we ask $\mappauno \in W^{1,2}\left(\Om\setminus\osaltou\right)$ and 
$\mappadue\in W^{1,2}\left(\Om\setminus \overline{J_\mappa \cup B^a_\zeta \cup B^b_\zeta}\right)$ }.
\end{Remark}

An example of map satisfying $({\rm u}1)$, (\hp 2)-(\hp 4) is given in 
Example \ref{exa:cerchio} below.

\begin{Remark}\label{rem:W12}\textup{
The first two items of 	hypothesis (\hp 2) guarantee that $\mappa \in {\rm D}\left(\Om\setminus\osaltou;\R^2\right)$. 
We observe that any 
$\mappav \in W^{1,2}(\Om \setminus J_{\mappav};\R^2)$ satisfies (\hp 2);
 the converse conclusion is false,
as shown by the map described in Example \ref{exa:cerchio}.  
We also observe that assuming in (\hp 2) the weaker condition
$u_1 \in W^{1,2}_{\rm loc}\left(\Omega \setminus \osaltou\right)$ is not 
enough for our proof to work, since we need $\mappa^\mu_\eps$ to be, 
in the intermediate region, of class $W^{1,2}\left(\Lambda(R_\eps\setminus
D_\eps^\mu); \R^2\right)$, see
the expression in step 2 below.
}
\end{Remark}

\begin{Theorem}\label{teo:general_main}
Suppose that $\mappa$ satisfies assumptions \textup{$({\rm u}1)$, (\hp 2)-(\hp 4)}.
Then there exists a sequence 
\begin{equation*}\label{ueW}
(\ue)_\eps \subset W^{1,2}(\Om;\R^2)
\end{equation*}
converging to $\mappa$ in $L^1(\Omega; \R^2)$ as $\eps \to 0^+$ 
satisfying \eqref{eq:lim_ue}. Moreover \eqref{limsup} holds.
\end{Theorem}
\subsection{Proof of Theorem \ref{teo:general_main}}

As we have already remarked, hypothesis (\hp 2) guarantees that 
$\mappa \in \DMjump$
and hence  the expression $\rel(\mappa, \Om\setminus\osaltou)$ 
in (\ref{eq:lim_ue}) is meaningful.

\smallskip
To prove the theorem, we follow the line of 
reasoning of the proof of
 Theorem \ref{teo:graph_main}. However
 we now have to overcome two different problems. 
 More precisely, 
\begin{itemize}
\item[-] the derivative of $\sigma^\pm$, whose graphs form the boundary of the 
domain $\dommap$,
could be unbounded at $\ppa=a$ and $\ppa=b$; this implies that
also $\vert\grad T_\eps\vert$ could be unbounded;
\item[-] near the crack tips the map $\mappa$ is not regular enough
to guarantee straightforwardly that $\ue$ is sufficiently regular.
\end{itemize}

Also in this case, we split the proof into various steps.
In the first step 
 we construct a family of surfaces $\Sigma_\mu$ approximating $\immmapmin$ and
parametrized on suitable domains $\dommap^\e\subseteq \dommap$
bounded by the graphs of $\sigma_\mu^\pm\in {\rm Lip}([a,b])$.

\medskip 
{\tt step 1}. 
If $\sigma^\pm \in {\rm Lip}(\ab)$, we do not need to approximate $\Sigma$, thus we can pass directly
to the next step  with $X_\mu=X$ and $\dommap^\mu=\dommap$.
Hence we can assume $\sigma^\pm\in {\rm Lip}_{\rm loc}(\ab)\setminus{\rm Lip}(\ab)$.
 
Let us suppose for example that the graph of $\sigma^+$ near the point $(a,0)$ is in the form \eqref{eq:taus} 
(while the derivative of $\sigma^+$ near $b$ is bounded and $\sigma^- \in {\rm Lip}(\ab)$)
\footnote{If also the other derivatives blow up, the construction 
of $X_\mu$ and $\dommap^\mu$ is similar.}, see Figure \ref{fig:domain_mu}. 
Then we  modify the domain $\dommap$
and the map $X$ near the point $(a,0)$ as follows.

For a small positive constant $c$ we adopt the following notation:
\begin{itemize}
 \item[-] $\ell_c$ is the portion of the line 
  over $\ab$
passing through $(a,0)$ with angular coefficient $c^{-1}$,
that is:
  $$
  \ell_c(t)=\frac{t-a}{c}, \quad \ppa \in \ab;
  $$   
 \item[-] $P_c$ is the first intersection point of $\ell_c$ with $\partial \dommap$ and its 
coordinates are denoted by $(t_c, \sigma^+(t_c))$ ($t_c>a$ thanks to our assumption about the graph of $\sigma^+$). 
\end{itemize}

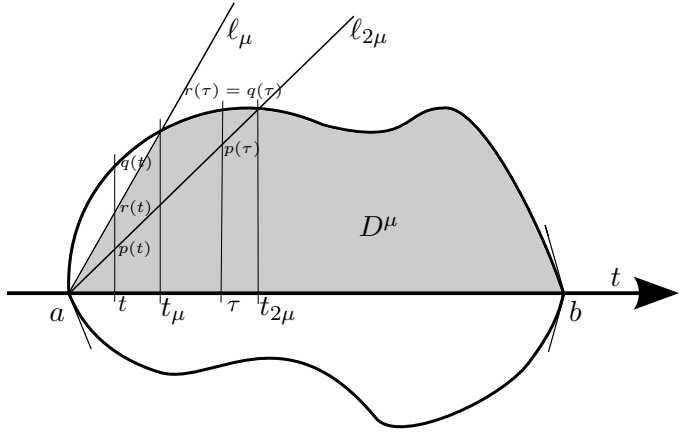
\begin{figure}[h!]
\centering
\def\svgwidth{10cm}
%% Creator: Inkscape inkscape 0.48.4, www.inkscape.org
%% PDF/EPS/PS + LaTeX output extension by Johan Engelen, 2010
%% Accompanies image file 'domain_mu.pdf' (pdf, eps, ps)
%%
%% To include the image in your LaTeX document, write
%%   \input{<filename>.pdf_tex}
%%  instead of
%%   \includegraphics{<filename>.pdf}
%% To scale the image, write
%%   \def\svgwidth{<desired width>}
%%   \input{<filename>.pdf_tex}
%%  instead of
%%   \includegraphics[width=<desired width>]{<filename>.pdf}
%%
%% Images with a different path to the parent latex file can
%% be accessed with the `import' package (which may need to be
%% installed) using
%%   \usepackage{import}
%% in the preamble, and then including the image with
%%   \import{<path to file>}{<filename>.pdf_tex}
%% Alternatively, one can specify
%%   \graphicspath{{<path to file>/}}
%% 
%% For more information, please see info/svg-inkscape on CTAN:
%%   http://tug.ctan.org/tex-archive/info/svg-inkscape
%%
\begingroup%
  \makeatletter%
  \providecommand\color[2][]{%
    \errmessage{(Inkscape) Color is used for the text in Inkscape, but the package 'color.sty' is not loaded}%
    \renewcommand\color[2][]{}%
  }%
  \providecommand\transparent[1]{%
    \errmessage{(Inkscape) Transparency is used (non-zero) for the text in Inkscape, but the package 'transparent.sty' is not loaded}%
    \renewcommand\transparent[1]{}%
  }%
  \providecommand\rotatebox[2]{#2}%
  \ifx\svgwidth\undefined%
    \setlength{\unitlength}{245.70092773bp}%
    \ifx\svgscale\undefined%
      \relax%
    \else%
      \setlength{\unitlength}{\unitlength * \real{\svgscale}}%
    \fi%
  \else%
    \setlength{\unitlength}{\svgwidth}%
  \fi%
  \global\let\svgwidth\undefined%
  \global\let\svgscale\undefined%
  \makeatother%
  \begin{picture}(1,0.56473802)%
    \put(0,0){\includegraphics[width=\unitlength]{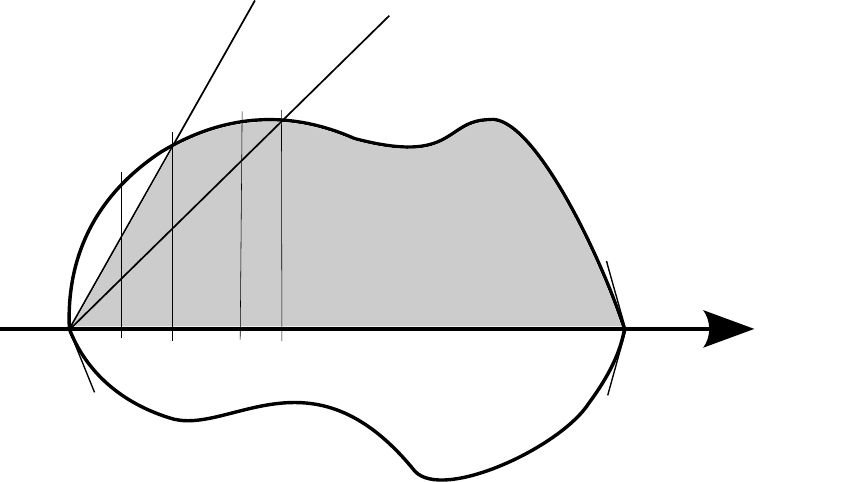}}%
    \put(0.05606413,0.1420096){\color[rgb]{0,0,0}\makebox(0,0)[lb]{\smash{$a$}}}%
    \put(0.20231223,0.14952821){\color[rgb]{0,0,0}\makebox(0,0)[lb]{\smash{$t_\mu$}}}%
    \put(0.79389323,0.18846483){\color[rgb]{0,0,0}\makebox(0,0)[lb]{\smash{$t$}}}%
    \put(0.14639303,0.15555207){\color[rgb]{0,0,0}\makebox(0,0)[lb]{\smash{\begin{footnotesize}$t$
\end{footnotesize}}}}%
    \put(0.33220342,0.14957488){\color[rgb]{0,0,0}\makebox(0,0)[lb]{\smash{$t_{2\mu}$}}}%
    \put(0.28913352,0.5159337){\color[rgb]{0,0,0}\makebox(0,0)[lb]{\smash{$\ell_{\mu}$}}}%
    \put(0.45111622,0.51634082){\color[rgb]{0,0,0}\makebox(0,0)[lb]{\smash{$\ell_{2\mu}$}}}%
    \put(0.46273709,0.25452155){\color[rgb]{0,0,0}\makebox(0,0)[lb]{\smash{$D^\mu$}}}%
    \put(0.13728338,0.23227001){\color[rgb]{0,0,0}\makebox(0,0)[lb]{\smash{\begin{tiny}
$p(t)$
\end{tiny}}}}%
    \put(0.13913345,0.28625229){\color[rgb]{0,0,0}\makebox(0,0)[lb]{\smash{\begin{tiny}
$r(t)$
\end{tiny}}}}%
    \put(0.13986897,0.34685579){\color[rgb]{0,0,0}\makebox(0,0)[lb]{\smash{\begin{tiny}
$q(t)$
\end{tiny}}}}%
    \put(0.28958205,0.15521012){\color[rgb]{0,0,0}\makebox(0,0)[lb]{\smash{\begin{footnotesize}$\tau$
\end{footnotesize}}}}%
    \put(0.27664316,0.36452208){\color[rgb]{0,0,0}\makebox(0,0)[lb]{\smash{\begin{tiny}
$p(\tau)$
\end{tiny}}}}%
    \put(0.2235298,0.44205534){\color[rgb]{0,0,0}\makebox(0,0)[lb]{\smash{\begin{tiny}
$r(\tau)=q(\tau)$
\end{tiny}}}}%
    \put(0.73976882,0.14252446){\color[rgb]{0,0,0}\makebox(0,0)[lb]{\smash{$b$}}}%
  \end{picture}%
\endgroup%

\caption{\small{
Modification of the domain $\dommap$ when the gradient 
of $\sigma^+$ blows up (in this case only near $\ppa=a$).}
}
\label{fig:domain_mu}
\end{figure}

For every  $\mu>0$ small enough, we define 
\begin{displaymath}
\sigma^+_\e:=
\begin{cases}
\ell_\e\, & \mbox{ in } [a, t_\e), \\
\sigma^+ \, & \mbox{ in }  [t_\e, b],
\end{cases}
\end{displaymath}
and
$$
\dommap^\e:=\{(t,s): \, t \in \ab, \, \sigma^-(t)\leq s\leq \sigma^+_\e(t)\}.
$$
In order to define the map $X_\mu$ on $\dommap^\mu$, we need to consider also the line $\ell_{2\mu}$
and the corrisponding intersection point $P_{2\mu}$.
For any $t\in [a,t_{2\mu}]$, we denote by $p(t)$, $r(t)$ and $q(t)$ the points with first coordinate $t$
on the segment bounded by $(a,0)$ and $P_{2\mu}$, on $\partial \dommap^\mu$ and on $\partial\dommap$ respectively (of course, $r(t)=q(t)$
for $t\in [t_\mu,t_{2\mu}]$), see Figure \ref{fig:domain_mu}. 

Thus we define 
$$X_\e:\dommap^\e\to\R^3$$ 
as follows: 
\begin{displaymath}
X_\e(t,s):=
\begin{cases}
\phi_\e(t,s)\, &\mbox{ if }t\in[a,t_{2\e}),\,  s\geq\ell_{2\e}(t), \\
X(t,s)\, &\mbox{ otherwise}, 
\end{cases}
\end{displaymath}
where $\phi_\e$ is linear on the vertical lines, $\phi_\e(p(t))=X(p(t))$ and $\phi_\e(r(t))=X(q(t))$.    
We observe that $X_\mu$ is still a semicartesian parametrization.

Denoting by $\Sigma_\mu$ the image of $\dommap^\mu$ through $X_\mu$,
we have $\H^2(\Sigma_\mu) \to \H^2(\Sigma) $ as $\mu \to 0^+$ (\cite[\textsection 305]{Nit:89}). 

\medskip
{\tt step 2}.

For every $\mu>0$ small enough, let us define the sequence $(\mappa^\mu_\eps)$ as follows:

\begin{displaymath}
u_{\eps 1}^\mu :=
\begin{cases}
\mappauno & \mbox{ in } \Om \setminus (\Lambda(\rettangolo_\eps)\cup 
B^a_{\delta\eps/2}\cup B^b_{\delta\eps/2} ),
\\
\mappa(T^\mu_\eps)_1 & \mbox{ in } \Lambda(\rettangolo_\eps \setminus \dommap^\mu_\eps) \setminus (B^a_{\delta\eps/2}\cup B^b_{\delta\eps/2} ),\\
X_{\mu 2}(t,s/\eps) & \mbox{ in }  \Lambda(\dommap_\eps^\mu)\setminus (B^a_{\delta\eps/2}\cup B^b_{\delta\eps/2} ),\\
\psi_\eps^\mu & \mbox{ in }  B^a_{\delta\eps/2}\cup B^b_{\delta\eps/2} ,\\
\end{cases}
\end{displaymath}
and
\begin{displaymath}
u_{\eps 2}^\mu :=
\begin{cases}
\mappadue & \mbox{ in } \Om \setminus \Lambda(\rettangolo_\eps),\\
\mappa(T^\mu_\eps)_2 & \mbox{ in } \Lambda(\rettangolo_\eps \setminus \dommap^\mu_\eps) ,\\
X_{\mu 3}(t,s/\eps) & \mbox{ in }  \Lambda(\dommap_\eps^\mu),\\
\end{cases}
\end{displaymath}
where:
\begin{itemize} 
\item[-] $\Lambda$, $\rettangolo$, $\ppa=\ppa(x,y)$, $\spa=\spa(x,y)$ are defined 
as in {\tt step 6} of the proof of Theorem \ref{teo:graph_main};
\item[-] $T^\mu_\eps$ is defined as at (\ref{eq:ue2}), with $\dommap^\mu$
in place of $\dommap$ and $\sigma_\mu^\pm$ instead of $\pm \sigma$;
\item[-] $B^a_{\delta\eps/2}$ (respectively $B^b_{\delta\eps/2}$) is
the disk centered at $\alpha(a)$ (respectively $\alpha(b)$) with radius $\delta\eps/2$;

\item[-] the function $\psi^\mu_\eps$ is linear along the radii and is equal to $\mappauno$ in $\alpha(a)$ and
to $u_{\eps 1}^\mu$ on $\partial B^a_{\delta\eps/2}$
(similarly in the other end point of the jump).
\end{itemize}
By construction and thanks to 
hypotheses ({\rm u}1), (\hp 2)-(\hp 4), the sequence $(\ue^\mu)$ is 
in $W^{1,2}(\Om;\R^2)$.

We have
\begin{itemize}
\item[-] 
$\rel\left(\ue^\mu,\Om\setminus\left(\Lambda(\rettangolo_\eps)
 \cup B^a_{\delta\eps/2 \cup B^a_{\delta\eps/2}}\right)\right) 
=  \rel\left(\mappa,\Om\setminus\left(
\Lambda(\rettangolo_\eps)
 \cup B^a_{\delta\eps/2} \cup B^a_{\delta\eps/2}\right)\right) $ by definition;
\item[-]
$\rel\left(\ue^\mu,\Lambda(\rettangolo_\eps
\setminus\Ce^\mu)\setminus
(B^a_{\delta\eps/2} \cup B^a_{\delta\eps/2})\right)\to 0$ as $\eps\to 0^+$
because of the estimates done in {\tt step 8}
 of the proof of Theorem \ref{teo:graph_main} since $|\grad T^\mu_\eps|$ is bounded;
\item[-] $\rel\left(\ue^\mu,\Lambda(\Ce^\mu)\setminus
(B^a_{\delta\eps/2} \cup B^a_{\delta\eps/2})\right)\to \H^2(\Sigma_\mu)$ as $\eps\to 0^+$: indeed 
the computations done in {\tt step 9} of Theorem \ref{teo:graph_main}
work, since they depend only
on the fact that the parametrization is in semicartesian form.
\end{itemize}
Thus 
\begin{displaymath}
\lim_{\eps \to 0^+}
\rel\left(
\ue^\mu , \Om \setminus (\saltou \cup  
((B^a_{\delta\eps/2}\cup B^b_{\delta\eps/2})\setminus \Lambda(\rettangolo_\eps))
\right)= \rel(\mappa, \Om \setminus \osaltou)+\H^2(\Sigma_\mu).
\end{displaymath}

In order to compute $\rel (\ue^\mu, B^a_{\delta\eps/2})$ we observe that 
\begin{equation}\label{eq:grad_psi}
|\grad \psi_\eps^\mu|\leq \frac{C}{\eps}
\end{equation}
for some constant $C>0$ independent of $\eps$ and $\mu$.
Thus for a possibly different value of the constant $C$ 
(still independent of $\eps$ and $\mu$),
\begin{displaymath}
\begin{split}
&\rel(\ue^\mu, B^a_\eps \setminus\osaltou)\\
&\leq C\int_{B^a_{\delta\eps/2}\setminus \osaltou} \left[ 1+ |\grad \psi^\mu_\eps| + |\grad u_{\eps 2}^\mu | + 
| \partial_x\psi_\eps^\mu ~\partial_yu_{\eps 2}^\mu-  
\partial_y\psi_\eps^\mu ~\partial_x u_{\eps 2}^\mu| \right]~dx\, dy\\
&\leq C \left[ \mathcal{O}(\eps^2) + \mathcal{O}(\eps)\right]
+ (1+C)\int_{B^a_{\delta\eps/2}\setminus \osaltou}|\grad u_{\eps 2}^\mu| 
~dx\, dy,
\end{split}
\end{displaymath}
where we have used \eqref{eq:grad_psi}.
Recalling that on $\Lambda(\dommap^\mu_\eps)$ we have $u_{\eps 2}(x,y)= X_{\mu 3}(t,s/\eps)$, the term
$$
\int_{\left(B^a_{\delta \eps/2}\cap \Lambda(\dommap^\mu_\eps)\right)\setminus \osaltou}|\grad u_{\eps 2}^\mu| 
~dx\, dy
$$
is negligible as $\eps \to 0^+$. 
On the other hand on $B^a_{\delta\eps/2}\setminus \Lambda(\rettangolo_\eps)$ we have $u_{\eps 2}^\mu=\mappadue$,
and thus
$$
\int_{B^a_{\delta\eps/2}\setminus \Lambda(\rettangolo_\eps)}|\grad u_{\eps 2}^\mu| dx\, dy=\mathcal{O}(\eps^2).
$$
{}Finally we get an analogous result also on $B^a_{\delta\eps/2}\cap\Lambda(\rettangolo_\eps \setminus \dommap^\mu_\eps)$
since here $u^\mu_{\eps 2}$ is defined as $(\mappa (T^\mu_\eps))_2$ and $T^\mu_\eps$ has bounded gradient and tends to the identity.  

Thus the area contribute on $B^a_{\delta\eps/2}$
is asymptotically negligible (and similarly on $B^b_{\delta\eps/2}$).

Finally, since $\H^2(\Sigma_\mu)$ tends to $\H^2(\immmapmin)$ as $\mu \to 0^+$, we can choose  
$\ue$ as $\ue^{\mu_\eps}$ for a suitable sequence $(\mu_\eps)$ 
converging to zero, so that we get
\eqref{eq:lim_ue}. Recalling that $W^{1,2}(\Om;\R^3)\subset\DM$ and applying Lemma \ref{lem:relaxation}
we obtain \eqref{limsup}.

\begin{Remark}
\textup{
If $\mappa$ satisfies $({\rm u}1)$, (\hp 2), (\hp 3)
and  $\Gamma[\mappa]$, defined as in \eqref{eq:Gammau}, 
 is  contained in a plane $\Pi$, then 
$$
\rel(\mappa, \Omega)= \rel(\mappa, \Omega \setminus \osaltou) + \H^2(\immmapmin).
$$
Indeed $\immmapmin$ is of course the portion of $\Pi$ bounded by $\Gamma$; 
moreover, thanks to the definition of $\Gamma$, the plane $\Pi$  
cannot be orthogonal to the versor $(1,0,0)$. Thus either the projection of $\immmapmin$
on the plane $\R^2_{(\ppa,\xi)}$ or its projection on $\R^2_{(\ppa,\eta)}$ 
is a domain with non-empty interior. On the symmetrization of this domain we can
define a semicartesian parametrization of $\immmapmin$ and, 
applying Theorem \ref{teo:general_main}, we find
$$
\rel(\mappa, \Omega)\leq \rel(\mappa, \Omega \setminus \osaltou) + \H^2(\immmapmin).
$$
On the other hand, in this case $\H^2(\immmapmin)=|D^s\mappa|(\Om)$ thus,
using relation (\ref{eq:altra_diseq}), we have also 
$$
\rel(\mappa, \Omega)\geq \rel(\mappa, \Omega \setminus \osaltou) + \H^2(\immmapmin).
$$
}
\end{Remark}

\section{On the existence of semicartesian parametrizations}\label{sec:semicart}
Now our goal is to state some conditions on $\Gamma$ 
which allow to construct a semicartesian parametrization
for the corresponding area-minimizing surface, in order 
to apply Theorem \ref{teo:general_main} for suitable maps $\mappa$.
Theorem \ref{prop:analytic} provides some sufficient conditions:
roughly, we shall assume that $\Gamma$ is the union of the graphs of two analytic
curves, joining in an analytic way and satisfying a further assumption
of non degeneracy. 
We stress that
 the analiticity forces the 
gradient of $\mappa$ to blow up near the crack tips.

The proof of Theorem \ref{prop:analytic} is quite involved and it is postponed to
section \ref{sec:par}.

\medskip

We start with the following definition.
\begin{Definition}[\textbf{Condition (A)}]\label{def:condA}
\textup{
We say that a curve $\Gamma$ union of two graphs satisfies condition (A) if
there exists an injective 
{\it analytic}  map 
$$
\parabordo  = (\parabordo_1,\parabordo_2,\parabordo_3)
: \partial\disco \to \R_t \times \target
$$
such 
 that
$$
\Gamma = 
\parabordo(\partial \disk)
$$
where, 
still denoting for simplicity  by $\parabordo$ the composition
$\parabordo \circ \parabordopalla$ (see \eqref{eq:bordo_palla}), and using the  
prime for differentiation with respect to 
$\theta$, the following properties are satisfied:
\begin{equation}\label{eq:nondege}
\begin{aligned}
& |g'(\theta)|\neq 0, \qquad \theta \in [0,2\pi),
\\ 
& \parabordo_1' < 0 \mbox{ in } (\theta_{\rm n},\theta_{\rm s}), 
\\
& \parabordo_1' > 0 \mbox{ in } (\theta_{\rm s},\theta_{\rm n}),
\\
& \parabordo_1''(\theta_{\rm s})>0, \qquad \parabordo_1''(\theta_{\rm n})<0.
\end{aligned}
\end{equation}
}
\end{Definition}

Note carefully that the last three conditions involve the first component
of $g$ only.

Our  result is the following.

\begin{Theorem}[\textbf{Existence of semicartesian parametrizations}]\label{prop:analytic}
Let $\Gamma\subset \R^3$ be a curve union of the two graphs 
of $\gamma^\pm$ and satisfying condition \textup{(A)}. 
Then there exist
an analytic, connected, simply connected, and bounded set
$\dommap$ and a disk-type area-minimizing
solution  $X\in\C^\omega(\overline{\dommap}; \R^3)$
of the Plateau's problem
for the contour $\Gamma$, 
satisfying Definition \ref{def:semicart_par}, 
with  
$X$  free of 
interior branch points and of
boundary branch points.
Moreover,
\begin{itemize}
\item[(i)]
near the point $(a,0)$, the curve $\partial D$ 
is of the form $\{(\tau(\spa), \spa)\}$, for $|\spa| $ small enough,
with $\tau$ as in \ref{eq:taus}
and $\alpha_2 > 0$, and similarly near the point $(b,0)$;
\item[(ii)] 
the Lipschitz constant of $\sigma^\pm$ on a relatively compact
subinterval of $(a,b)$ is bounded
by the Lipschitz constant of the restriction of $\gamma^\pm$ on the 
same subinterval.
\end{itemize}
\end{Theorem}

\begin{Remark}\label{rk:simmetrization}
\textup{
The semicartesian parametrization provided by Theore \ref{prop:analytic} 
could not satisfy the condition
$$
(a,b)\times \{0\} \subset {\rm int}(\dommap).
$$
We can obtain a semicartesian parametrization
fulfilling condition (\hp 4) of Theorem \ref{teo:general_main} 
by symmetrizing the domain,  
as in 
{\tt step 4} of the proof of Theorem \ref{teo:graph_main}
}
\end{Remark}

{}From Remark \ref{rk:simmetrization} and Theorems \ref{teo:general_main} 
and \ref{prop:analytic} we get the following result.

\begin{Corollary}\label{cor:analytic}
Suppose that
 $\mappa$ satisfies $({\rm u}1)$, (\hp \textup{2}), (\hp \textup{3}) and
that  $\Gamma[\mappa]$
satisfies condition \textup{(A)}. Then there exists a sequence 
\begin{equation*}
(\ue)_\eps \subset W^{1,2}(\Om;\R^2)
\end{equation*}
converging to $\mappa$ in $L^1(\Omega; \R^2)$ as $\eps \to 0^+$ 
satisfying \eqref{eq:lim_ue}. Hence \eqref{limsup} holds.
\end{Corollary}

\begin{Remark}\label{rem:order}\rm
Before proving Theorem \ref{prop:analytic}, the following comments are in order.
\begin{itemize}
\item[-] Note the special  structure of the curve 
$\Gamma[\mappa]$ in Corollary \ref{cor:analytic}:  it is not the graph of 
an $\R^2$-valued  function over $\osaltou$,
but it is instead the union of 
two analytic graphs, {\it joining together in an
analytic way}, of the two $\R^2$-valued functions $\mappa^\pm$.
Since globally the map $\parabordo$ is required to be 
analytic, it results that $\mappa^\pm$ are not independent. 
\item[-] The nondegeneracy
requirement \eqref{eq:nondege} of $\parabordo$ 
at the south and north poles are necesssary
in order the proof 
of Theorem \ref{lem:piani} to work. In particular, 
it is needed to ensure that the restriction of the height function
$h$ to $\partial B$ is a Morse function ({\tt step 4} of the proof of Theorem \ref{lem:piani}). 
\item[-] As we shall see, the analiticity requirement 
 in hypothesis (\hp 4) is needeed in order to prevent 
the existence of boundary branch points in a disk-type
solution of the Plateau's problem with boundary $\Gammau$.
\end{itemize}
\end{Remark}

\begin{Example}[\textbf{Maps satisfying the hypotheses of Theorem
\ref{prop:analytic}
}]\label{exa:cerchio}\rm
In this example we present a map $\mappa$ satisfying $({\rm u} 1)$, (\hp 2) and  (\hp 3)
and whose $\Gamma=\Gamma[\mappa]$ satisfies condition 
(A) and hence, from Theorem \ref{prop:analytic}, also condition (\hp 4). 
The map $\mappa$ is defined so that $\Gamma$ is a perturbation of the circle: indeed $\Gamma$ is exactly 
the boundary of the unit disk contained in the plane $\R^2_{(t,\xi)}$ if $\mappadue$ is identically zero.
Consequently, the nondegeneracy
conditions expressed in \eqref{eq:nondege} hold. It is clear that,
starting from $\mappa$, several other maps satisfying the same conditions can be constructed.

\smallskip

Let $\Om$ be an open connected subset of $\R_{(\ppa, \spa)}^2$ containing the 
square $[-1,1]^2$  and let us consider the map $\mappa=(\mappauno, \mappadue): \Om \to \R^2$
defined by 
\begin{displaymath}
  \mappauno(\ppa,\spa):=
    \begin{cases}
      \sqrt{1-\ppa^2+\spa^2}, &\mbox{ if }|\ppa |<1,\,s>0\\
      - \sqrt{1-\ppa^2+\spa^2}, &\mbox{ if }|\ppa |<1,\,s<0\\
      \spa, &\mbox{ otherwise }
    \end{cases}
\end{displaymath}
and $\mappadue\in \C^1(\Om)\cap W^{1,\infty}(\Om)$ such that condition (A) holds 
(the simpler example is of course $\mappadue\equiv 0$).
We notice that $\mappauno \in {\rm Lip}_{\rm loc}\left(\Om\setminus\osaltou\right)\cap \C^1\left(\Om\setminus \osaltou\right)$ 
(with $\osaltou=[-1,1]\times\{0\}$);
it is in $W^{1,1}(\Om\setminus \osaltou)$ but it fails to be in $W^{1,2}(\Om\setminus\osaltou)$, 
as it can be checked directly (see Remark \ref{rem:W12}).
Consequently, since the gradient of $\mappadue$ is supposed to be bounded,
we get also that $\M(\grad \mappa)$ is in $L^1(\Omega\setminus\osaltou;\R^6)$.
In this particular case $\Gamma$ is a close simple analytic curve lying 
on the cylinder with base the unit disk, 
thus the existence of a semicartesian parametrization is obvious since 
the area-minimizing surface spanning $\Gamma$ can be described 
as a graph on the disk (Theorem \ref{teo:convex_proj}). 
We stress that $\Gamma$ satisfies (A) and thus we could apply 
the argument in the proof of Theorem \ref{prop:analytic}. It is easy to modify this example
keeping the same behaviour near the poles but losing the convexity of the projection of $\Gamma$.
\end{Example}

\begin{Example}\label{exa:manu}\rm
Two other interesting examples of curves $\Gamma$ satisfying condition (A)
have already been discussed in the introduction and plotted in 
Figure \ref{fig:intro}.
\end{Example}

\section{Proof of Theorem \ref{prop:analytic}}\label{sec:par}
In this section 
we prove Theorem \ref{prop:analytic}. The proof is involved,
and we split it into various points.  

Let $\immmapmin$ be an area-minimizing surface spanning
$\Gamma$ and having 
the topology of the disk. Let
\begin{equation}\label{eq:mappaY}
	\begin{split}
		\paramap: (
\parahilduno
,\parahilddue)\in \overline B\subset \R^2_{(\parahilduno,\parahilddue)} &\to 
(\paramap_1(\parahilduno,\parahilddue),\paramap_2(\parahilduno,\parahilddue),\paramap_3(\parahilduno,\parahilddue)) \in 
\R^3 = \R_\primoparametroastratto
\times \target
	\end{split}
\end{equation}
be a conformal parametrization of $\immmapmin$ (see Theorem 
\ref{teo:existence}).

Since we can assume the three points condition (Remark 
\ref{rk:3points}), we suppose that
\begin{equation}\label{eq:duepunti}
\begin{split}
\paramap(0,-1)&
=(a,\gp(a))
=(a,\gm(a))=:\ps
\\ 
\paramap(0,1)&=
(b,\gp(b))=(b,\gm(b))=:\pn,
\end{split}
\end{equation} 
and we fix a third condition as we wish (respecting
the monotonicity on the boundary parametrization), for definitiveness
$$
\paramap(1,0)=((a+b)/2,\gp((a+b)/2)).
$$

\medskip

In Section \ref{sec:trans} we show a transversality property. 
We will make use of Morse relations for manifolds with boundary,
in order to exclude, for a suitable
Morse function,  the presence of critical
points of index one. The absence of 
boundary branch points for $Y$ will be used 
in the proof. 

In Section \ref{sec:global} we explain how this transversality property ensures the
existence of a local semicartesian parametrization and, using some compactness argument
and the simply connectedness of $\immmapmin$, also of a global semicartesian
parametrization.

Finally in Section \ref{sec:domshape} we provide the regularity and the
shape of the domain of this semicartesian parametrization.

\subsection{A transversality result}\label{sec:trans}

Let $\mathcal{P}$ be the family of parallel planes orthogonal to the unit vector $e_\ppa=(1,0,0)$, 
that is the planes in the form
$$
 \Big\{(\primoparametroastratto,\xi,\eta) \in \R_t \times\target: \ppa 
= {\rm const}\Big\}. 
$$
Notice that
each point of $\immmapmin$ is met by some $\Pi \in \mathcal P$.

The next result is one of the most delicate parts of the proof of Theorem \ref{prop:analytic}.

\begin{Theorem}[\textbf{Transversality}]\label{lem:piani}
In the same hypotheses of Theorem \ref{prop:analytic}, none of the planes of $\mathcal P$ 
is tangent to $\immmapmin$.
\end{Theorem}
\begin{proof}
We have to show that
the normal direction to $\immmapmin$ at a point of $\immmapmin$ 
is never parallel to $(1,0,0)$; at self-intersection points
of $\immmapmin$, the statement 
refers to all normal directions.

Our strategy is to introduce a height function 
having the planes of the family $\mathcal P$ as level sets, namely 
the function 
given by the first coordinate $\primoparametroastratto$
in $\R^3 = \R_t \times \target$, restricted to an extension of $\immmapmin$. 
The proof consists then in 
proving that the only critical points of the height function are the minimum
and the maximum corresponding to points $\ps$ and $\pn$ (see \eqref{eq:duepunti}). 

\smallskip

Since $\partial \immmapmin = \Gamma$ is non-empty, 
in order to deal with boundary critical points 
first of all it is convenient
to extend $\immmapmin$ across $\Gamma$. 

By condition (A)  the curve
 $\Gamma$ is analytic; therefore (Theorem \ref{teo:ext})
we can extend $\immmapmin$ to an analytic  minimal 
surface $\extimmmapmin$ 
across $\Gamma$; $\extimmmapmin$ 
can be parametrized on a bounded smooth simply connected  open set $\extdisk
\supset \disk$
through an analytic map $\extmapmin =(\extmapmin_1, \extmapmin_2, \extmapmin_3)$ 
which coincides with $Y$ on $\overline B$, is harmonic, i.e. $\Delta \extmapmin=0$ in $\extdisk$,
and satisfies the conformality relations
$
\vert \extmapmin_\phu\vert^2 = 
\vert \extmapmin_\phd \vert^2$, 
$\extmapmin_\phu\cdot \extmapmin_\phd  =0$ in $\extdisk$.
In addition, 
from Theorems \ref{teo:nointbrapo} and \ref{teo:noboundarybp}, $\extmapmin$ has 
no interior (i.e., in $\disk$) and no
boundary (i.e., on $\partial \disk$) branch points. Hence, possibly reducing $\extdisk$, we can 
suppose that $\extmapmin$ has no branch points in $\extdisk$.

Therefore, the Gauss map 
\begin{equation}\label{eq:gauss}
	\gauss: (\phu,\phd) \in  \extdisk \to \gauss(\phu,\phd):= 
		\frac{\extmapmin_\phu(\phu, \phd) \wedge \extmapmin_\phd(\phu, \phd)}
			{|\extmapmin_\phu(\phu, \phd)\wedge \extmapmin_\phd(\phu, \phd) |}
\end{equation}
is well-defined in $\extdisk$\footnote{$\gauss$ is also harmonic and satisfies the conformality relations, see \cite[Chapter 1.2]{DiHiSa:10}.}. 

Let us define 
\begin{displaymath}
		\height: (\phu,\phd) \in \extdisk  \to 
\height(\phu,\phd) :=  \extmapmin_1(\phu,\phd) \in 
\R_\primoparametroastratto.
\end{displaymath}  

Observe that 
$(\phu_0, \phd_0)
\in \extdisk$ is a critical point for $h$ if and only if 
the plane 
$\{(\primoparametroastratto,\xi,\eta)\in \R^3: \primoparametroastratto = \extmapmin_1(\phu_0,\phd_0)\}$ 
is tangent to $\extimmmapmin$ at $\extmapmin(\phu_0,\phd_0)$.
Indeed, criticality implies  
$\partial_\phu \extmapmin_1 (\phu_0 ,\phd_0)
=\partial_\phd \extmapmin_1 (\phu_0 ,\phd_0)=0$,
and one checks from \eqref{eq:gauss} that  
\begin{equation}\label{eq:N}
\gauss(\phu_0,\phd_0) =(1,\,0,\,0).
\end{equation}
On the other hand, if $\gauss(\phu_0,\phd_0)=(1,0,0)$
the image of any vector through the differential of $Y$ at $(\phu_0,\phd_0)$
is orthogonal to $(1,0,0)$. If in particular we consider the image of $e_\phu=(1,0)$
and $e_\phd=(0,1)$ we obtain $Y_{1\phu}(\phu_0,\phd_0)=0=Y_{1\phd}(\phu_0,\phd_0)$.

{}From the above observation, it follows that the thesis of the theorem
reduces to show that the function $h$ has no critical points in $\overline
B$, except for $(0,\pm1)$, for which we shall prove separately that $\gauss(0,\pm 1)\neq (1,0,0)$.

\medskip

At first, we shall show that the thesis of the theorem holds true up to
a small rotation of $\extimmmapmin$ around a line in the orthogonal space to $(1,0,0)$ 
that takes a direction in a suitable set to become $(1,0,0)$; moreover this set of 
directions is dense in a small neighborhood of $(1,0,0)$.

In the last step we will show that the statements holds true \textit{without applying this rotation}.

\medskip

{\tt step 1}. Up to a suitable rotation in $\R^3$, the function $\height$ has no degenerate critical points.

\medskip
%Let us first prove that we can assume
%that all critical points $h$ are nondegenerate,
% namely that $h$ is non-degenerate.

We notice that
any  degenerate critical point of $\height$
is  a critical point also for the Gauss map.
Indeed let $(\phu_0, \phd_0)\in \extdisk$ be  critical:
using \eqref{eq:N} we have for the coefficients
of the second fundamental form
$$
\extmapmin_{\phu \phu} \cdot \gauss
=\extmapmin_{1 \phu \phu}
=
h_{\phu \phu}, 
\qquad
\extmapmin_{\phu \phd} \cdot \gauss
=\extmapmin_{1 \phu\phd}=h_{\phu\phd},  \qquad
\extmapmin_{\phd \phd} \cdot \gauss
=\extmapmin_{1 \phd\phd} = 
h_{\phd\phd}.
$$
If in addition
$(\phu_0, \phd_0)$ is degenerate, then
the determinant of the 
Hessian of $\height$ at $(\phu_0, \phd_0)$ vanishes,
and this implies that also the determinant of the second fundamental form is zero.
That is $(\phu_0, \phd_0)$ is a critical point for the
Gauss map.

{}From Sard's lemma, it follows that we can find a rotation around a 
line in the orthogonal space  to $(1,0,0)$,
as close as we want to the identity,  so that
the $t$-direction does not belong to the set of critical values of the Gauss map.
Moreover such a rotation can be freely chosen in a set that is dense in a neighborhood
of the identity.
We also remark that for a sufficiently small rotation condition (A) 
remains valid although the values $\theta_{\rm n}$ and $\theta_{\rm s}$ of the
parameter leading to maximal and minimal value of the
$\primoparametroastratto$-component are perturbed of a small amount.

Therefore, from now on we assume that
$$
{\rm  all~
critical~ points~ of~ } \height  {\rm ~in} ~\extdisk {\rm  ~are~ nondegenerate}.
$$
{\tt step 2}. The height function $h$ has no critical points on $\partial B$.

\medskip

Suppose by contradiction that there exists $(\parahilduno,\parahilddue)\in \partial B\setminus\{(0,\pm 1)\}$
such that $\grad h (\parahilduno,\parahilddue)=0$, namely 
$(\parahilduno,\parahilddue)$ is a critical point of $h$ different from $(0,\pm 1)$.
We claim that if $\tandisco\in \R^2$, $|\tandisco|=1$, $\tandisco$ tangent to $\partial B$ at $(\parahilduno,\parahilddue)$ ,
then for some $\lambda\neq 0$
\begin{displaymath}
Y_{\tandisco}(\parahilduno,\parahilddue)=\lambda \tau_\Gamma(\parahilduno,\parahilddue),
\end{displaymath}
where $\tau_\Gamma(\parahilduno,\parahilddue)$ is a tangent unit vector to $\Gamma$ at $Y(\parahilduno,\parahilddue)$ 
and $Y_{\tau_{\partial \disk}}$ is the derivative of $Y$ along $\tau_{\partial \disk}$.
Indeed, since $Y$ is smooth up to $\partial B$, it follows that $Y_{\tandisco}(\parahilduno,\parahilddue)$
is tangent to $\Gamma$ at $Y(\parahilduno,\parahilddue)$. 
Now write $\tandisco=\alpha e_\parahilduno+\beta e_\parahilddue$, $\alpha ^2 +\beta ^2 =1$ and $e_\parahilduno=(1,0)$, 
$e_\parahilddue=(0,1)$. 
Since
\begin{displaymath}
Y_{\tandisco}(\parahilduno, \parahilddue)=\alpha Y_\parahilduno(\parahilduno, \parahilddue)+\beta Y_\parahilddue(\parahilduno, \parahilddue),
\end{displaymath}
the conformality relations imply
\begin{displaymath}
\vert Y_{\tandisco} (\parahilduno, \parahilddue) \vert ^2=(\alpha ^2 +\beta ^2)\vert Y_\parahilduno (\parahilduno, \parahilddue) \vert ^2.
\end{displaymath}
Then the absence of boundary branch points guarantees that
$\vert Y_{\tandisco} (\parahilduno, \parahilddue) \vert ^2 \neq 0$.
Hence $Y_{\tandisco}(\parahilduno, \parahilddue)$
is a non-zero vector parallel to $\tau_\Gamma (\parahilduno, \parahilddue)$
and the claim follows.
Observe now that, by assumption, $\tau_\Gamma(\phu, \phd)$ has non-zero
$t$-component, so that
\begin{equation}\label{eq:no_crit_point}
\alpha Y^1_\phu(\phu, \phd) + \beta Y^1_\phd(\phu, \phd)\neq 0,
\end{equation}
which contradicts the criticality of $(\phu,\phd)$ for $\height$.
Thus \eqref{eq:no_crit_point} shows that $h$ has no critical points
on $\partial B \setminus \{(0,\pm 1)\}$.
In order to exclude that $\ps$ (and similarly $\pn$) is a critical point for $h$,
we observe that condition (A) implies that the convex hull
of $\Gamma$, and hence the convex hull of $\immmapmin$\footnote{
Any connected minimal surface $X$ with a parameter domain $D$ is contained 
in the convex hull of $X_{\vert\partial D}$. 
See \cite[Theorem 1, chapter 4.1]{DiHiTr:10a}. },
is contained in a wedge having the tangent to $\Gamma$ at its 
lowest point as ridge and the two slopes are strictly increasing starting
from the ridge. Thus the normal vector to $\extimmmapmin$ in $\ps$ cannot 
be parallel to $(1,0,0)$. 

\medskip

As a consequence of {\tt step 2} we can suppose that all 
critical points of $h$ are contained in $B$. 

\medskip

{\tt Step 3}. 
The function $\height$ has neither local maxima nor local 
minima in $\disk$.

\medskip

Indeed, assume by contradiction that  $p = Y(\phu_0,\phd_0)
\in \immmapmin$, where
$(\phu_0,\phd_0)
\in \disk$ is a local minimum point for $\height$.  
Then locally the surface $\immmapmin$ is contained in a half-space delimited by the 
tangent plane $\{(\primoparametroastratto, \primacoordtarget, \secondacoordtarget): 
\primoparametroastratto= Y_1(\phu_0,\phd_0)\}$,
the intersection with this tangent plane being locally only the point
$Y(p)$.
We now construct a competitor surface $\Sigma'$ as follows: 
we remove from $\immmapmin$ a small portion locally around
$p$, obtained 
by cutting $\immmapmin$ locally  
with a plane at a level slightly 
higher than the minimal value. We fill the removed portion
with a portion of plane, and this givs $\Sigma'$\footnote{If the cut level is close 
enough to the critical level,
$\Sigma'$ is the image of a map in 
$\C(\Gamma)$ (see Appendix \ref{sec:app}).}. 
Then the area of $\Sigma'$ is strictly
smaller than the area of $\immmapmin$,
a contradiction.

A similar argument holds for a local maximum point and therefore
the proof of {\tt step 3} is concluded.

\medskip

Employing the notation of Section \ref{sec:appb} ,
 we have therefore
$$
m_0(h,\disk) = m_2(h,\disk) =0.
$$

The next  step is a consequence of the monotonicity and nondegeneracy assumptions 
expressed in (\ref{eq:nondege}), and of the conformality and analiticity of $\immmapmin$.

\medskip

{\tt Step 4}. The restriction $h_{\vert \partial \disk}$ of $h$ to $\partial \disk$ is a Morse function;
moreover $m_0^-(h_{\vert \partial_h^- \disk}) =1$ and 
$m_1^-(h_{\vert \partial_h^- \disk}) =0$ (Section \ref{sec:appb}).
%
% and it has one  
%minimum in $(0,-1)$ and one maximum $(0,1)$ and they are nondegenerate.
%

\medskip

We observe that condition (A) implies that there exists a parametrization of $\Gamma$ 
on $\partial \disk$ whose first components is a Morse function. We have to show that also
the parametrization induced by the area-minimizing minimal surface $Y$
has the same property. 

As already done for the function $\parabordo$, we denote by $\extmapmin_{|\partial \disk}$
and by $h_{|\partial \disk}$ the composition $\extmapmin \circ \parabordopalla$ and $h \circ \parabordopalla$ respectively
(see \eqref{eq:bordo_palla}) and we use the prime for differentiation with respect to $\theta$.
At first, we observe that out of branch points, all the directional derivatives of $\extmapmin$ are non zero. 
Thus in particular, from the absence of boundary branch points on $\partial \disk$, we deduce that 
$$
|(\extmapmin_{|\partial \disk}S)'(\theta)|\neq 0,\quad \theta\in[0,2\pi).
$$ 
On the other hand since $g$ is analytic with differentiable inverse, there exists a $\C^1$ function $\psi$ from $[0,2\pi]$
in itself such that $\psi(2\pi)=\psi(0)+2\pi$ and
$$
\extmapmin_{|\partial \disk}(\theta)=g(\psi(\theta)),\quad \theta \in [0,2\pi).
$$ 
Differentiating the last expression and remembering from \eqref{eq:nondege} that $|g'|\neq 0$, 
we get that also $\psi '$ never vanishes, indeed:
$$
0\neq |(\extmapmin_{|\partial \disk})'(\theta)|=|g'(\psi(\theta))||\psi'(\theta)|.
$$
{}From the semicartesian form of $\Gamma$, $h_{|\partial \disk}$ has just a minimum and a maximum
in corrispondence of $N=(0,1)$ and $S=(0,-1)$. From the properties of $\parabordo$ we infer
that $\psi(\theta_{\rm s})$ is the value of $\theta$ corresponding to $S$ and similarly for $N$.
Since
$$
h_{|\partial \disk}''(\theta)= \parabordo_1''(\psi(\theta))(\psi(\theta))^2 + g_1'(\psi(\theta))\psi''(\theta),
$$ 
computing for the values corresponding to $S$ and $N$ we get that the first addend is non-zero while the
second vanishes. We have thus proven that $h_{|\partial \disk}$ is a Morse function, with a maximum in $(0,1)$
and a minimum in $(0,-1)$.

\medskip

Following once more Section \ref{sec:appb} (see \eqref{eq:bordomeno}), we now set
$$
\partial^-_h \disk := 
\{(\phu,\phd) \in \partial \disk : \grad h(\phu,\phd)\cdot \nu_\disk(\phu,\phd) <0\},
$$
where $\nu_B(\phu,\phd)$ denotes the outward unit normal to $\partial \disk$
at $(\phu,\phd) \in \partial \disk$.

We prove that 
$$
(0,-1) 
\in  \partial_h^- \disk \qquad {\rm and} \qquad
(0,1) 
\notin  \partial_h^- \disk.
$$
Indeed if
$\grad h(0,-1)\cdot \nu_\disk(0,-1) \geq0$, we get a contradiction
from the same argument used in {\tt step 2} to prove that  $(0,-1)$
is not critical for $h$. Similarly $(0,1) 
\notin  \partial_h^- \disk$.

\medskip

We have thus obtained that 
$$
m_0^-(h_{\vert \partial_h^- \disk}) =1, \qquad 
m_1^-(h_{\vert \partial_h^- \disk}) =0.
$$

\medskip

{\tt Step 5}. The function $h$ has no saddle points in $\disk$.

\medskip

The Morse function $h$ ({\tt step 1}) has 
no points of index zero (minima) in $\disk$ 
and no points of index two (maxima) in $\disk$ by 
{\tt step 3}: again following the notation of 
Section \ref{sec:appb} (see \eqref{def:Mi}), we have 
$$
M_0(h, \disk \cup\partial \disk) = 1, \qquad M_2(h, \disk \cup \partial \disk) =0.
$$
In addition, using {\tt steps 2} and {\tt 4} we can apply Theorem \ref{teo:mors},
and obtain, being $\chi(\disk)=1$, 
\begin{equation*}
		M_1(h)  = M_0(h, \disk \cup\partial \disk) + M_2(h, \disk \cup\partial \disk) - \chi(\disk) =0.
\end{equation*}
%, and of
%the theorem.

\medskip

{\tt step 6.} It is not necessary to apply any rotation.

\medskip
It is sufficient to show that the direction given by
$(1,0,0)$ is actually not critical for the Gauss map.
At first we can assume that $\Gamma$ is not contained in a plane. Indeed if it 
were planar, necessarily 
$$\gauss(\phu,\phd)=\nu_0,\quad (\phu, \phd) \in \disk$$ 
for some constant unit vector $\nu_0\neq (1,0,0)$, since $\Gamma$ is union of two graphs.
 
Assuming that $\Gamma$ is non planar, we reason by contradiction and suppose that there is a degenerate critical
point $p=Y(\phu_0,\phd_0)$ for the height function $h$ in the inside of $\immmapmin$.
This means that $(\phu_0,\phd_0)$ is a critical point for the Gauss map, that is the 
product $\kappa_1 \kappa_2$ of the two principal curvatures is $0$; 
because of the minimality of $\immmapmin$
we get that $p$ is an umbilical point, with $\kappa_1=0=\kappa_2$. 
Recalling that in a non-planar minimal surface the umbilical points are isolated
(see for example \cite[Remark 2, chapter 5.2]{DiHiSa:10}), we can find a direction
in a small neighborhood of $(1,0,0)$ that is normal to $\immmapmin$ in
a neighborhood of the degenerate critical point $p$ and is not a critical
value for the Gauss map.
If we rotate $\immmapmin$ taking this direction to become vertical,
we have a nondegenerate critical point for the height function which is a
contradiction in view of the previous steps.
\end{proof}

\subsection{The semicartesian parametrization}\label{sec:global}
We can apply Theorem \ref{theo:trasv}
 to $\extimmmapmin$ with the family of planes of Theorem \ref{lem:piani},
obtaining a \textit{local} semicartesian parametrization. More precisely, for any point $p\in \extimmmapmin$
there exists an open domain $\dommap_p\subset\R^2_{(\ppa,\spa)}$  and an analytic\footnote{{}From
the proof of Theorem \ref{theo:trasv} one infers 
that the regularity of the local semicartesian map is the same as the surface.}, 
conformal semicartesian map 
$X_p$ parametrizing an open neighbourhood of $p$ on $\extimmmapmin$:
\begin{equation}\label{eq:locpar}
	\begin{split}
		X_p :\quad  D_p & \to \extimmmapmin,\\
			(t, \secondoparametroastratto_p) & \to 
			(t, X_{p2}(t,\secondoparametroastratto_p), X_{p3}(t,\secondoparametroastratto_p)).
	\end{split}
\end{equation}
\begin{Proposition}[\textbf{Global semicartesian parametrization}]\label{prop:globalpar}
In the hypotheses of Theorem \ref{prop:analytic}, $\immmapmin=X(\dommap)$ 
admits an analytic parametrization of the form \eqref{eq:good_par}.
\end{Proposition}

\begin{proof}

The local parametrization in \eqref{eq:locpar} is unique 
up to an additive constant: $s_p \mapsto \secondoparametroastratto_p + \rho$.
 Indeed,
if $\ppa_p$ is the $\ppa$-coordinate of $p$,
 the direction of $\partial_{\spa_p} X_p$ 
is given by the intersection of the tangent plane to $\extimmmapmin$
and the plane $\{t=t_p\}$,
since its $t$-component is zero.
The vector $\partial_t X_p$ is then uniquely determined
 by being in the tangent plane to
$\extimmmapmin$, orthogonal to $\partial_{\secondoparametroastratto_p} X_p$ and 
having $1$ as $t$-component. This in turn determines
the norm of $\partial_{\secondoparametroastratto_p} X_p$ and 
hence $\partial_{\secondoparametroastratto_p} X_p$ itself (up to a choice of the
orientation of $\extimmmapmin$)\footnote{Incidentally 
we note here that $|\partial_{\secondoparametroastratto_p} X_p| 
= |\partial_t X_p| \geq 1$ (which excludes branch points).}.
Functions $X_{p2}(t,\secondoparametroastratto_p)$ and 
$X_{p3}(t,\secondoparametroastratto_p)$ can now be 
obtained by integrating the vector
field $\partial_{\secondoparametroastratto_p} X_p$ 
along the curve $\{\primoparametroastratto = t_p\}\cap \extimmmapmin$ 
and transported as constant along the curves
$\{\spa = \text{const}\}$.
Now we can cover $\immmapmin \cup \Gamma$ with  a finite
 number of such neighbourhoods (local charts) having
connected pairwise intersection, and we can choose 
the constant in such a way that on the intersection of two
neighborhoods the different parametrizations coincide.
In this way we can ``transport'' the parametrization 
from a fixed chart along a chain of pairwise intersecting charts.
This definition is wellposed if we can prove that the 
transported parametrization is independent
of the actual chain, or equivalently that transporting 
the parametrization along a closed chain of
charts produces the original parametrization.
This is a consequence of the simple connectedness 
of the surface\footnote{This is one of the points where it is important
to consider disk-type area-minimizing surfaces.}, 
indeed we can take a closed curve
that traverses the original chain of charts and 
let it shrink until it is contained in a single
chart.

Thus we can construct a \textit{global} semicartesian parametrization $X$ 
defined on a open domain $\extdommap \subset \R^2$ as required in 
hypothesis (\hp 4). 
Eventually 
\begin{equation}\label{eq:domainD}
\dommap:= X^{-1}(\immmapmin\cup\Gamma)
\end{equation} 
is a closed bounded (connected and simply connected) set 
such that the intersection  with the line
 $\{t=k\}$, for $k \in (a,b)$,  is an interval (not reduced to a point); 
indeed if the intersection were composed 
by two (or more) connected components, there would be at least 
$4$ points on the intersection of $\Gamma$ with the plane 
$\{t=k\}$, and this is impossible 
since $\Gamma$ is union of two graphs on $t$.

\end{proof}

Before proving that 
the domain $D$ satisfies the local Lipschitz conditions required by
Definition \ref{def:semicart_par}, we need the following regularity result.

\begin{Lemma}\label{lemma:regularity_boundary}
The domain $\dommap$ defined in \eqref{eq:domainD} has analytic boundary.
\end{Lemma}
\begin{proof}
The boundary of $\dommap$
is the image of an analytic map defined on
$\partial \disk$.
This latter fact follows directly from 
the analiticity of the map $Y : \disk \to \immmapmin$ (see \eqref{eq:mappaY}) 
and
of the map $X : \dommap \to \immmapmin$.
The fact that $\immmapmin$ can have self-intersections 
is not a problem here because the preimages of
points (in either $\disk$ or $\dommap$) 
in a self-intersection 
are well separated, so that we can restrict to
small patches of the surface and reason locally.
\end{proof}

We are now in a position to 
specify a further property of $\partial \dommap$\footnote{
The analiticity of $\partial \dommap$ 
in particular implies that we cannot have a global
Lipschitz constant for $\sigma^\pm$, so that the result in Proposition \ref{prop:domain_glob_par} is optimal.}.

\begin{Proposition}\label{prop:domain_glob_par}

In the hypotheses and with the notation of Proposition \ref{prop:globalpar},
$\dommap$ has the form in \eqref{eq:formadom}, where
the two functions $\sigma^\pm : [a,b] \to \R$ satisfy \eqref{eq:sigmapm} and
condition (ii) of Theorem \ref{prop:analytic}.
\end{Proposition}

\begin{proof}
Since, as noticed in 
Lemma 
\ref{lemma:regularity_boundary}, $D\cap\{t=k\}$ is a interval, not reduced to a point, 
for any $k\in (a,b)$, $\dommap$ is in the form \eqref{eq:formadom} with $\sigma^-<\sigma^+$ in $(a,b)$;
up to traslation we can suppose also $\sigma^+(a)=0=\sigma^-(a)$. 

Let $(t,s) \in \partial D$ and let
$p = X(t,s) \in \Gamma$. Let us suppose that $s=\sigma^-(t)$ (the case $s=\sigma^+(t)$ being similar) and 
let us write $\sigma$ in place of $\sigma^-$ for simplicity. 
We have to show that
\begin{equation}\label{eq:sigmaprimo}
|\sigma'(t)|\leq|\gamma'(t)|.
\end{equation}
Let
$\vartheta(t,s) 
\in [-\pi/2,\pi/2]$ be the angle between 
the tangent line 
to $\Gamma$ at $p$ (spanned by $\frac{\Gamma'(t)}{\vert
\Gamma'(t)\vert}$) and the direction of 
$X_t(t,s)$.
Note that if  $\vartheta(t,s) 
\in (-\pi/2,\pi/2)$ we have
\begin{equation}\label{eq:vartheta}
{\rm tg}(\vartheta(t,s))= \sigma'(t).
\end{equation}
Indeed, take a vector 
$\ell$ generating the tangent line to $\partial D$ at $(t,s)$,
for instance
$\ell = (\sigma'(t),1)$.
Using also the conformality of $X$, the derivative $X_\ell$ 
of $X$ along the direction of $\ell$ is given by 
$X_\ell(t,s) = \sigma'(t) X_s(t,s) + X_t(t,s)$, and 
is a vector generating the tangent line to $\Gamma$ at $p$,
 and \eqref{eq:vartheta} follows.

Let now 
$\Theta(t,s)\in [0,\pi/2]$ be the angle between 
the tangent line to $\Gamma$ at $p$ and the line generated by $e_t=(1,0,0)$.
If $\Theta(t,s)\in [0,\pi/2)$ 
we have, writing $\gamma$ in place of $\gamma^-$,
$$
{\rm tg}(\Theta(t,s))= \vert \gamma'(t)\vert.
$$
Hence, to show \eqref{eq:sigmaprimo}, 
it is sufficient to show that 
$\vartheta(t,s) \leq \Theta(t,s)$, or
equivalently
\begin{equation}\label{eq:thetaTheta}
\frac{\pi}{2} - \vartheta(t,s) 
\geq \frac{\pi}{2} - \Theta(t,s).
\end{equation}
Consider 
$\frac{\Gamma'(t)}{\vert \Gamma'(t)\vert}$ 
as a point on 
$\mathbb S^2 \subset \R^3$ and think of $e_t$ as
the vertical direction (Figure \ref{fig:vett_b}). 
We have that 
$\frac{\pi}{2} - \Theta(t,s)$ is the latitude of 
$\frac{\Gamma'(t)}{\vert \Gamma'(t)\vert}$. On the other
hand, remembering that $X_s(t,s)$ 
is orthogonal to $e_t$,
 we have that $\frac{\pi}{2} - \vartheta(t,s)$ 
(the angle between $\frac{\Gamma'(t)}{\vert \Gamma'(t)\vert}$
and $X_s(t,s)$ by conformality)
is the geodesic distance (on $\mathbb S^2$)
between 
$\frac{\Gamma'(t)}{\vert \Gamma'(t)\vert}$ and the point obtained as 
the intersection between $T_p(\Sigma_{\rm min})$ and the equatorial plane.
Hence inequality \eqref{eq:thetaTheta} holds true.
\end{proof}

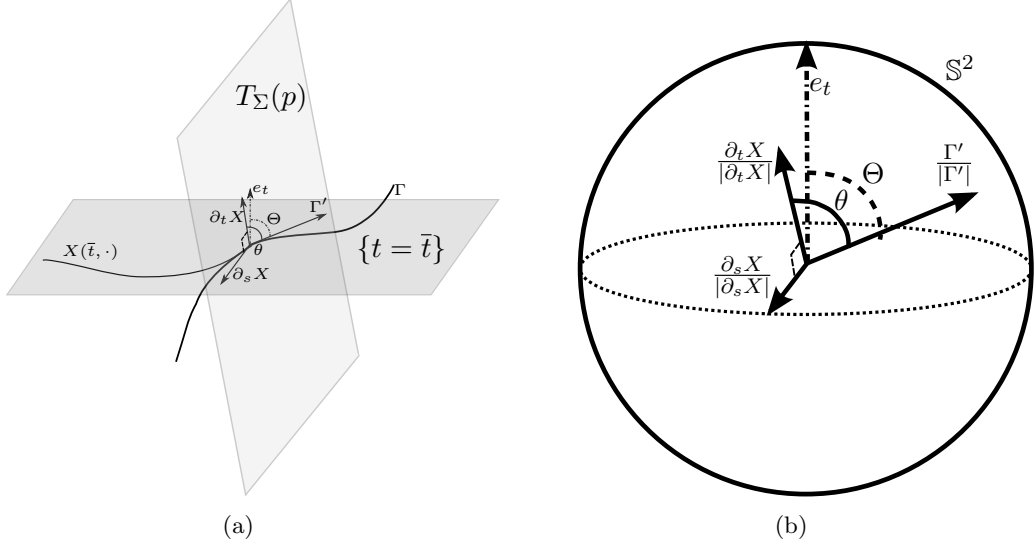
\begin{figure}[htbp]
\centering%
\subfigure [\label{fig:vett_a}]%
{
\def\svgwidth{6.5cm}
%% Creator: Inkscape inkscape 0.48.4, www.inkscape.org
%% PDF/EPS/PS + LaTeX output extension by Johan Engelen, 2010
%% Accompanies image file 'vettori.pdf' (pdf, eps, ps)
%%
%% To include the image in your LaTeX document, write
%%   \input{<filename>.pdf_tex}
%%  instead of
%%   \includegraphics{<filename>.pdf}
%% To scale the image, write
%%   \def\svgwidth{<desired width>}
%%   \input{<filename>.pdf_tex}
%%  instead of
%%   \includegraphics[width=<desired width>]{<filename>.pdf}
%%
%% Images with a different path to the parent latex file can
%% be accessed with the `import' package (which may need to be
%% installed) using
%%   \usepackage{import}
%% in the preamble, and then including the image with
%%   \import{<path to file>}{<filename>.pdf_tex}
%% Alternatively, one can specify
%%   \graphicspath{{<path to file>/}}
%% 
%% For more information, please see info/svg-inkscape on CTAN:
%%   http://tug.ctan.org/tex-archive/info/svg-inkscape
%%
\begingroup%
  \makeatletter%
  \providecommand\color[2][]{%
    \errmessage{(Inkscape) Color is used for the text in Inkscape, but the package 'color.sty' is not loaded}%
    \renewcommand\color[2][]{}%
  }%
  \providecommand\transparent[1]{%
    \errmessage{(Inkscape) Transparency is used (non-zero) for the text in Inkscape, but the package 'transparent.sty' is not loaded}%
    \renewcommand\transparent[1]{}%
  }%
  \providecommand\rotatebox[2]{#2}%
  \ifx\svgwidth\undefined%
    \setlength{\unitlength}{330.85bp}%
    \ifx\svgscale\undefined%
      \relax%
    \else%
      \setlength{\unitlength}{\unitlength * \real{\svgscale}}%
    \fi%
  \else%
    \setlength{\unitlength}{\svgwidth}%
  \fi%
  \global\let\svgwidth\undefined%
  \global\let\svgscale\undefined%
  \makeatother%
  \begin{picture}(1,1.0182325)%
    \put(0,0){\includegraphics[width=\unitlength]{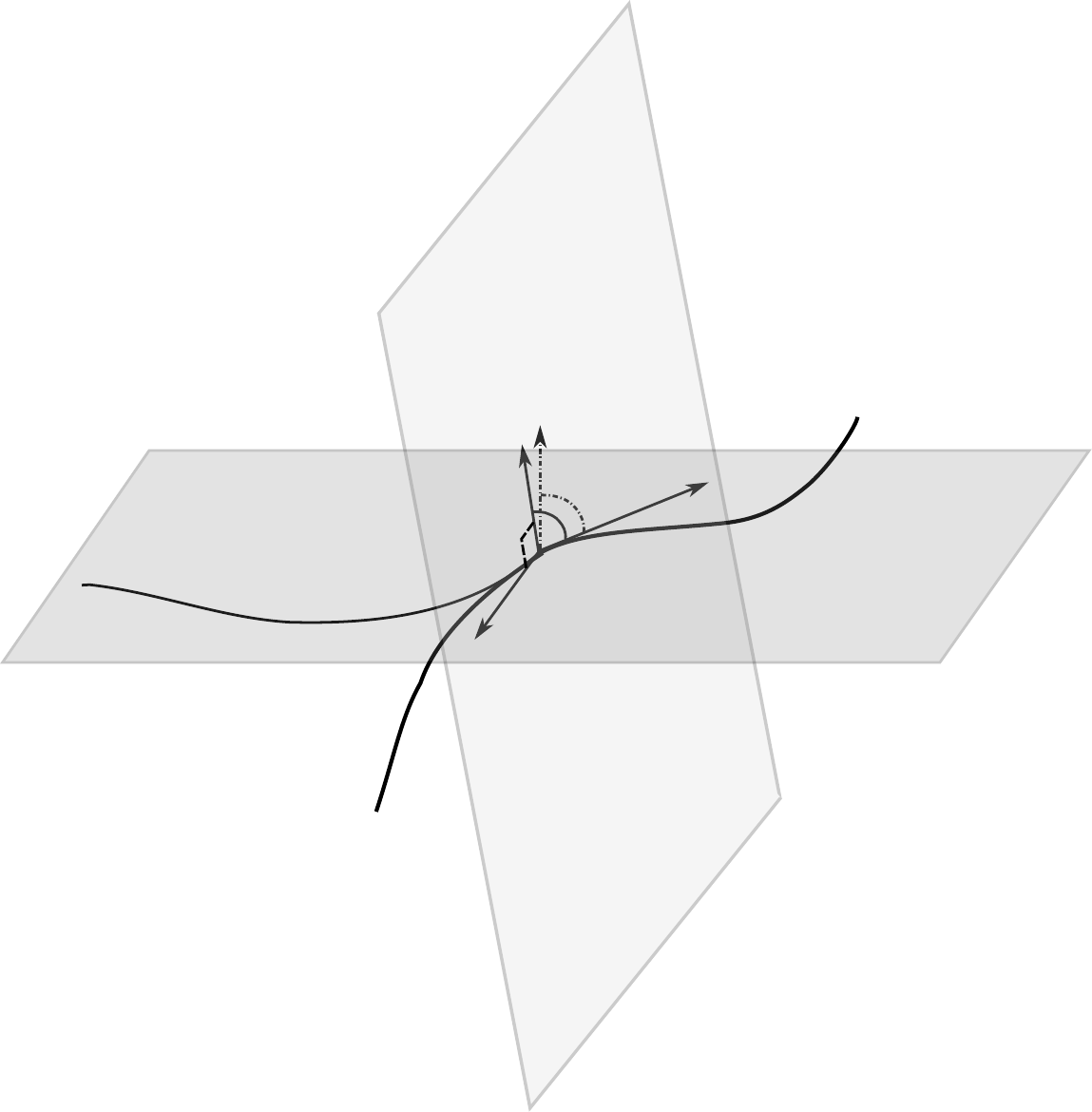}}%
    \put(0.78611288,0.61445758){\color[rgb]{0,0,0}\makebox(0,0)[lb]{\smash{\tiny{$\Gamma$}}}}%
    \put(0.6144791,0.58079431){\color[rgb]{0,0,0}\makebox(0,0)[lb]{\smash{\tiny{$\Gamma'$}}}}%
    \put(0.4984703,0.62162701){\color[rgb]{0,0,0}\makebox(0,0)[lb]{\smash{\tiny{$e_t$}}}}%
    \put(0.41184481,0.56043802){\color[rgb]{0,0,0}\makebox(0,0)[lb]{\smash{\tiny{$\partial_t X$}}}}%
    \put(0.45825514,0.44363303){\color[rgb]{0,0,0}\makebox(0,0)[lb]{\smash{\tiny{$\partial_s X$}}}}%
    \put(0.11472172,0.4908016){\color[rgb]{0,0,0}\makebox(0,0)[lb]{\smash{\tiny{$X(\overline{t},\cdot)$}}}}%
    \put(0.52771806,0.55622376){\color[rgb]{0,0,0}\makebox(0,0)[lb]{\smash{\tiny{$\Theta$}}}}%
    \put(0.50173322,0.4920878){\color[rgb]{0,0,0}\makebox(0,0)[lb]{\smash{\tiny{$\theta$}}}}%
    \put(0.46469947,0.79933907){\color[rgb]{0,0,0}\makebox(0,0)[lb]{\smash{$T_\Sigma(p)$}}}%
    \put(0.71057949,0.4892282){\color[rgb]{0,0,0}\makebox(0,0)[lb]{\smash{$\{t=\overline{t}\}$}}}%
  \end{picture}%
\endgroup%

}\qquad
\subfigure[ \label{fig:vett_b}]%
{
\def\svgwidth{6cm}
%% Creator: Inkscape inkscape 0.48.4, www.inkscape.org
%% PDF/EPS/PS + LaTeX output extension by Johan Engelen, 2010
%% Accompanies image file 'vettori_b.pdf' (pdf, eps, ps)
%%
%% To include the image in your LaTeX document, write
%%   \input{<filename>.pdf_tex}
%%  instead of
%%   \includegraphics{<filename>.pdf}
%% To scale the image, write
%%   \def\svgwidth{<desired width>}
%%   \input{<filename>.pdf_tex}
%%  instead of
%%   \includegraphics[width=<desired width>]{<filename>.pdf}
%%
%% Images with a different path to the parent latex file can
%% be accessed with the `import' package (which may need to be
%% installed) using
%%   \usepackage{import}
%% in the preamble, and then including the image with
%%   \import{<path to file>}{<filename>.pdf_tex}
%% Alternatively, one can specify
%%   \graphicspath{{<path to file>/}}
%% 
%% For more information, please see info/svg-inkscape on CTAN:
%%   http://tug.ctan.org/tex-archive/info/svg-inkscape
%%
\begingroup%
  \makeatletter%
  \providecommand\color[2][]{%
    \errmessage{(Inkscape) Color is used for the text in Inkscape, but the package 'color.sty' is not loaded}%
    \renewcommand\color[2][]{}%
  }%
  \providecommand\transparent[1]{%
    \errmessage{(Inkscape) Transparency is used (non-zero) for the text in Inkscape, but the package 'transparent.sty' is not loaded}%
    \renewcommand\transparent[1]{}%
  }%
  \providecommand\rotatebox[2]{#2}%
  \ifx\svgwidth\undefined%
    \setlength{\unitlength}{190.07842686bp}%
    \ifx\svgscale\undefined%
      \relax%
    \else%
      \setlength{\unitlength}{\unitlength * \real{\svgscale}}%
    \fi%
  \else%
    \setlength{\unitlength}{\svgwidth}%
  \fi%
  \global\let\svgwidth\undefined%
  \global\let\svgscale\undefined%
  \makeatother%
  \begin{picture}(1,1.01029955)%
    \put(0,0){\includegraphics[width=\unitlength]{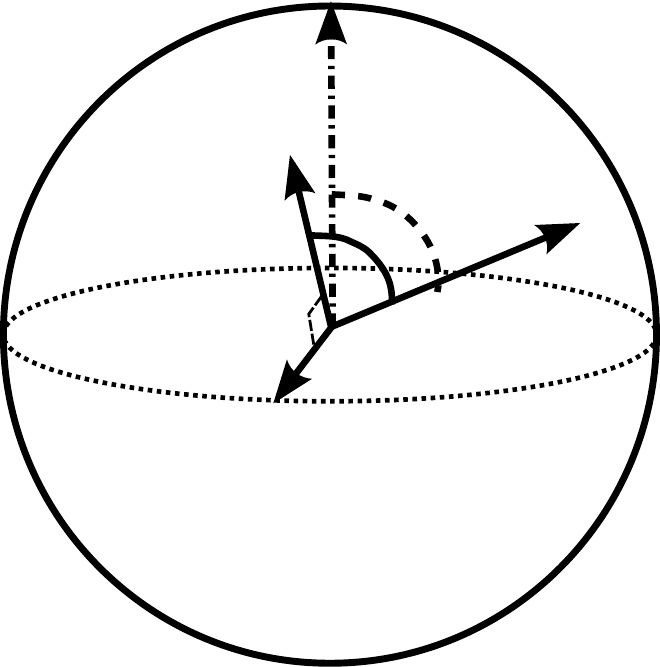}}%
    \put(0.62281313,0.69247301){\color[rgb]{0,0,0}\makebox(0,0)[lb]{\smash{$\Theta$}}}%
    \put(0.56208664,0.63362053){\color[rgb]{0,0,0}\makebox(0,0)[lb]{\smash{$\theta$}}}%
    \put(0.78088951,0.72020023){\color[rgb]{0,0,0}\makebox(0,0)[lb]{\smash{$\frac{\Gamma'}{|\Gamma'|}$}}}%
    \put(0.30105248,0.72419521){\color[rgb]{0,0,0}\makebox(0,0)[lb]{\smash{$\frac{\partial_t X}{|\partial_t X|}$}}}%
    \put(0.29313984,0.47277564){\color[rgb]{0,0,0}\makebox(0,0)[lb]{\smash{$\frac{\partial_s X}{|\partial_s X|}$}}}%
    \put(0.50767534,0.89678688){\color[rgb]{0,0,0}\makebox(0,0)[lb]{\smash{$e_t$}}}%
    \put(0.80724265,0.92126829){\color[rgb]{0,0,0}\makebox(0,0)[lb]{\smash{$\mathbb{S}^2$}}}%
  \end{picture}%
\endgroup%

}
\caption{\small{
(a): The dotted vector $e_t$ is  perpendicular to the plane 
$\{t=\overline{t}\}$ on which we have represented a part of the curve
$\{X(\overline{t},s): s\in[\sigma^-(t),\sigma^+(t)]\}$. 
 $\Gamma$ is also drawn, 
and passes through the plane $\{t=\overline{t}\}$ transversally. 
The other plane is the tangent plane to $\immmapmin$
at $p=X(\overline{t}, \sigma^-(\overline{t}))$ and the three vectors
are
the conformal basis of the tangent plane ${\rm span}\{\partial_t X, \partial_sX\}$ and the vector $\Gamma'(\overline{t})$. 
The angles $\theta$ and $\Theta$ are also displayed. (b): the same 
vectors normalized and represented on the sphere $\mathbb{S}^2$.}}    
\end{figure}

\subsection{Shape of the parameter domain}\label{sec:domshape}

In order to conclude the proof of Theorem \ref{prop:analytic}, we need to 
study the behaviour of $\partial \dommap$ near $(a,0)$ and $(b,0)$.

\begin{Proposition}
Assertion (i) of Theorem \ref{prop:analytic} holds.
\end{Proposition}
\begin{proof}
%With a suitable global translation in the $\secondoparametroastratto$
% direction we can also fix $\sigma^-(a) = \sigma^+(a) = 0$.
Let us consider the point $(a,0)$.
From the analiticity of $\partial \dommap$ (Lemma \ref{lemma:regularity_boundary}) and the fact 
that $(a,0)$ minimizes the $t$-component in $\partial \dommap$, we
can express it locally 
in a neighborhood
 of $(a,0)$ as the graph 
$(\tau(\secondoparametroastratto),\secondoparametroastratto)$ of a function
$\tau: (s^-, s^+) \to \R$ defined in a neighborhood  $(s^-,s^+)$
of the origin that
can be Taylor expanded as
\begin{equation*}\label{eq:shapetaylor}
\tau(\secondoparametroastratto) = a + \alpha_2 \secondoparametroastratto^2 
+ \alpha_3 \secondoparametroastratto^3 + \alpha_4 \secondoparametroastratto^4 + o(\secondoparametroastratto^4), \qquad
s \in (s^-, s^+),
\end{equation*}
with $\alpha_2 \geq 0$.

Assume by contradiction that 
\begin{equation*}\label{eq:alpha2}
\alpha_2 = 0.
\end{equation*} 
Since $\dommap$ is contained in the half-plane $\{t\geq a\}$
it follows that 
\begin{equation*}\label{eq:alphatre}
\alpha_3 = 0 \quad \mbox{ and } \quad \alpha_4 \geq 0.
\end{equation*}
We shall now compute the area $A(\eps)$ of 
$$\immmapmin^\eps:=\immmapmin \cap \{t < a + \eps\}=X(\dommap\cap S_\eps)$$ 
for small positive values of $\eps$, where 
$S_\eps:=\{(t,s):\,a\leq t<a+\eps\}$.
Using the conformal map $X$ we need 
to integrate the area element over the set
$\dommap\cap S_\eps$.
However the integrand is the modulus of the external 
product of the two derivatives of $X$ with
respect to $\ppa$ and to $\spa$, 
which is always greater than or equal to $1$, 
so that, integrating, we get
\begin{equation}\label{eq:shapeestimate}
A(\eps) \geq \mathcal{L}^2(\dommap\cap S_\eps) \geq c \eps^{1 + 1/4}
\end{equation}
for some positive constant $c$ independent of $\eps$.

We now want to show that the minimality 
of $\immmapmin$ entails that $\H^2(\immmapmin^\eps) \leq c \eps^{1 + 1/2}$, 
which is in contradiction with \eqref{eq:shapeestimate}.
Indeed we can compare the area of $\immmapmin$ with the competitor surface
$$
\Sigma:=\Sigma_1\cup\Sigma_2\cup\Sigma_3\cup \Sigma_4,
$$
where (see Figure \ref{fig:competitor}):
\begin{itemize}
\item[-]  $\Sigma_1$ is the parabolic sector delimited by the 
osculating parabola to $\Gamma$ in the minimum point and 
by the plane $\{t=a+\eps\}$;
\item[-] $\Sigma_2$ is the portion of the plane $\{t=a+\eps\}$ 
between the curve $\immmapmin\cap\{t=a+\eps\}$ and the 
boundary of $\Sigma_1$;
\item[-] $\Sigma_3$ is obtained connecting linearly 
each point of the osculating parabola
with the point of $\Gamma$ having the same $t$-coordinate;
\item[-] $\Sigma_4:=\immmapmin\cap\{a+\eps\leq t\leq b\}$.
\end{itemize}

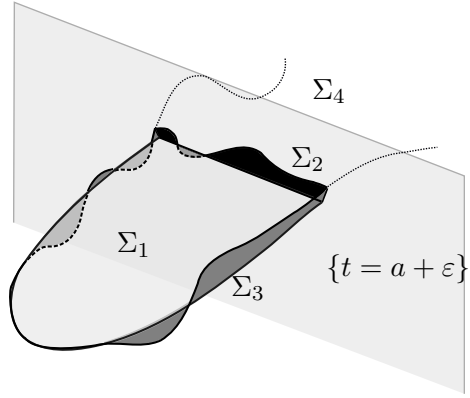
\begin{figure}
\centering
\def\svgwidth{6cm}
%% Creator: Inkscape inkscape 0.48.4, www.inkscape.org
%% PDF/EPS/PS + LaTeX output extension by Johan Engelen, 2010
%% Accompanies image file 'competitor.pdf' (pdf, eps, ps)
%%
%% To include the image in your LaTeX document, write
%%   \input{<filename>.pdf_tex}
%%  instead of
%%   \includegraphics{<filename>.pdf}
%% To scale the image, write
%%   \def\svgwidth{<desired width>}
%%   \input{<filename>.pdf_tex}
%%  instead of
%%   \includegraphics[width=<desired width>]{<filename>.pdf}
%%
%% Images with a different path to the parent latex file can
%% be accessed with the `import' package (which may need to be
%% installed) using
%%   \usepackage{import}
%% in the preamble, and then including the image with
%%   \import{<path to file>}{<filename>.pdf_tex}
%% Alternatively, one can specify
%%   \graphicspath{{<path to file>/}}
%% 
%% For more information, please see info/svg-inkscape on CTAN:
%%   http://tug.ctan.org/tex-archive/info/svg-inkscape
%%
\begingroup%
  \makeatletter%
  \providecommand\color[2][]{%
    \errmessage{(Inkscape) Color is used for the text in Inkscape, but the package 'color.sty' is not loaded}%
    \renewcommand\color[2][]{}%
  }%
  \providecommand\transparent[1]{%
    \errmessage{(Inkscape) Transparency is used (non-zero) for the text in Inkscape, but the package 'transparent.sty' is not loaded}%
    \renewcommand\transparent[1]{}%
  }%
  \providecommand\rotatebox[2]{#2}%
  \ifx\svgwidth\undefined%
    \setlength{\unitlength}{265.8352536bp}%
    \ifx\svgscale\undefined%
      \relax%
    \else%
      \setlength{\unitlength}{\unitlength * \real{\svgscale}}%
    \fi%
  \else%
    \setlength{\unitlength}{\svgwidth}%
  \fi%
  \global\let\svgwidth\undefined%
  \global\let\svgscale\undefined%
  \makeatother%
  \begin{picture}(1,0.86727757)%
    \put(0,0){\includegraphics[width=\unitlength]{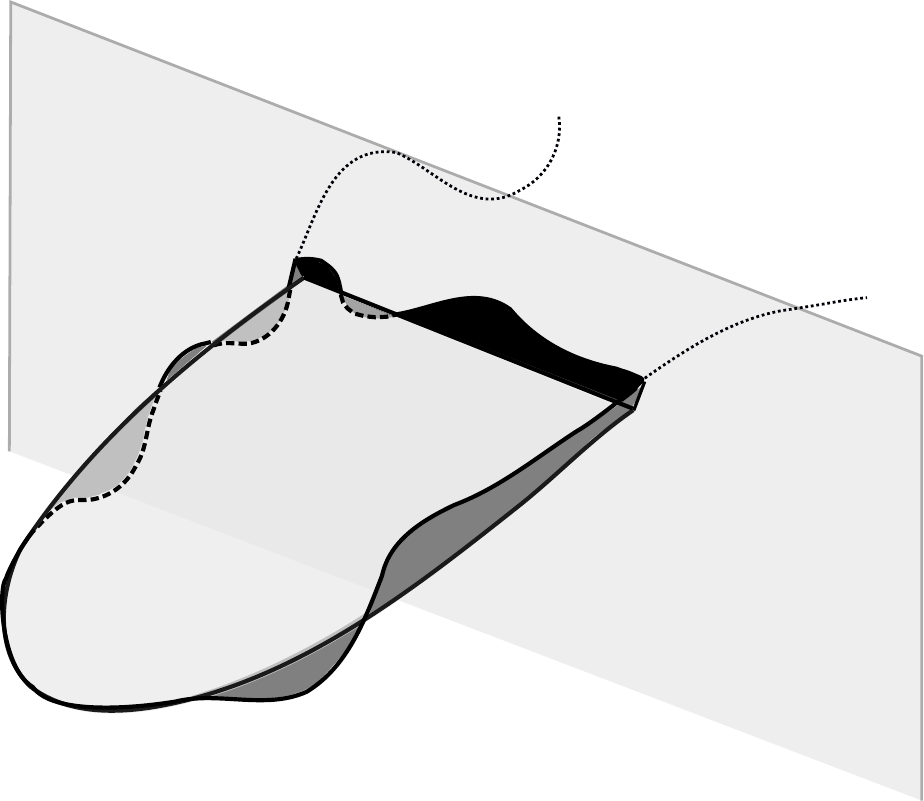}}%
    \put(0.69609895,0.26005424){\color[rgb]{0,0,0}\makebox(0,0)[lb]{\smash{$\{t=a+\eps\}$}}}%
    \put(0.23621444,0.31287418){\color[rgb]{0,0,0}\makebox(0,0)[lb]{\smash{$\Sigma_1$}}}%
    \put(0.61796567,0.5010257){\color[rgb]{0,0,0}\makebox(0,0)[lb]{\smash{$\Sigma_2$}}}%
    \put(0.48705749,0.21663898){\color[rgb]{0,0,0}\makebox(0,0)[lb]{\smash{$\Sigma_3$}}}%
    \put(0.66461104,0.65149486){\color[rgb]{0,0,0}\makebox(0,0)[lb]{\smash{$\Sigma_4$}}}%
  \end{picture}%
\endgroup%

\caption{\small{The competitor surface $\Sigma$.
$\Sigma_1$, $\Sigma_2$ and $\Sigma_3$ 
are the light gray, black and dark gray surface respectively.}}
\label{fig:competitor}
\end{figure}

Notice that $\Sigma$ is a Lipschitz surface and $\partial \Sigma=\Gamma$. Moreover 
$\immmapmin = \immmapmin^\eps \cup \Sigma_4$ with $\immmapmin^\eps \cap \Sigma_4=\emptyset$.
Thus, using also the minimality of $\immmapmin$, we get
$$
\H^2(\immmapmin)=A(\eps) + \H^2(\Sigma_4)\leq\H^2(\Sigma)\leq \sum_{i=1}^4 \H^2(\Sigma_i) ,
$$
which implies $A(\eps)\leq \H^2(\Sigma_1)+\H^2(\Sigma_2)+\H^2(\Sigma_3)$.
Now, we notice that, for a constant $c$ independent of $\eps$:
\begin{itemize}
\item[-]$\H^2(\Sigma_1) \leq c \eps^{1+1/2}$, since it is a parabolic sector, 
\item[-]$\H^2(\Sigma_2) \leq c\eps^{1+1/2}$ because $\immmapmin$ is bounded
by the two planes of the wedge,
\item[-]$\H^2(\Sigma_3)=o(\eps^{1+1/2})$
because $\immmapmin^\eps$ is contained in
the inside of a cylindrical shape obtained by translation 
of $\Gamma$ in the direction orthogonal
to both the tangent vector to $\Gamma$ 
in its minimum point and the vector $(1,0,0)$.
\end{itemize}
Thus we get the contradicting relation:
$$
c_1\eps^{1+1/2} \geq A(\eps) \geq c_2\eps^{1+1/4},
$$
where $c_1$ and $c_2$ are two positive constants independent of $\eps$.
\end{proof}

%
%\section{An example with an analytic $\Gamma$}\label{sec:ex}
%\input{example}
%
\section{Appendix 1: some useful results on the Plateau's problem}\label{sec:app}
In this appendix we briefly collect all definitions
and results on the Plateau's problem,
with the related references, needed in the proofs
of Theorems \ref{teo:graph_main}, \ref{teo:general_main}, \ref{prop:analytic} and \ref{lem:piani}.

\subsection{Parametric approach}

Let $B\subset \R^2_{(\parahilduno,\parahilddue)}$ be the unit open disk
and $\Gamma$ be an oriented\footnote{The orientation is provided by fixing a homeomorphism from $\partial \disk$ onto $\Gamma$.}
rectifiable closed simple curve in $\R^3$.
We are interested in minimizing the area functional
$$
 \int_B |\paramap_\parahilduno \wedge 
\paramap_\parahilddue| ~d\parahilduno
\, d\parahilddue
$$
in the class\footnote{Since $\Gamma$ is rectifiable, we have
$\mathcal{C}(\Gamma)\neq\emptyset$.}
$$
\mathcal{C}(\Gamma)=\left\{\paramap \in H^{1,2}(B;\R^3)\cap \mathcal C(\partial \disk; \R^3), 
~ \paramap_{|\partial \disk}(\partial \disk) = \Gamma, \,Y_{|\partial \disk}\mbox{ weakly monotonic}
\right\}.
$$
The set $Y(B)$ for $Y\in \mathcal{C}(\Gamma)$ is called a \textit{disk-type surface} spanning $\Gamma$.
\begin{Definition}[\textbf{Disk-type area-minimizing solution}]
\textup{
We refer to a solution of the minimum problem 
\begin{equation}\label{eq:plateau}
\inf_{Y\in\mathcal{C}(\Gamma)}\int_\disk|\paramap_\parahilduno \wedge 
\paramap_\parahilddue| ~d\parahilduno
\, d\parahilddue
\end{equation}
as \textit{disk-type area-minimizing solution
of Plateau's problem for the contour $\Gamma$}.
Its image in $Y(B)\subset \R^3$ is called \textit{area-minimizing surface} spanning $\Gamma$,
but sometimes, with a small abuse of language, also \textit{area-minimizing solution}, identifying the 
image and the parametrization. We usually denote such a $Y(B)$ by $\immmapmin$.}
\end{Definition}
For further details about the formulation of Plateau's problem
we refer to  \cite[chapter 4, p. 270]{DiHiSa:10}.

\smallskip

Concerning the existence of a solution of \eqref{eq:plateau} the following 
holds.

\begin{Theorem}[\textbf{Existence of minimizers and interior regularity}]\label{teo:existence}
Problem \eqref{eq:plateau}
admits a solution $\paramap \in \mathcal C^2(B) \cap 
\mathcal C(\overline B)$, such that 
\begin{equation}\label{eq:lapzero}
\Delta \paramap =0 \qquad {\rm in}~ B
\end{equation}
and the conformality relations hold:
	\begin{equation}\label{eq:conformalrelations}
		|\paramap_\parahilduno|^2
=|\paramap_\parahilddue|^2\qquad{\rm and}\qquad  
\paramap_\parahilduno\cdot\paramap_\parahilddue=0 \qquad {\rm in}~ B.
	\end{equation}
Moreover the restriction $Y_{|\partial\disk}$ is a continuous, strictly monotonic
map onto $\Gamma$.
\end{Theorem} 
\begin{proof}
See for instance \cite[Main Theorem 1, chapter 4, p. 270]{DiHiSa:10}.
\end{proof}

\begin{Remark}[\textbf{Three points condition}]\label{rk:3points}\rm
One can impose on 
 a minimizer 
$\paramap$ the so-called three points condition:
 this means that we can fix three points $\omega_1$, $\omega_2$ and 
 $\omega_3$ on $\partial \disk$ and three points $P_1$, $P_2$ and $P_3$ 
 on $\Gamma$ (in such a way that the orientation of $\Gamma$ is respected) and find a 
solution $\paramap$  of \eqref{eq:plateau}
such that 
$\paramap(\omega_j)=P_j$ for any $j=1,2,3$.
\end{Remark}

\begin{Definition}[\textbf{Minimal surface}]
\textup{
A map $\paramap \in \C^2(B)\cap \C(\overline{B})$ satisfying
\eqref{eq:lapzero} and \eqref{eq:conformalrelations} mapping $\partial B$ onto $\Gamma$
in a weakly monotonic way is called a \textit{minimal surface spanning $\Gamma$}.
}
\end{Definition}

\smallskip

Concerning the regularity of a map $Y:\disk\to\R^3$
parametrizing a minimal surface, we cannot a priori 
avoid \textit{branch points}.

\begin{Definition}[\textbf{Branch point}]\label{def:intbranch}
\textup{
A point $\omega_0 \in B$ is called an \textit{interior branch point} for 
$\paramap \in \C^2(B)\cap\C(\overline{B})$ if
\begin{equation}\label{eq:branch}
	|\paramap_\parahilduno (\omega_0) \wedge \paramap_\parahilddue(\omega_0)|=0.
\end{equation}
If $\paramap$ is differentiable on $\partial \disk$, and $\omega_0\in 
\partial \disk$ 
is such that \eqref{eq:branch} holds, then $\omega_0$ is called
a \textit{boundary branch point}.
}
\end{Definition}
Observe that 
if $\omega_0$ is a branch point and \eqref{eq:conformalrelations} holds, then 
$\paramap_\parahilduno(\omega_0) = \paramap_\parahilddue(\omega_0)=0$.

It is known that interior branch points for a solution
 of \eqref{eq:plateau}
can be excluded.

\begin{Theorem}[\textbf{Absence of interior branch points}]\label{teo:nointbrapo}
Let $\paramap$ be as in Theorem \ref{teo:existence}.
Then $\paramap$ has no interior branch points. 
\end{Theorem}
\begin{proof}
See \cite[Main Theorem]{Os:70}.
\end{proof}

Under the stronger assumption that $\Gamma$ is analytic the classical Lewy's regularity theorem \cite{Lewy:51}
guarantees that the solution  of \eqref{eq:plateau} is analytic on $\overline{B}$ .

\begin{Theorem}[\textbf{Absence of boundary branch points}]\label{teo:noboundarybp}
Let $\Gamma$ be analytic and $\paramap$ be a solution
of \eqref{eq:plateau}. 
Then  $\paramap$ is analytic up to $\Gamma$ and 
has  no  boundary branch points. 
\end{Theorem}
\begin{proof}
See \cite{Gu_Le:73}.
\end{proof}

\begin{Theorem}[\textbf{Analytic extension}]\label{teo:ext}
Let $\Gamma$ be analytic and $\paramap$ be a minimal surface spanning $\Gamma$. 
Then $\paramap$ can be extended as a minimal surface 
across $\Gamma$, that is there exist an open set $\extdisk\supset\overline{B}$ and 
an analytic map $Y^{{\rm ext}}:\extdisk \to \R^3$ such that 
$Y^{{\rm ext}}=\paramap$ in $\overline{B}$ and
$Y^{{\rm ext}}$ satisfies \eqref{eq:lapzero} and \eqref{eq:conformalrelations}
in $\extdisk$.
\end{Theorem}

\begin{proof}
From \cite[Theorem 1, chapter 2.3]{DiHiTr:10a} one can extend 
a minimal surface across an analytic subarc of $\Gamma$. We apply 
this result twice to two overlapping
subarcs  covering $\Gamma$. Where the two extensions overlap, they have to coincide
due to analiticity.
\end{proof}

The following classical result can be found 
in \cite[p. 66]{DiHiSa:10}.
\begin{Theorem}[\textbf{Local semicartesian parametrization}]\label{theo:trasv}
If a minimal surface Y is intersected by a family of parallel planes
$\mathcal{P}$ none of which is tangent to the given surface and if each point of the
surface belongs to some plane $\Pi \in \mathcal{P}$, then the intersection lines of these
planes with the minimal surface form a family of
curves which locally belong to a net of conformal parameters on the surface.
\end{Theorem}

\subsection{Non-parametric approach}
Concerning the so-called non-parametric problem and the minimal
surface equation, we give the following definition and we refer to \cite{Gi:84}
for more.

\begin{Definition}[\textbf{Non-parametric solution}]
\textup{
Let $U\subset\R^2$ be a connected, bounded, open set and let $\phi\in\C(\partial U; \R^2)$.
A \textit{solution of the minimal surface equation for the boundary datum $\phi$}
is a solution $z\in \C^2(U)\cap\C(\overline{U})$ of
\begin{equation}\label{eq:min_surf_eq}
\begin{cases}
{\rm div}\left(\frac{\grad z}{\sqrt{1+|\grad z|^2}}\right)=0 &\mbox{ in }U\\
z=\phi &\mbox{ on }\partial U.
\end{cases}
\end{equation} 
}
\end{Definition}

The existence of a solution of \eqref{eq:min_surf_eq} is given by the following result.
\begin{Theorem}[\textbf{Existence of non-parametric solutions}]\label{teo:non_par_existence}
Let $U\subset\R^2$ be bounded and open and suppose that $\partial U$
is $\C^2$ and has non negative curvature. Then \eqref{eq:min_surf_eq} admits a solution.
\end{Theorem}
\begin{proof}
See \cite[Theorem 13.6]{Gi:84}.
\end{proof}

If $\Gamma$ can be 
described as the graph of a continuous function defined on 
the boundary of a bounded convex open set, then the following representation result holds.

\begin{Theorem}\label{teo:convex_proj}
If $\Gamma$ admits a one-to-one parallel projection
onto a plane Jordan curve bounding a convex domain $U$, then \eqref{eq:plateau}
has a unique solution $X$, up to conformal $\C^1$ diffeomorphisms of $B$.
Moreover $X(B)$ can be represented as the graph of a solution $z:U\to \R$ 
of \eqref{eq:min_surf_eq} with boundary datum a function $\phi$ whose graph is $\Gamma$.
\end{Theorem}
\begin{proof}
See \cite[Theorem 1, chapter 4.9]{DiHiSa:10}.
\end{proof}
We conclude this appendix with a regularity result for
a solution of \eqref{eq:min_surf_eq}.
\begin{Theorem}\label{teo:giusticonvesso}
Let $U \subset \R^2$ be bounded open convex set with 
$\partial U$ of class $\C^2$ 
and let $z$ be a solution of \eqref{eq:min_surf_eq} with 
boundary datum $\phi\in \C^{1,\lambda}(\partial U)$ for some $\lambda \in (0,1]$.
Then $z\in \C^{0,1}(\overline{U})$.
\end{Theorem}  
\begin{proof}
See \cite[Theorem 13.7]{Gi:84}.
\end{proof}

\section{Appendix 2: a result from Morse theory}\label{sec:appb}
In this short section 
we report a result from  \cite[Theorem 10]{Mo_Vansc:34} on critical points
of Morse functions.
The result holds in any dimension, but we need and state it only for $n=2$.

Let $U$ be a bounded open subset of $\R^2$ 
and let 
$\openMorse$ 
be an 
open subset of $U$ of class $\mathcal C^3$
 with  $\overline \openMorse\subset U$. Suppose that 
\begin{itemize}
	\item[-] $f: U \to \R$ is a Morse function;
        \item[-] $\openMorse$ contains all critical points of $f$;
	\item[-] all critical points of 
the restriction 
$f_{|\partial \openMorse}$
of $f$ to $\partial \openMorse$ are non degenerate (i.e.,
$f_{|\partial \openMorse}$ is a Morse function).
\end{itemize}  

Define
\begin{equation}\label{eq:bordomeno}
\partial^-_f \openMorse:=\{b \in \partial 
\openMorse:\,\ \grad f(b) \cdot \nu_\openMorse(b) <0\},
\end{equation}
where $\nu_\openMorse(b)$ 
denotes the outward unit normal to $\partial \openMorse$ at $b
\in \partial \openMorse$. 

{}For $i=0,\,1,\,2$, 
denote by
 $m_i(f,\openMorse)$ 
the number of critical points of 
index $i$ of $f$ in $\openMorse$ and by 
$m_i(f_{\vert \partial_f^- \openMorse})$ 
the number of critical points of index $i$ of 
$f_{\vert \partial \openMorse}$ 
on $\partial^-_f \openMorse$, with $m_2(f_{\vert \partial_f^- \openMorse}):=0$. 
Define
\begin{equation}\label{def:Mi}
M_i(f,\openMorse\cup\partial \openMorse) := m_i(f, \openMorse) + 
m_i(f_{\vert \partial^-_f \openMorse}), \qquad i=0,\,1\,,2.
\end{equation}

The following result holds.

\begin{Theorem}\label{teo:mors}
We have
\begin{equation*}\label{eq:morse}
	 M_0(f,\openMorse\cup\partial \openMorse) - M_1(f,\openMorse\cup\partial \openMorse)+M_2(f,\openMorse\cup\partial \openMorse) = \chi (\openMorse),
\end{equation*} 
where $\chi(\openMorse)$ is the Euler characteristic of $\openMorse$.
\end{Theorem}

\section{Appendix 3: the space $\DM$}\label{sec:appdomain}
In this section we discuss a property of the space $\DM$ 
introduced at the beginning of Section \ref{sec:notation}.

In \cite{AcDa:94} the following result is proven.
\begin{Theorem}\label{teo:a_dm}
Let $\mappav\in \BVo$. 
The following conditions are equivalent: 
\begin{itemize}
\item[-]  $\rel(\mappav, \Om)=\displaystyle 
\int_\Om |\M\left(\grad \mappav(x)\right)|\,dx\, dy<+\infty$;
\item[-] $\mappav\in \Wuu$, $\M(\grad \mappav) 
\in L^1(\Om; \R^6)$ and there exists a 
sequence $(\mappav^\mu) \subset\C^1(\Om; \R^2)$ 
converging to $\mappav$ in $L^1(\Om; \R^2 )$ such that 
the sequence $(\M (\grad \mappav^\mu))$ 
converges to $\M(\grad \mappav)$ in $L^1(\Om; \R^6)$.
\end{itemize}
\end{Theorem}

Hence $\DM$ is the subset of ${\rm BV}(\Om;\R^2)$
satisfying one of the two equivalent conditions of Theorem \ref{teo:a_dm}.
The following lemma shows that  $\rel$
 can be obtained also by relaxing $\A$ in $\DM$.

\begin{Lemma}\label{lem:relaxation}
Let $\mappa\in\BVo$. Then
\begin{equation}\label{eq:relax}
\rel (\mappa,\Om)= \inf \left\{ \liminf_{\eps \to 0^+} \rel (\ue,\Om),\,
\,(\ue)_\eps\subset \DM,\,\,\ue \to \mappa \mbox{ \textup{in} } L^1(\Om; \R^2)\right\}.
\end{equation}
\end{Lemma}

\begin{proof}
Trivially $\rel(u,\Om)$ is larger than or equal to the right hand side of \eqref{eq:relax},  
since $\C^1(\Om; \R^2)\subset \DM$ and $\rel=\A$ on $\C^1(\Om;\R^2)$.\\
In order to prove the opposite inequality, let $(\mappav_\eps)$ be a 
sequence in $\DM$ such that
$$
\lim_{\eps \to 0^+} \rel(\mappav_\eps, \Om)=\inf \left\{ \liminf_{\eps \to 0^+} 
\rel (\mappa_\eps,\Om),\,\,(\mappa_\eps)\subset \DM,\,\,\mappa_\eps \to \mappa \mbox{ \textup{in} } L^1(\Om; \R^2)\right\}.
$$
Thanks to Theorem \ref{teo:a_dm}, for each $\eps >0$ 
we can find a sequence $(\mappav_\eps^\mu)_\mu$ in $\Cuno$ converging to $\mappav_\eps$ in $L^1(\Om;\R^2)$
as $\mu\to 0^+$ such that 
$$
\A(\mappav_\eps^\mu,\Om)=\int_\Om |\M\left(\grad \mappav_\eps^\mu(x)\right)|\, dx 
\overset{\mu \to 0^+}{\longrightarrow}\int_\Om |\M\left(\grad \mappav_\eps(x)\right)|\, dx=\rel(\mappav_\eps,\Om).
$$
Thus by a diagonal process we obtain a sequence $(\mappav_\eps^{\mu(\eps)})\subset \Cuno$ 
converging to $\mappa$ in $L^1(\Om;\R^2)$ as $\eps \to 0^+$ such that the
right hand side of \eqref{eq:relax} equals
$$
\lim_{\eps\to 0^+} \A(\mappa^{\mu(\eps)}_\eps,\Om)=\lim_{\eps\to 0^+}\rel(\mappav_\eps,\Om),
$$
and this concludes the proof.

\end{proof}

\end{document}